\documentclass[10pt]{article}
\date{}
\makeatletter
\@addtoreset{equation}{section}
\makeatother
\usepackage[
hmarginratio={1:1},     
vmarginratio={1:1},     
textwidth=15cm,        
textheight=22cm,
heightrounded,          
]{geometry}
\usepackage[colorlinks,
linkcolor=blue,
anchorcolor=blue,
citecolor=blue]{hyperref}
\usepackage{graphicx}
\usepackage{hyperref}
\hypersetup{hypertex=true,
	colorlinks=true,
	linkcolor=blue,
	anchorcolor=blue,
	citecolor=blue}
\usepackage[nottoc,notlot,notlof]{tocbibind}
\usepackage{hyperref,amsmath,amssymb,amscd}
\usepackage{amssymb,amsmath,enumerate}
\usepackage{indentfirst}
\usepackage{url}
\usepackage[capitalise]{cleveref}
\usepackage{amsthm}
\usepackage{cite}
\newtheorem{theorem}{Theorem}[section]

\newtheorem{lemma}[theorem]{Lemma}
\newtheorem{proposition}[theorem]{Proposition}
\theoremstyle{definition}

\newtheorem{remark}[theorem]{\bf Remark}
\newtheorem{openproblem}[]{Open problem}
\allowdisplaybreaks[4]

\numberwithin{equation}{section}

\begin{document}
	\title{
		{Spectral Characterization, Local Rigidity and Morse Index of Bubble Solutions to the CR Fractional Yamabe Problem}}
\author{
	Jianlun Liu, Vicen\c{t}iu D. R\u{a}dulescu, Hong-Rui Sun and Ziheng Zhang
}
	\maketitle
	\baselineskip 17pt

	{\bf Abstract.} We study bubble solutions to the critical CR fractional Yamabe equation on the Heisenberg group for the full range \(0<\gamma<\frac{Q}{2}\), including the admissible integer-order cases. The CR fractional operator is genuinely
	nonlocal and is described by Gamma-function spectral multipliers. Its
	structure differs from that of the Euclidean fractional Laplacian, so the
	linearized analysis cannot be obtained by a direct extension of
	integer-order arguments. We develop a direct spectral approach based on the
	Cayley intertwining relation and the complex spherical-harmonic
	decomposition, which characterizes the kernel of the linearized operator in
	terms of infinitesimal translation and dilation modes without reducing the
	equation to an integral equation or requiring a separate decay analysis. As
	a byproduct, we obtain the exact normalization constant of the standard
	bubble throughout the admissible range. We then prove three main results:
	$(i)$ the non-degeneracy of the bubble manifold; $(ii)$ the local rigidity of
	the bubble manifold, via a uniform invertibility estimate on the orthogonal
	complement of the tangent space, with an explicit optimal constant; $(iii)$
	the Morse index of every positive bubble is exactly one. The results provide a
	unified spectral and variational framework for the CR Yamabe problem on the Heisenberg group, covering both integer and noninteger orders.

	

\vspace{-0.3cm}
\tableofcontents

	\section{Introduction and main results}
	
	The Yamabe problem is one of the fundamental problems in conformal geometry. Its counterpart in Cauchy-Riemann geometry asks whether a given strictly pseudoconvex CR manifold admits, within a prescribed
	conformal class of contact forms, a contact form with constant Webster scalar curvature. The CR Yamabe problem was initiated and developed by Jerison and Lee \cite{Jerison-Lee1987,Jerison-Lee1988}. In particular,
	their analysis of the standard CR sphere and its non-compact model, the Heisenberg group $\mathbb H^n$, led to the sharp Folland-Stein Sobolev inequality and to the complete description of its positive extremals. These extremals form a finite-dimensional family generated by the Heisenberg translations and anisotropic dilations, and are usually referred to bubble solutions.
	
	Fractional CR geometry provides a nonlocal extension of this theory. On the standard CR sphere $\mathbb S^{2n+1}$, Branson et al. \cite{Branson-Fontana-Morpurgo2013} studied the intertwining operators and gave an explicit description of their spectrum in terms of the complex spherical harmonic decomposition. On the Heisenberg group, Frank and Lieb \cite{Frank-Lieb2012} obtained the sharp Hardy-Littlewood-Sobolev inequality, together with the classification of its optimizers and the corresponding sharp Sobolev inequalities for conformally invariant fractional powers of the sub-Laplacian. Subsequently,
	Frank et al. \cite{Frank-Gonzalez-Monticelli-Tan2015} realized these fractional operators as scattering operators and established an associated extension problem and energy identity on the Siegel upper half-space. For further analytic properties of fractional powers of the sub-Laplacian on $\mathbb{H}^n$, we refer to \cite{Roncal-Thangavelu2016,Liu-Zhang2015} and the references therein. 
	
	While the sharp fractional Sobolev inequality and the complete classification of its positive extremals have provided a foundational understanding of the fractional critical CR Yamabe problem, deeper analytical properties of these solutions remain to be fully explored. A fundamental issue in the analysis of critical equations is the non-degeneracy of these extremals (or bubbles). Roughly speaking, non-degeneracy means that all solutions of the linearized equation are generated by the continuous symmetries of the original equation. This property is a basic ingredient in Lyapunov-Schmidt reductions, perturbation arguments, the construction of concentrating solutions, and the description of bubbling phenomena. Motivated by this, in this paper, we investigate spectral characterization, local rigidity and Morse index of bubble solutions to the following critical CR fractional Yamabe equation
	\begin{equation}\label{main-problem}
		\mathcal L_\gamma u=u^\frac{Q+2\gamma}{Q-2\gamma }\ \ \mbox{on}\ \ \mathbb H^n,
	\end{equation}
	where $n\geq 1$, $0<\gamma<\frac{Q}{2}$, $Q:=2n+2$ is the homogeneous dimension of $\mathbb H^n$, and $\mathcal L_\gamma$ is the CR fractional sub-Laplacian on $\mathbb H^n$ which defined by
	$$
	\mathcal L_\gamma :=(2|T|)^\gamma \frac{\Gamma\Big(\frac{1+\gamma}{2}+\frac{\mathcal L}{2|T|}\Big)}{\Gamma\Big(\frac{1-\gamma }{2}+\frac{\mathcal L}{2|T|}\Big)},
	$$
	with $T:=\frac{\partial}{\partial t}$ and $\mathcal L$ is the standard sub-Laplacian on $\mathbb H^n$. The operator used in the present paper is the normalized CR fractional sub-Laplacian 
	$$
	\mathcal L_\gamma:=2^{-\gamma}\mathcal P^{\theta}_\gamma
	$$
	where $\mathcal P^{\theta}_\gamma$ is the CR fractional sub-Laplacian in \cite{Frank-Gonzalez-Monticelli-Tan2015}. As we know, the sharp fractional Sobolev inequality implies that all positive extremals are obtained from a standard bubble via Heisenberg translations, dilations and multiplication by constants. After normalizing the coefficient on the right-hand side of (\ref{main-problem}) to one, we denote this family by
	\begin{equation}\label{eqn:mathcal-Z-define}
		\mathcal{Z} :=\Big \{\omega_{\lambda,\xi} : \lambda > 0,\  \xi \in \mathbb{H}^n\Big\},
	\end{equation}
	where $\lambda$ is the dilation parameter and $\xi$ is the translation center. The standard bubble is denoted by $\omega_{1,0}$.  In the Euclidean fractional setting, D\'avila et al. \cite{Davila-DelPino-Sire2013} proved the non-degeneracy of the extremals of the fractional Sobolev inequality. For the fractional critical CR Yamabe equation on the Heisenberg group, Chen and Wang \cite[Theorem 2.3]{Chen-Wang2017} established the corresponding non-degeneracy result. To be precise, they demonstrated that a function $u \in \mathcal{D}^\gamma(\mathbb{H}^n)$ satisfies the linearized equation
	\begin{equation}\label{eqn:linear-equation-f}
		\mathcal L_\gamma u = (Q^*-1) \omega_{1,0}^\frac{4\gamma}{Q-2\gamma}u
	\end{equation}
	if and only if there exist real coefficients $\mu \in \mathbb{R}$ and $\widetilde{\mu}_j \in \mathbb{R} $ ($j=1,\dots,2n+1$) such that $u$ can be represented as a linear combination of the following partial derivatives:
	$$
	u = \mu \frac{\partial \omega_{\lambda, \xi}}{\partial \lambda} \bigg|_{(\lambda, \xi) = (1, 0)} + \sum_{j=1}^{2n+1} \widetilde{\mu}_j\frac{\partial \omega_{\lambda, \xi}}{\partial \xi_j} \bigg|_{(\lambda, \xi) = (1, 0)},
	$$
	where $Q^*:=\frac{2Q}{Q-2\gamma}$.
	
	Other aspects of the fractional critical CR Yamabe equation (\ref{main-problem}) on the Heisenberg group have also received considerable attention. For instance, Krist\'aly \cite{Kristaly2020} studied the sign-changing counterpart of 
	$$
	\mathcal L_\gamma u=|u|^{\frac{4\gamma}{Q-2\gamma}}u\ \ \text{on}\ \ \mathbb H^n,
	$$
	and proved that for
	$$
	\gamma\in\bigcup_{k=1}^{n}\left[k,\frac{kQ}{Q-1}\right),
	$$
	there exist several sequences of sign-changing solutions with distinct nodal properties. Guidi et al.  \cite{Guidi-Maalaoui-Martino2018} investigated the lack of compactness of the fractional CR Yamabe functional on the standard CR sphere. They established a global decomposition theorem for Palais–Smale sequences, showing that after passing to a subsequence, every such sequence decomposes into a weak solution on the sphere together with finitely many concentrating profiles obtained from entire solutions through translations, dilations, and the Cayley transform of the corresponding critical equation on the Heisenberg group. Chtioui et al. \cite{Chtioui-Hajaiej-Yacoub2021} studied the prescribed fractional CR $Q$-curvature equation on the standard sphere for $0<\gamma<\frac{Q}{2}$. By using Bahri's theory of critical points at infinity, they constructed suitable pseudo-gradient vector fields and gave a quantitative description of the noncompact flow lines of the associated variational functional. In particular, they characterized the possible blow-up configurations, identified the corresponding concentration points and single- or multi-bubble profiles, computed their Morse indices, and derived several existence results for prescribed fractional CR $Q$-curvatures. Although the problems considered by Guidi et al. \cite{Guidi-Maalaoui-Martino2018} and Chtioui et al. \cite{Chtioui-Hajaiej-Yacoub2021} are formulated primarily on the CR sphere, the Heisenberg group plays an essential role in both analyses. In the former work, each concentrating component of a Palais–Smale sequence is described in terms of an entire solution of the critical equation on $\mathbb H^n$. In the latter, the Cayley transform converts the spherical prescribed-curvature equation into
	$$
	\mathcal L_\gamma u=(K\circ\tau)u^{\frac{Q+2\gamma}{Q-2\gamma}}
	\qquad\text{on }\mathbb H^n,
	$$
	and the spherical bubbles are obtained from the standard Heisenberg bubble by left translations and dilations. Regarding the symmetry and classification of nonnegative $C^2$ solutions of $(\ref{main-problem})$, Prajapat and Varghese \cite{Prajapat-Varghese2025} employed the moving plane method for the equivalent integral equation on the Heisenberg group, together with the Hardy–Littlewood–Sobolev inequality, to establish a complete classification without imposing any a priori cylindrical symmetry or additional decay condition at infinity.  More precisely, they proved that every nontrivial nonnegative $C^2$ solution of (\ref{main-problem}) is cylindrically symmetric and even with respect to the central variable. Furthermore, they also showed that every positive solution of (\ref{main-problem}) is obtained from the standard bubble $\omega_{1,0}$ by Heisenberg translations and dilations.
	
	Recently, Qiang et al. \cite{Qiang-Tang-Zhang2026} proved a non-degeneracy result for the critical CR Yamabe problem and applied it to the construction of sign-changing concentrating solutions in the local setting ($\gamma=1$). Inspired by their work, the purpose of the present paper is threefold. First, we give an explicit spectral realization of the kernel of the linearized operator at $\omega_{1,0}$. Although Chen and Wang \cite{Chen-Wang2017} previously established the non-degeneracy of the extremals, our argument aims to identify every kernel direction through the Cayley transform and the complex spherical-harmonic
	decomposition on $\mathbb S^{2n+1}$. Second, we derive local rigidity of the bubble manifold $\mathcal Z$ from non-degeneracy. Third, we determine the sign
	structure of the linearized quadratic form and compute the Morse index of the standard positive bubble $\omega_{1,0}$ and use its invariance under Heisenberg translations and dilations to compute the Morse index of $\omega_{\lambda,\xi}$. Our first main result is stated below.

	\begin{theorem}\label{main-result-1}
		Assume that $\gamma \in (0, \frac{Q}{2})$ and let $\omega_{1,0}$ be the standard positive bubble solution to (\ref{main-problem}). Then $\omega_{1,0}$ is non-degenerate. More precisely, for every solution $u\in\mathcal D^\gamma(\mathbb H^n)$ of (\ref{eqn:linear-equation-f}), there exist real coefficients $\mu_1, \mu_2, \alpha_j, \beta_j$ ($j=1, \dots, n$) such that 
		$$
		u = \mu_1\partial_\lambda \omega_{\lambda, \xi}\big|_{(\lambda,\xi)=(1,0)} + \mu_2\partial_t \omega_{\lambda, \xi} \big|_{(\lambda,\xi)=(1,0)} + \sum_{j=1}^n \left( \alpha_j X^R_j \omega_{\lambda,\xi} \big|_{(\lambda,\xi)=(1,0)} + \beta_j Y^R_j \omega_{\lambda,\xi} \big|_{(\lambda,\xi)=(1,0)} \right),
		$$ 
		where $X^R_j :=\partial_{x_j} -2y_j \partial_t$ and $Y^R_j :=\partial_{y_j} + 2x_j \partial_t$.
	\end{theorem}
	
	\begin{remark}
		The conformally invariant fractional sub-Laplacian $\mathcal L_\gamma$ is defined for the full range
		$$
		0<\gamma<\frac{Q}{2}=n+1.
		$$
		Among the references cited here, the explicit normalization constant needed in our analysis can be obtained from Roncal and Thangavelu \cite[Page 143]{Roncal-Thangavelu2016}, whose formula is valid for
		$$
		0<\gamma<\frac{n+1}{2}=\frac{Q}{4}.
		$$
		Their argument involves the functions $u_{-\gamma,\delta}$ in the ground state representation, for which the relevant $L^2$-integrability requires $\gamma<\frac{Q}{4}$. This is the source of the above restriction in that approach. 
		
		Since our subsequent analysis requires the exact normalization constant $C(n,\gamma)$ (defined in (\ref{eqn:C-kappa-define-relation})), we derive it directly from the spectral formula for the spherical operator $\mathcal A_\gamma$ and the Cayley intertwining relation in Lemma \ref{solution-sca-tran}. This removes the above restriction and yields
		$C(n,\gamma)=\left[2^{2\gamma}\left(\frac{\Gamma\left(\frac{Q+2\gamma}{4}\right)}{\Gamma\left(\frac{Q-2\gamma}{4}\right)}\right)^2\right]^{\frac{Q-2\gamma}{4\gamma}}$ for the full range $0<\gamma<\frac{Q}{2}$, which agrees with the normalization obtained from Roncal and Thangavelu \cite{Roncal-Thangavelu2016} in the range $0<\gamma<\frac{Q}{4}$. 
	\end{remark}
	
	By exploiting the invariance of the fractional critical CR Yamabe equation (\ref{main-problem}) under Heisenberg translations and dilations, the non-degeneracy result for the standard bubble $\omega_{1,0}$ in Theorem \ref{main-result-1} can be directly generalized to any rescaled and translated bubble solution $\omega_{\lambda, \xi}$ of (\ref{main-problem}) with the dilation parameter $\lambda > 0$ and the translation center $\xi=(a,b,s) \in \mathbb{H}^n$. This leads to the following result:
	\begin{theorem}\label{main-result-2}
		Suppose that $\gamma \in (0, \frac{Q}{2})$. Then, for every $\lambda>0$ and $\xi=(a,b,s) \in \mathbb{H}^n$,
		$$
		\ker\mathcal L_{\omega_{\lambda,\xi}}=\operatorname{span}\left\{\partial_\lambda\omega_{\lambda,\xi},\partial_s
		\omega_{\lambda,\xi},\partial_{a_j}\omega_{\lambda,\xi},\partial_{b_j}\omega_{\lambda,\xi}:\ j=1,\ldots,n\right\},
		$$
		where $\omega_{\lambda,\xi}$ is given in (\ref{eqn:w-lambda-xi-define}) and
		$
		\mathcal L_{\omega_{\lambda,\xi}}  := \mathcal{L}_\gamma - (Q^\ast - 1) \omega_{\lambda,\xi}^{\frac{4\gamma}{Q-2\gamma}} .
		$.
	\end{theorem}

	\begin{remark}
		$(i)$ Theorems \ref{main-result-1} and \ref{main-result-2} are closely related. For any $\xi=(a,b,s)\in\mathbb H^n$, we have 
		$$
		\partial_{a_j}\omega_{\lambda,\xi}\big|_{(\lambda,\xi)=(1,0)}
		=-X_j^R\omega_{1,0},
		\qquad
		\partial_{b_j}\omega_{\lambda,\xi}\big|_{(\lambda,\xi)=(1,0)}
		=-Y_j^R\omega_{1,0},
		$$
		and
		$$
		\partial_s\omega_{\lambda,\xi}\big|_{(\lambda,\xi)=(1,0)}
		=-\partial_t\omega_{1,0}.
		$$
		
		Thus, Theorem \ref{main-result-1} explicitly identifies the tangent directions of the bubble family at $\omega_{1,0}$ with the first nontrivial real spherical eigenspace through the Cayley transform, and consequently characterizes the kernel of the linearized operator at the standard bubble. Theorem \ref{main-result-2} then transports this kernel characterization from $\omega_{1,0}$ to every bubble $\omega_{\lambda,\xi}$ by means of Heisenberg translations and dilations.
		
		$(ii)$ The non-degeneracy of $\omega_{1,0}$ was also established by Chen and Wang \cite[Theorem 2.3]{Chen-Wang2017}. Although both arguments use the Cayley transform and the spherical harmonic decomposition, the technical approaches are different. Chen and Wang first used the fundamental solution of the fractional CR operator to rewrite the linearized equation as an integral equation on $\mathbb H^n$, and then transformed it into an eigenvalue problem for a spherical integral operator. Their proof also requires a decay iteration to show that the transformed function extends continuously across the missing point of the sphere. In contrast, our proof works directly with the intertwining identity between $\mathcal L_\gamma$ and the spherical operator $\mathcal A_\gamma$, and uses the explicit spectral decomposition of $\mathcal A_\gamma$ on the spaces $\mathcal H_{j,k}$. The explicit computations of the parameter derivatives, or equivalently the corresponding right-invariant vector-field derivatives, then identify the kernel with the first nontrivial real spherical eigenspace without passing through the integral formulation or requiring a separate decay argument.
	\end{remark}

	Based on the above non-degeneracy result, the local rigidity of the bubble manifold $\mathcal Z$ is studied below. The main difficulty is that $\mathcal Z$ is a non-compact manifold generated by translations and dilations, so the linearized operator is non-invertible on the whole space, since its kernel is exactly $T_\omega\mathcal Z$. Moreover, the nonlocal nature of $\mathcal L_\gamma$ makes it difficult to obtain an estimate uniform in the bubble parameters. These difficulties are addressed by establishing a uniform lower bound for the linearized operator on $(T_\omega\mathcal Z)^\perp$ through Fredholm theory and conformal invariance. After choosing the nearest bubble, the error is decomposed orthogonally to the tangent space, and the linear term is compared with the higher-order nonlinear remainder. This yields the following local rigidity result.

	\begin{theorem}\label{main-result-4}
		Let $\gamma \in (0, \frac{Q}{2})$. There exists $\epsilon > 0$ such that if $u \in\mathcal{D}^\gamma(\mathbb{H}^n)$ is a positive solution to (\ref{main-problem}) and satisfies
		$$
		\inf_{\omega \in \mathcal{Z}} \Vert{}u - \omega\Vert{}_{\mathcal{D}^\gamma(\mathbb H^n)} < \epsilon,
		$$
		then $u \in \mathcal{Z}$, where $\mathcal Z$ is given in (\ref{eqn:mathcal-Z-define}).
	\end{theorem}
	
	\begin{remark}
		Theorem \ref{main-result-4} gives a local rigidity (or uniqueness) result near the bubble manifold $\mathcal Z$. In contrast, Prajapat and Varghese \cite[Theorem 1.13]{Prajapat-Varghese2025} proved a global classification for the equivalent critical integral equation by combining the moving-plane method with CR type inversion. Their result shows that every positive solution can be obtained from the standard bubble by a Heisenberg translation and a dilation. Nevertheless, Theorem \ref{main-result-4} has a different purpose. It gives a local rigidity result directly from the non-degeneracy of the bubble manifold, without using the moving-plane method or a global classification argument. More precisely, the proof relies only on the kernel characterization of the linearized operator, a uniform coercive estimate on the orthogonal complement of the tangent space, and a local estimate for the nonlinear remainder. This local argument is therefore well suited to perturbative problems in which the exact symmetry or the global integral structure required by moving-plane arguments may no longer be available.
	\end{remark}
	Following the variational framework for the fractional operators established in \cite{Branson-Fontana-Morpurgo2013,Frank-Lieb2012,Frank-Gonzalez-Monticelli-Tan2015}, we define the linearized operator $\mathcal{L}_{\omega_{1,0}}$ around the standard positive bubble $\omega_{1,0}$ as
	\begin{equation}\label{eqn:mathcal-L-+-define}
		\mathcal{L}_{\omega_{1,0}} := \mathcal{L}_\gamma - (Q^\ast - 1)\omega_{1,0}^{\frac{4\gamma}{Q-2\gamma}}.
	\end{equation}
	Since the linearized operator $\mathcal{L}_{\omega_{1,0}}$ naturally maps $\mathcal{D}^{\gamma}(\mathbb{H}^{n})$ into its dual space $\mathcal{D}^{-\gamma}(\mathbb{H}^{n})$, its eigenvalue problem is not directly well defined without specifying a Hilbert space realization. Therefore, we adopt the variational framework from Frank et al. \cite[Definition 3.2]{Frank-Lenzmann-Silvestre2016} and let $m(\omega_{1,0})$ denote the Morse index of $\omega_{1,0}$. This index is defined as the maximal dimension of a subspace of $\mathcal{D}^{\gamma}(\mathbb{H}^{n})$ on which the associated quadratic form
	\begin{equation}\label{eqn:Q-w-10-define}
		\mathcal{Q}_{\omega_{1,0}}(f, f) := \int_{\mathbb{H}^n} \Big( f \mathcal{L}_\gamma f - (Q^\ast - 1)\omega_{1,0}^{\frac{4\gamma}{Q-2\gamma}} f^2 \Big)d \eta
	\end{equation}
	is strictly negative. Equivalently, this dimension corresponds to the number of strictly negative eigenvalues (counting multiplicities) of the self-adjoint operator $I - \mathcal{K}_{\omega_{1,0}}$ on $\mathcal{D}^{\gamma}(\mathbb{H}^{n})$ (which will be introduced in Section \ref{proof of main result}). Via the Cayley transform, it also equals the negative spectrum count of the shifted intertwining operator 
	$
	\mathcal P_\gamma:=\mathcal A_\gamma-\lambda_{1,0}(\gamma)I
	$ 
	on $\mathbb{S}^{2n+1}$.
	\begin{theorem}\label{main-result-3}
		Let $\gamma \in (0, \frac{Q}{2})$ and $\omega_{1,0}$ be the standard positive solution to (\ref{main-problem}). Then the Morse index of $\omega_{1,0}$ is exactly $1$, i.e., $m(\omega_{1,0}) = 1$. More precisely, the negative subspace of the quadratic form $\mathcal Q_{\omega_{1,0}}$ is one dimensional. Equivalently, the shifted spherical operator
		$\mathcal P_\gamma$ has exactly one negative eigenvalue, which is simple, and its eigenspace is the space $\mathcal H_{0,0}$ of constant functions, where $\mathcal A_\gamma$ is the general intertwining operator, $\lambda_{1,0}(\gamma)$ is the eigenvalue of $\mathcal{A}_\gamma$ on  $\mathcal{H}_{1,0}$  and $I$ is the identity operator.
	\end{theorem}
	
	Since every bubble $\omega_{\lambda,\xi}$ is obtained from $\omega_{1,0}$ by a Heisenberg translation and a dilation, and these transformations preserve the quadratic form and its negative directions, the same conclusion holds for $\omega_{\lambda,\xi}$. This gives the following result.
	
	\begin{theorem}\label{main-result-5}
		Suppose that $\gamma\in(0,\frac{Q}{2})$. Then, for every $\lambda>0$ and $\xi\in\mathbb H^n$, the Morse index of $\omega_{\lambda,\xi}$ is equal to that of the standard bubble $\omega_{1,0}$. That is,
		$
		m(\omega_{\lambda,\xi})=m(\omega_{1,0})=1.
		$
	\end{theorem}
	
	The preceding Morse index results identify the negative and zero spectral subspaces associated with the linearized quadratic form. The explicit spectrum also provides more precise information about the linearized operator. In particular, it allows us to determine the optimal constant in an estimate for the linearized operator on the orthogonal complement of its kernel. Together with the invariance under Heisenberg translations and dilations, this leads to a sharp estimate that holds uniformly along the bubble manifold. The following result strengthens the uniform invertibility estimate established in Lemma \ref{lem:fredholm-operator}.
	
	\begin{theorem}\label{main-result-6}
		Let \(0<\gamma<\frac{Q}{2}\). For every \(\omega\in\mathcal Z\) and \(v\in(T_\omega\mathcal Z)^\perp\), the linearized operator \(\mathcal L_\omega\) satisfies
		\[
		\|\mathcal L_\omega v\|_{\mathcal D^{-\gamma}(\mathbb H^n)}
		\geq c_{n,\gamma}\|v\|_{\mathcal D^\gamma(\mathbb H^n)},
		\]
		where
		\[
		c_{n,\gamma}:=\frac{4\gamma}{Q+2\gamma+4}.
		\]
		Moreover, this constant is optimal. If \(v\) also satisfies
		\[
		(v,\omega)_{\mathcal D^\gamma(\mathbb H^n)}=0,
		\]
		then
		\[
		\langle\mathcal L_\omega v,v\rangle\geq
		c_{n,\gamma}\|v\|_{\mathcal D^\gamma(\mathbb H^n)}^2.
		\]
		The constant in this inequality is also optimal.
	\end{theorem}
	
	The paper is organized as follows. Section \ref{sec:preliminary} gives some necessary preliminaries on the Heisenberg group. In Section \ref{proof of main result}, we establish Theorems \ref{main-result-1} and \ref{main-result-2} by using the Cayley transform, and also prove the local rigidity result in Theorem \ref{main-result-4}. Finally, Section \ref{proof of main result-3} computes the Morse index of $\omega_{\lambda,\xi}$, and proves Theorems \ref{main-result-3} and \ref{main-result-5}. The optimal estimates in Theorem \ref{main-result-6} are established at the end of that section.

	\section{Preliminaries}\label{sec:preliminary}
	The Heisenberg group $\mathbb{H}^n$ is the space $\mathbb{C}^n \times \mathbb{R}$ equipped with the group law
	\begin{equation}\label{cric-define}
		(z, t) (z', t') := (z + z', t + t' + 2 \operatorname{Im} (z \cdot \overline{z'})),
	\end{equation}
	where $z = (z_1, \ldots, z_n)$,  $z' = (z'_1, \ldots, z'_n)$ and 
	$$
	z \cdot \overline{w} := \sum_{j=1}^n z_j \overline{w}_j
	$$ 
	is the standard Hermitian inner product on $\mathbb{C}^n$. Here, $\operatorname{Im}$ denotes the imaginary part of a complex number. The identity element of $\mathbb{H}^n$ is $(0,0)$, and the inverse of $(z,t)$ is 
	$$
	(z,t)^{-1}=(-z,-t).
	$$
	Writing $z=x+iy$ with $x,y\in\mathbb R^n$, we identify $\mathbb C^n$ with $\mathbb R^n\times \mathbb R^n$ and use $\eta=(x,y,t)$ as real coordinates on $\mathbb H^n$.  Since
	$
	\overline{z'}=x'-iy',
	$
	we have 
	$
	z\cdot \overline{z'}=(x+iy)(x'-iy').
	$
	Expanding this expression componentwise gives
	$$
	(x_j+iy_j)(x'_j-iy'_j)=x_jx'_j+y_jy'_j+i(y_jx'_j-x_jy'_j).
	$$
	Therefore,
	$$
	z\cdot\overline{z'}=\sum_{j=1}^n(x_jx'_j+y_jy'_j)+i\sum_{j=1}^n(y_jx'_j-x_jy'_j), 
	$$
	which implies that
	$$
	\operatorname{Im}(z\cdot\overline{z'})=\sum_{j=1}^n(y_jx'_j-x_jy'_j)=y\cdot x'-x\cdot y'.
	$$
	Consequently, the group law can be written in real coordinates as
	$$
	(x,y,t)(x',y',t')=\left(x+x',y+y',t+t'+2(y\cdot x'-x\cdot y')\right).
	$$
	Lebesgue measure on $\mathbb{R}^{2n+1}$ is a Haar measure on $\mathbb{H}^n$. We denote it by
	$$
	d\eta:=dxdydt.
	$$
	The homogeneous dimension of $\mathbb{H}^n$ is $Q=2n+2$.  For $j=1,\ldots,n$, the system of vector fields, given as differential operators
	$$ 
	X_j := \frac{\partial}{\partial x_j} + 2y_j \frac{\partial}{\partial t},\ \ \ \ \ Y_j := \frac{\partial}{\partial y_j} - 2x_j \frac{\partial}{\partial t}
	$$
	and
	\begin{equation}\label{eqn{partial-T-define}}
		T := \frac{\partial}{\partial t}
	\end{equation}
	forms a basis of the left-invariant vector fields on $ \mathbb{H}^n $. The vectors $X_j $ and $Y_j $ ($j =1, \dots, n$) form a basis of the horizontal bundle and $T_\theta=2T$ ($T_\theta $ is the characteristic direction on $\mathbb H^n$). In order to be consistent with the fractional operator used below, we use the positive normalization of the standard sub-Laplacian
	$$
	\mathcal{L}:=-\frac{1}{4}\sum_{j=1}^n\left(X_j^2+Y_j^2\right).
	$$
	
	The sphere $\mathbb S^{2n+1}$ is the boundary of the unit ball in $\mathbb{C}^{n+1}$. In coordinates, $\zeta = (\zeta_1, \ldots, \zeta_{n+1}) \in \mathbb S^{2n+1}$ if and only if $\zeta \cdot \bar{\zeta} = \sum_{j=1}^{n+1} |\zeta_j|^2 = 1$. The standard Euclidean surface measure on $\mathbb{S}^{2n+1}$ is denoted by $d\zeta$, and the measure of $\mathbb{S}^{2n+1}$ is denoted by $\omega_{2n+1}$. It follows from \cite[Section 3.3]{Frank-Gonzalez-Monticelli-Tan2015} that the Heisenberg group $\mathbb{H}^n$ and the sphere $\mathbb{S}^{2n+1}$ are connected by the Cayley transform
	$$
	\mathcal{C}:\mathbb{H}^n\rightarrow\mathbb{S}^{2n+1}\setminus\{(0,\ldots,0,-1)\},
	$$
	which is defined by
	$$
	\mathcal{C}(\eta)=\mathcal{C}(z,t):=\left(\frac{2iz}{t+i(1+|z|^2)},\frac{-t+i(1-|z|^2)}{t+i(1+|z|^2)}\right)=\left(\frac{2z}{1+|z|^2-it},\frac{1-|z|^2+it}{1+|z|^2-it}\right).
	$$
	Its inverse is given by
	$$
	\eta=\mathcal{C}^{-1}(\zeta)=\left(\frac{\zeta_1}{1+\zeta_{n+1}},\ldots,\frac{\zeta_n}{1+\zeta_{n+1}},\operatorname{Re}\left(i\frac{1-\zeta_{n+1}}{1+\zeta_{n+1}}\right)\right),
	$$
	where $\zeta=(\zeta_1,\ldots,\zeta_{n+1})\in\mathbb{S}^{2n+1}\setminus\{(0,\ldots,0,-1)\}$. The Jacobian determinant (acting as a volume density) of this transformation is explicitly given by
	\begin{equation}\label{Jc-define}
		|J_{\mathcal{C}}(\eta)|=|J_{\mathcal{C}}(z,t)|  := \frac{2^{2n+1}}{\left( (1 + |z|^2)^2 + t^2 \right)^{n+1}}= \frac{2^{Q-1}}{\left( (1 + |z|^2)^2 + t^2 \right)^{n+1}}.
	\end{equation}
	Consequently, for any integrable function $F$ on $\mathbb{S}^{2n+1}$, one has
	\begin{equation}\label{eqn:int-F-FC}
		\int_{\mathbb S^{2n+1}} F d\zeta = \int_{\mathbb{H}^n} (F \circ \mathcal{C})|J_{\mathcal{C}}| d\eta.
	\end{equation}
	The homogeneous norm on $\mathbb{H}^n$ is defined as
	\begin{equation}\label{norm-define}
		|\eta|=|(z, t)| := (|z|^4 + t^2)^\frac{1}{4},
	\end{equation}
	and the distance between $\eta = (z, t)$ and $\tilde{\eta} = (\tilde{z}, \tilde{t} )$ is
	$$
	d(\eta,\tilde{\eta}) := |\tilde{\eta}^{-1}\eta| = \left(|z - \tilde{z}|^4 + (t - \tilde{t} +2\operatorname{Im}(z \cdot \bar{\tilde{z}}))^2\right)^\frac{1}{4}.
	$$
	For any $\zeta,\tilde{\zeta}\in\mathbb S^{2n+1}$, the distance function on $\mathbb S^{2n+1}$ is defined by
	$$
	d(\zeta, \tilde{\zeta}) := (2|1 - \zeta \cdot \bar{\tilde{\zeta}}|)^\frac{1}{2} = \left| |\zeta - \tilde{\zeta}|^2 - 2i \, \operatorname{Im} (\zeta \cdot \bar{\tilde{\zeta}}) \right|^\frac{1}{2}  = \left( |\zeta - \tilde{\zeta}|^4 + 4 \cdot \operatorname{Im}^2(\zeta \cdot \bar{\tilde{\zeta}}) \right)^\frac{1}{4}.
	$$
	Moreover, for $\eta = (z, t)$, $\tilde{\eta} = (\tilde{z}, \tilde{t})$ and $\zeta = \mathcal{C}(\eta)$, $\tilde{\zeta} = \mathcal{C}(\tilde{\eta})$, a simple calculation shows that
	$$
	\frac{|1 - \zeta \cdot \bar{\tilde{\zeta}}|}{2} = |\tilde{\eta}^{-1}\eta|^2 \left(\frac{1}{ (1 + |z|^2)^2 + t^2} \right)^{\frac{1}{2}} \left(\frac{1}{ (1 + |\tilde{z}|^2)^2 + \tilde{t}^2 }\right)^{\frac{1}{2}},
	$$
	which means
	$$
	d(\zeta,\tilde{\zeta}) = d(\eta, \tilde{\eta}) \left( \frac{4}{(1 + |z|^2)^2 + t^2} \right)^\frac{1}{4} \left( \frac{4}{(1 + |\tilde{z}|^2)^2 + \tilde{t}^2} \right)^\frac{1}{4}.
	$$
	The Hilbert space $ L^2(\mathbb S^{2n+1}) $, endowed with the inner product
	$$
	(U, V)_{L^2(\mathbb S^{2n+1})} := \int_{\mathbb S^{2n+1}} U \overline{V} \, d\zeta,
	$$
	admits the orthogonal decomposition 
	$$
	L^2(\mathbb S^{2n+1}) = \bigoplus_{j,k \geq 0} \mathcal H_{j,k},
	$$ 
	where $\mathcal H_{j,k} $ is the space of harmonic polynomials $ p(z,\bar{z})$ on $ \mathbb{C}^{n+1} $ restricted to $\mathbb S^{2n+1}$ that are homogeneous of degree $ j$ and $k $ in the variables $z$ and $\bar{z}$, respectively. The dimension of $\mathcal H_{j,k} $ is
	\begin{equation}\label{eqn:mathcal-H-m-ij-dim}
		\dim(\mathcal H_{j,k}) = m_{j,k} := \frac{(j + n - 1)!(k + n - 1)!(j + k + n)}{n!(n - 1)!j!k!}.
	\end{equation}
	Let $\{Y_{j,k}^{\ell}\}_{l=1,\dots,m_{j,k}}$ be an orthonormal basis for $\mathcal H_{j,k} $. Then the corresponding zonal harmonics are denoted by
	$$
	\Phi_{j,k}(\zeta, \tilde{\zeta}) := \sum_{\ell=1}^{m_{j,k}} Y_{j,k}^{\ell}(\zeta) \overline{Y_{j,k}^{\ell}(\tilde{\zeta})},
	$$
	which can be represented as
	$$
	\Phi_{j,k}(\zeta, \tilde{\zeta}) =\Phi_{j,k}(\zeta\cdot\bar{\tilde{\zeta}}):= \frac{(j+ n - 1)!(j + k + n)}{\omega_{2n+1}n! j !} (\zeta \cdot \bar{\tilde{\zeta}})^{j-k} P_k^{(n-1,j-k)}(2|\zeta \cdot \bar{\tilde{\zeta}}|^2 - 1)
	$$
	if $k\leq j$, and $\Phi_{j,k}(\zeta, \tilde{\zeta}) =\Phi_{j,k}(\zeta\cdot\bar{\tilde{\zeta}}):=\overline{\Phi_{k,j}(\zeta\cdot\bar{\tilde{\zeta}})}$ if $j\leq k$,
	where $ P_k^{(n,\ell)}$ are the Jacobi polynomials (see \cite{Vilenkin-Klimyk1993}) and $\omega_{2n+1}$ is the usual  measure of $\mathbb S^{2n+1}$. For any $ U \in L^2(\mathbb S^{2n+1}) $, its Fourier expansion is
	$$
	U = \sum_{j,k \geq 0} \sum_{l=1}^{m_{j,k}} c_{j,k}^l (U)Y_{j,k}^l
	$$
	with Fourier coefficients $c_{j,k}^l(U) = \int_{\mathbb S^{2n+1}} UY_{j,k}^l d\zeta$. Accordingly, 
	$$
	{D}^\frac{\gamma}{2} U := \sum_{j,k \geq 0} \sum_{l=1}^{m_{j,k}} (\nu_j \nu_k)^\frac{\gamma}{2} c_{j,k}^l (U)Y_{j,k}^l,
	$$
	where $0<\gamma<\frac{Q}{2}$ and $\nu_j:=j+\frac{n}{2}$.
	The fractional Sobolev space over $\mathbb S^{2n+1} $ is defined as
	$$
	H^{\gamma}(\mathbb S^{2n+1}) = W^{\gamma,2}(\mathbb S^{2n+1}) := \Big\{ U \in L^2(\mathbb S^{2n+1}) : {D}^\frac{\gamma}{2} U \in L^2(\mathbb S^{2n+1}) \Big\},
	$$
	endowed with the inner product and norm
	$$
	(U,V)_{H^{\gamma}(\mathbb S^{2n+1})} := \int_{\mathbb S^{2n+1}} {D}^\frac{\gamma}{2}U \overline{{D}^\frac{\gamma}{2} V} d\zeta
	$$
	and
	$$
	\|U\|_{H^{\gamma}(\mathbb S^{2n+1})} =\sqrt{ (U,U)_{H^{\gamma}(\mathbb S^{2n+1})}}:= \Big( \sum_{j,k \geq 0} \sum_{l=1}^{m_{j,k}} (\nu_j \nu_k)^{\gamma} |c_{j,k}^l(U)|^2 \Big)^\frac{1}{2}.
	$$
	The norm $ \|\cdot\|_{H^{\gamma}(\mathbb S^{2n+1})} $ is equivalent to the norm induced by the inner product
	$$
	(U,V)_{\gamma} := \sum_{j,k \geq 0} \sum_{l=1}^{m_{j,k}} \lambda_j(\gamma) \lambda_k(\gamma) c_{j,k}^l(U) \overline{c_{j,k}^l(V)},
	$$
	where
	\begin{equation}\label{eqn:lambda-j-define}
		\lambda_j(\gamma) := \frac{\Gamma\Big(\frac{Q+2\gamma}{4} + j\Big)}{\Gamma\Big(\frac{Q-2\gamma}{4} + j\Big)}
	\end{equation}
	for $j \in\{0, 1, 2, \dots\} $. Indeed, by asymptotic approximation of the Gamma function $\Gamma$, one has $\lambda_j(\gamma) \sim j^\gamma$. The intertwining operator $\mathcal{A}_\gamma$ of order $2\gamma$ on $\mathbb S^{2n+1}$ is given by
	$$
	|J_{\hat{\tau}}|^{\frac{Q+2\gamma}{2Q}}(\mathcal{A}_\gamma U)\circ\hat{\tau}= \mathcal{A}_\gamma\Big( |J_{\hat{\tau}}|^{\frac{Q-2\gamma}{2Q}}(U \circ\hat{\tau}) \Big) 
	$$
	for all $\hat{\tau} \in \operatorname{Aut}(\mathbb S^{2n+1})$ and $U \in C^\infty(\mathbb S^{2n+1})$, where $\operatorname{Aut}(\mathbb S^{2n+1})$ and $J_{\hat{\tau}}$ denote the group of automorphisms on $\mathbb S^{2n+1}$ and the Jacobian of $\hat{\tau} \in \operatorname{Aut}(\mathbb S^{2n+1})$, respectively. In fact, the latter definition can be extended to every $U \in H^\gamma(\mathbb S^{2n+1})$. The intertwining property determines $\mathcal{A}_\gamma$ up to a multiplicative constant, and we fix the normalization by requiring
	\begin{equation}\label{eqn:A-gamma-action}
		\mathcal{A}_\gamma Y_{j,k} = \lambda_j(\gamma)\lambda_k(\gamma)Y_{j,k}, \quad Y_{j,k} \in \mathcal{H}_{j,k},
	\end{equation}
	see \cite[(1.11)]{Branson-Fontana-Morpurgo2013} for details.
	It also follows from (\ref{eqn:lambda-j-define}) and (\ref{eqn:A-gamma-action}) that the intertwining operator $\mathcal{A}_\gamma$ possesses the eigenvalues
	\begin{equation}\label{eqn:lambda-i-j-expression}
		\lambda_{j,k}(\gamma) := \lambda_{j}(\gamma)\lambda_{k}(\gamma)=
		\frac{\Gamma\left(\frac{Q+2\gamma}{4} + j\right)\Gamma\left(\frac{Q+2\gamma}{4} + k\right)}{\Gamma\left(\frac{Q-2\gamma}{4} + j\right)\Gamma\left(\frac{Q-2\gamma}{4} + k\right)} .
	\end{equation}
	Therefore,
	$$
	(U,V)_\gamma = \int_{\mathbb S^{2n+1}} \overline{V}\mathcal{A}_\gamma U d\zeta. 
	$$
	Moreover, according to Frank and Lieb \cite{Frank-Lieb2012}, for every real-valued function $U \in H^\gamma(\mathbb S^{2n+1})$, the following sharp fractional Sobolev inequality holds on the CR sphere $\mathbb S^{2n+1}$
	\begin{equation}\label{eqn:sharp-Sobolev-CR}
		\Big( \int_{\mathbb S^{2n+1}} |U(\zeta)|^{\frac{2Q}{Q-2\gamma}} d\zeta \Big)^{\frac{Q-2\gamma}{Q}} \leq \frac{\Gamma\Big(\frac{n+1-\gamma}{2}\Big)^2}{\Gamma\Big(\frac{n+1+\gamma}{2}\Big)^2} \omega_{2n+1}^{-\frac{\gamma}{n+1}} \int_{\mathbb S^{2n+1}} U(\zeta)\mathcal{A}_\gamma U(\zeta) d\zeta.
	\end{equation}
	It follows immediately from the definition by direct computation that $\mathcal{L}_1 = \mathcal{L}$, $\mathcal{L}_2 = \mathcal{L}^2 - |T|^2$. Moreover, \cite[Section 4]{Frank-Gonzalez-Monticelli-Tan2015} implies that $\mathcal P^\theta_\gamma$ agrees with the intertwining operators on the CR sphere calculated in \cite{Branson-Fontana-Morpurgo2013}, hence the relation between $\mathcal{L}_\gamma$ and $\mathcal{A}_\gamma$ is
	\begin{equation}\label{eqn:relation-L-A}
		\begin{aligned}
			&2^{-\gamma}\mathcal P^\theta_\gamma\Big( (2|J_\mathcal{C}|)^{\frac{Q-2\gamma}{2Q}} (U\circ \mathcal{C}) \Big)\\ &=\mathcal{L}_\gamma\Big( (2|J_\mathcal{C}|)^{\frac{Q-2\gamma}{2Q}} (U\circ \mathcal{C}) \Big)\\
			& = (2|J_\mathcal{C}|)^{\frac{Q+2\gamma}{2Q}} (\mathcal{A}_\gamma U) \circ\mathcal{C}, \ \ U \in H^\gamma(\mathbb S^{2n+1}).
		\end{aligned}
	\end{equation}
	The multiplicative factor $2$ appears because we use the standard volume elements for $\mathbb{H}^n$ and $\mathbb{S}^{2n+1}$ instead of the volume forms associated with the standard contact forms $\theta_0$ and $\theta$ of these two spaces. Further details are provided in \cite[Pages 11 and 21]{Branson-Fontana-Morpurgo2013}. Moreover, the fractional Sobolev space over $\mathbb{H}^n$ is defined by
	$$
	\mathcal D^\gamma(\mathbb{H}^n) := \Big\{ u \in L^{\frac{2Q}{Q-2\gamma}}(\mathbb{H}^n) : a_\gamma[u] < +\infty \Big\},
	$$
	where the quadratic form $ a_\gamma $ is associated with the operator $ \mathcal{L}_\gamma $, that is,
	\begin{equation}\label{eqn:a-gamma-u-define}
		a_\gamma[u] := \frac{1}{2^\gamma}\int_{\mathbb{H}^n}u \mathcal{P}^\theta_\gamma u d\eta=\int_{\mathbb{H}^n}u \mathcal{L}_\gamma u d\eta.
	\end{equation}
	The form $ a_\gamma $ can be equivalently represented by means of spectral decomposition, see \cite[Page 126]{Frank-Gonzalez-Monticelli-Tan2015}. Moreover, we equip
	$\mathcal D^\gamma(\mathbb H^n)$ with the inner product
	\begin{equation}\label{eqn:sec2-D-int-defn}
		(u,v)_{\mathcal D^\gamma(\mathbb H^n)}:=\int_{\mathbb H^n}v \mathcal L_\gamma ud\eta,
		\ \  u,v\in\mathcal D^\gamma(\mathbb H^n).
	\end{equation}
	The corresponding norm is given by
	$$
	\|u\|_{\mathcal D^\gamma(\mathbb H^n)}=\sqrt{(u,u)_{\mathcal D^\gamma(\mathbb H^n)}}
	=(a_\gamma[u])^\frac{1}{2}.
	$$
	We denote its dual space by $\mathcal{D}^{-\gamma}(\mathbb{H}^{n})$, so that the fractional sub-Laplacian $\mathcal{L}_{\gamma}$ acts as an isometry from $\mathcal{D}^{\gamma}(\mathbb{H}^{n})$ to $\mathcal{D}^{-\gamma}(\mathbb{H}^{n})$. We use \(\langle\cdot,\cdot\rangle\) to denote the duality pairing between \(\mathcal D^{-\gamma}(\mathbb H^n)\) and \(\mathcal D^\gamma(\mathbb H^n)\). For any fixed $\hat{\xi} \in \mathbb{H}^n$, we define the left translation $\tau_{\hat{\xi}} : \mathbb{H}^n \to \mathbb{H}^n$ by $\tau_{\hat{\xi}}(\eta) := \hat{\xi} \eta$, utilizing the group law introduced in (\ref{cric-define}). Furthermore, for any $\lambda > 0$, we define the dilation $\delta_{\lambda} : \mathbb{H}^n \to \mathbb{H}^n$ as
	$$
	\delta_{\lambda}(\eta)=\delta_{\lambda}(z,t) := (\lambda z, \lambda^2 t).
	$$
	It is straightforward to verify that $\delta_{\lambda}$ is a group automorphism satisfying 
	$$
	\delta_{\lambda}(\hat{\xi}  \eta) = \delta_{\lambda}(\hat{\xi})  \delta_{\lambda}(\eta)
	$$ 
	for any $\eta, \hat{\xi} \in \mathbb{H}^n$. Observe that the norm $|\cdot|$ defined in (\ref{norm-define}) is homogeneous of degree one with respect to these dilations, that is,
	$$
	|\delta_{\lambda}(\eta)|= \lambda |\eta| 
	$$
	for any $\eta \in \mathbb{H}^n$ and $\lambda > 0$. Finally, let $B_d(\eta, r)$ denote the Kor\'anyi ball centered at $\eta$ with radius $r$ associated with the distance $d$. By the properties of translation and dilation, these balls satisfy the relations 
	$
	B_d(\eta, r) = \tau_{\eta}(B_d(0, r))
	$ 
	and 
	$
	B_d(0, r) = \delta_r(B_d(0, 1)).
	$
	
	We now recall the sharp Hardy-Littlewood-Sobolev inequality on $\mathbb{H}^n$.
	\begin{proposition}\cite[Theorem 2.1]{Frank-Lieb2012}\label{Frank-Lieb-Theorem}
		Let $0 < d < Q$ and $p := \frac{2Q} {2Q - d }$. Then for any $f, g \in L^p(\mathbb{H}^n)$,
		$$
		\Big| \int_{\mathbb{H}^n}\int_{\mathbb{H}^n} \frac{\overline{f(\eta)}g(\xi)}{|\eta^{-1}\xi|^d} \, d\eta \, d\xi \Big| \le \Big( \frac{\pi^{n+1}}{2^{n-1}n!} \Big)^\frac{d}{Q} \frac{n! \Gamma(\frac{Q-d}{2})}{\Gamma^2(\frac{2Q-d}{4})} \|f\|_p \|g\|_p,
		$$
		with equality if and only if
		$$
		f(\eta) = c H(\bar{\delta}(a^{-1}\eta)), \quad g(\xi) = c' H(\bar{\delta}(a^{-1}\xi))
		$$
		for some $c, c' \in \mathbb{C}$, $\bar{\delta}> 0$ and $a \in \mathbb{H}^n$ (unless $f \equiv 0$ or $g \equiv 0$).  Here 
		$$
		H(\eta)=H(z,t):=\frac{1}{((1+|z|^2)^2+t^2)^\frac{2Q-d}{4}}.
		$$
	\end{proposition}
	Since the fundamental solution of $\mathcal L_\gamma$ is given by (see \cite[(1.34)]{Branson-Fontana-Morpurgo2013}) 
	$$
	\mathcal L^{-1}_\gamma (\eta,\xi)=\frac{2^{n-\gamma-1}\Gamma^2(\frac{Q-2\gamma}{4})}{\pi^{n+1}\Gamma(\gamma)}|\eta^{-1}\xi|^{2\gamma -Q},
	$$
	using a duality argument (similar to \cite[Theorem 8.3]{Lieb-Loss2001} for the Euclidean case) and setting $d=Q-2\gamma$, Proposition \ref{Frank-Lieb-Theorem} is equivalent to the following:
	\begin{proposition}\cite[Theorem 2.2]{Chen-Wang2017}\label{Chen-Wang2017-result}
		Let $\gamma \in (0, \frac{Q}{2})$ and set $q := \frac{2Q}{Q-2\gamma}$. Then for any $u \in \mathcal D^\gamma(\mathbb{H}^n)$, it holds that
		$$
		\|u\|_{L^{q}(\mathbb{H}^n)}^2 \le \frac{2^{n-\gamma-1-\frac{2(n-1)\gamma}{Q}} (\pi^{n+1})^{\frac{2\gamma}{Q}-1}}{(n!)^{\frac{2\gamma}{Q}-1}} \cdot \frac{\Gamma(\frac{Q-2\gamma}{2}) \Gamma^2(\frac{Q-2\gamma}{4})}{\Gamma^2(\frac{Q-\gamma}{2}) \Gamma(\gamma)} a_\gamma[u],
		$$
		where $a_\gamma$ is given in (\ref{eqn:a-gamma-u-define}).	Moreover, all optimizers for the above inequality are translates, dilations and constant multiples of 
		\begin{equation}\label{eqn:omega-define}
			\omega_0(\eta) = \omega_0(z,t): =  \frac{1}{((1+|z|^2)^2 + t^2)^\frac{Q-2\gamma}{4}}.
		\end{equation}
	\end{proposition}
	We next determine the exact normalization for which the standard extremal profile in (\ref{eqn:omega-define}) solves (\ref{main-problem}).
	\begin{lemma}\label{solution-sca-tran}
		Let $\gamma\in(0,\frac{Q}{2})$,
		\begin{equation}\label{eqn:C-kappa-define-relation}
			C(n,\gamma):=\bigl(\kappa(n,\gamma)\bigr)^{\frac{Q-2\gamma}{4\gamma}}
		\end{equation}
		and
		$$
		\kappa(n,\gamma):=2^{2\gamma}\left(\frac{\Gamma\left(\frac{Q+2\gamma}{4}\right)}{\Gamma\left(\frac{Q-2\gamma}{4}\right)}\right)^2.
		$$
		Then,
		\begin{equation}\label{eqn:omega-u-define}
			\omega(\eta):=C(n,\gamma)\omega_0(\eta)=C(n,\gamma)\omega_0(z,t)=\frac{C(n,\gamma)}{((1+|z|^2)^2 + t^2)^\frac{Q-2\gamma}{4}}
		\end{equation}
		is a positive solution to (\ref{main-problem}). Moreover, for any $\lambda>0$ and $\xi\in\mathbb H^n$, the function defined by
		\begin{equation}\label{eqn:w-lambda-xi-define}
			\omega_{\lambda, \xi}(\eta) := \lambda^{\frac{Q-2\gamma}{2}} \omega(\delta_\lambda(\tau_{\xi^{-1}}(\eta)))= \lambda^{\frac{Q-2\gamma}{2}} \omega(\delta_\lambda(\xi^{-1}\eta))
		\end{equation}
		is also a positive solution to (\ref{main-problem}).
	\end{lemma}
	
	\begin{proof}
		Since the space $\mathcal H_{0,0}$ consists of constant functions on $\mathbb S^{2n+1}$, we have $1\in\mathcal H_{0,0}$. Moreover, it follows from (\ref{eqn:lambda-j-define}) and (\ref{eqn:A-gamma-action}) that
		$$
		\mathcal A_\gamma 1=\lambda_0(\gamma)^2=\left(\frac{\Gamma\left(\frac{Q+2\gamma}{4}\right)}{\Gamma\left(\frac{Q-2\gamma}{4}\right)}\right)^2=\lambda_{0,0}(\gamma).
		$$
		Taking $U\equiv1$ in (\ref{eqn:relation-L-A}), we have
		$$
		\mathcal L_\gamma\Big((2|J_{\mathcal C}|)^{\frac{Q-2\gamma}{2Q}}\Big)
		=\lambda_{0,0}(\gamma)
		(2|J_{\mathcal C}|)^{\frac{Q+2\gamma}{2Q}}.
		$$
		By using (\ref{Jc-define}) and (\ref{eqn:omega-define}), one has
		$$
		\begin{aligned}
			(2|J_{\mathcal C}|)^{\frac{Q-2\gamma}{2Q}}&=\left(\frac{2}
			{\left((1+|z|^2)^2+t^2\right)^\frac{1}{2}}\right)^\frac{Q-2\gamma}{2}\\
			&=2^\frac{Q-2\gamma}{2}\cdot\frac{1}{\left((1+|z|^2)^2+t^2\right)^\frac{Q-2\gamma}{4}}
			=2^{\frac{Q-2\gamma}{2}}\omega_0
		\end{aligned}
		$$
		and
		$$
		\begin{aligned}
			(2|J_{\mathcal C}|)^{\frac{Q+2\gamma}{2Q}}&=\left(\frac{2}
			{\left((1+|z|^2)^2+t^2\right)^\frac{1}{2}}\right)^\frac{Q+2\gamma}{2}\\
			&=2^\frac{Q+2\gamma}{2}\cdot\frac{1}{\left((1+|z|^2)^2+t^2\right)^\frac{Q+2\gamma}{4}}
			=2^{\frac{Q+2\gamma}{2}}\omega_0^\frac{Q+2\gamma}{Q-2\gamma}.
		\end{aligned}
		$$
		Therefore,
		$$
			2^{\frac{Q-2\gamma}{2}} \mathcal L_\gamma \omega_0=\mathcal L_\gamma\Big(2^{\frac{Q-2\gamma}{2}}\omega_0\Big)
			=\mathcal L_\gamma\Big((2|J_{\mathcal C}|)^{\frac{Q-2\gamma}{2Q}}\Big)
			=\lambda_{0,0}(\gamma)
			(2|J_{\mathcal C}|)^{\frac{Q+2\gamma}{2Q}}
			=\lambda_{0,0}(\gamma)2^{\frac{Q+2\gamma}{2}}\omega_0^\frac{Q+2\gamma}{Q-2\gamma},
		$$
		which yields
		$$
		\mathcal L_\gamma \omega_0=2^{2\gamma}\lambda_{0,0}(\gamma)
		\omega_0^{\frac{Q+2\gamma}{Q-2\gamma}}=2^{2\gamma}
		\left(\frac{\Gamma\left(\frac{Q+2\gamma}{4}\right)}{\Gamma\left(\frac{Q-2\gamma}{4}\right)}\right)^2\omega_0^{\frac{Q+2\gamma}{Q-2\gamma}}=\kappa(n,\gamma)\omega_0^{\frac{Q+2\gamma}{Q-2\gamma}}.
		$$
		Since $\omega=C(n,\gamma)\omega_0$, we have
		$$
		\mathcal L_\gamma\omega=C(n,\gamma)\mathcal L_\gamma\omega_0
		=C(n,\gamma)\kappa(n,\gamma)\omega_0^{\frac{Q+2\gamma}{Q-2\gamma}}.
		$$
		Combining this with the definition of $C(n,\gamma)$ in (\ref{eqn:C-kappa-define-relation}), we deduce that
		$$
		\begin{aligned}
			\mathcal L_\gamma\omega
			&=C(n,\gamma)\kappa(n,\gamma)\omega_0^{\frac{Q+2\gamma}{Q-2\gamma}}
			=C(n,\gamma)C(n,\gamma)^{\frac{4\gamma}{Q-2\gamma}}\omega_0^{\frac{Q+2\gamma}{Q-2\gamma}}\\
			&=C(n,\gamma)^{\frac{Q+2\gamma}{Q-2\gamma}}\omega_0^{\frac{Q+2\gamma}{Q-2\gamma}}
			=(C(n,\gamma)\omega_0)^{\frac{Q+2\gamma}{Q-2\gamma}}=\omega^{\frac{Q+2\gamma}{Q-2\gamma}}.\\
		\end{aligned}
		$$
		Thus, $\omega$ is a positive solution to (\ref{main-problem}). Furthermore, by the translation and dilation invariance of (\ref{main-problem}) (see \cite[Page 537]{Chen-Wang2017}), we conclude that $\omega_{\lambda, \xi}(\eta)$
		is also a positive solution to (\ref{main-problem}).
		
	\end{proof}

	\section{Non-degeneracy and Local Rigidity}\label{proof of main result}
	Given a function $g$ on $\mathbb{H}^n$, we define an associated function $\rho$ on $\mathbb{S}^{2n+1}$ by setting
	$$
	\rho(\zeta) := g(\mathcal{C}^{-1}(\zeta)) \Big(\frac{|J_{\mathcal{C}^{-1}}(\zeta)|}{2}\Big)^{\frac{Q-2\gamma}{2Q}} 
	$$
	for $\zeta \in \mathbb{S}^{2n+1} \setminus \{(0, \ldots, 0, -1)\}$. Equivalently, for $\eta=(z,t) \in \mathbb{H}^n$ with $\zeta=\mathcal C(\eta)$, the inverse relation reads as
	\begin{equation}\label{eqn:gu-define}
		g(\eta) = (2|J_{\mathcal{C}}(\eta)|)^{\frac{Q-2\gamma}{2Q}} (\rho \circ \mathcal{C})(\eta) = (2|J_{\mathcal{C}}(\eta)|)^{\frac{Q-2\gamma}{2Q}}\rho(\mathcal{C}(\eta)).
	\end{equation}
	Furthermore, a direct computation yields
	\begin{equation}\label{eqn:mathcalS-Hn-norm-eq}
		\begin{aligned}
			\|\rho\|^\frac{2Q}{Q-2\gamma}_{L^\frac{2Q}{Q-2\gamma}(\mathbb{S}^{2n+1})}
			&=\int_{\mathbb{S}^{2n+1}}|\rho(\zeta)|^\frac{2Q}{Q-2\gamma}d\zeta
			=\int_{\mathbb{H}^{n}}|\rho(\mathcal C(\eta))|^\frac{2Q}{Q-2\gamma}|J_\mathcal C(\eta)|d\eta\\
			&=\int_{\mathbb{H}^{n}}\left|g(\eta)\Big(\frac{1}{2|J_\mathcal C(\eta)|}\Big)^\frac{Q-2\gamma }{2Q} \right|^\frac{2Q}{Q-2\gamma}|J_\mathcal C(\eta)|d\eta\\
			&=\frac{1}{2}\int_{\mathbb{H}^{n}}|g(\eta)|^\frac{2Q }{Q-2\gamma}d\eta=\frac{1}{2}\|g\|^\frac{2Q}{Q-2\gamma}_{L^\frac{2Q}{Q-2\gamma}(\mathbb{H}^{n})}.
		\end{aligned}
	\end{equation}
	In the following, we recall the relationship between the fractional Sobolev spaces on the CR sphere and the Heisenberg group, which can be found in \cite[Proposition 3.1]{Kristaly2020}.
	\begin{proposition}\label{prop:Cayley-correspondence}
		Let $0<\gamma<\frac{Q}{2}$. Then $\rho\in H^\gamma(S^{2n+1})$ is a weak solution of
		$$
		A_\gamma\rho=|\rho|^\frac{4\gamma}{Q-2\gamma}\rho\ \ \ \text{on}\ \ \ \mathbb S^{2n+1}
		$$
		if and only if $g\in D^\gamma(\mathbb H^n)$ is a weak solution of
		$$
		\mathcal L_\gamma g=|g|^\frac{4\gamma}{Q-2\gamma}g\ \ \ \text{on}\ \ \ \mathbb H^n.
		$$
	\end{proposition}
	
	\textbf{Proof of Theorem \ref{main-result-1}.} By virtue of \eqref{eqn:gu-define} and \eqref{eqn:relation-L-A}, we deduce that
	\begin{equation}\label{eqn:L-gamma-g-express}
		\mathcal{L}_\gamma g(\eta) = \mathcal{L}_\gamma\Big( (2|J_\mathcal{C}(\eta)|)^{\frac{Q-2\gamma}{2Q}} (\rho\circ \mathcal{C})(\eta) \Big) = (2|J_\mathcal{C}(\eta)|)^{\frac{Q+2\gamma}{2Q}} (\mathcal{A}_\gamma \rho)(\mathcal{C}(\eta)).
	\end{equation}
	As a consequence of \eqref{Jc-define}, \eqref{eqn:omega-define}, and \eqref{eqn:w-lambda-xi-define}, the Jacobian determinant $|J_{\mathcal{C}}(\eta)|$ admits the formulation
	\begin{equation}\label{eqn:J-C-omega-10-relation}
		\begin{aligned}
			|J_{\mathcal{C}}(\eta)|=	|J_{\mathcal{C}}(z,t)| &:= \frac{2^{2n+1}}{\left( (1 + |z|^2)^2 + t^2 \right)^{n+1}} 
			= \frac{2^{Q-1}}{\left( (1 + |z|^2)^2 + t^2 \right)^\frac{Q}{2}}\\
			&= 2^{Q-1}\left(\frac{1}{((1+|z|^2)^2 + t^2)^{\frac{Q-2\gamma}{4}}}\right)^ \frac{2Q}{Q-2\gamma}\\
			&= 2^{Q-1}(\omega_{0}(z,t))^{Q^*} = 2^{Q-1}\Big(\frac{\omega_{1,0}(z,t)}{C(n,\gamma)}\Big)^{Q^*},
		\end{aligned}
	\end{equation}
	where $\omega_{1,0}=C(n,\gamma)\omega_0$. Utilizing this identity alongside \eqref{eqn:linear-equation-f}, we obtain that
	$$
	\begin{aligned}
		\mathcal{L}_\gamma g(\eta) 
		&= (Q^*-1) \omega_{1,0}^\frac{4\gamma}{Q-2\gamma} g(\eta) 
		= (Q^*-1) \omega_{1,0}^\frac{4\gamma}{Q-2\gamma} \Big((2|J_\mathcal{C}(\eta)|)^{\frac{Q-2\gamma}{2Q}} \rho\circ\mathcal{C}\Big)\\
		&= (Q^*-1) \left(\frac{C(n,\gamma)^{Q^*}|J_\mathcal{C}(\eta)|}{2^{Q-1}}\right)^{\frac{1}{Q^*}\cdot\frac{4\gamma}{Q-2\gamma}} \Big((2|J_\mathcal{C}(\eta)|)^{\frac{Q-2\gamma}{2Q}} \rho\circ\mathcal{C}\Big)\\
		&= (Q^*-1)C(n,\gamma)^\frac{4\gamma}{Q-2\gamma}|J_\mathcal{C}(\eta)|^\frac{2\gamma}{Q}2^\frac{2\gamma(1-Q)}{Q}(2|J_\mathcal{C}(\eta)|)^\frac{Q-2\gamma}{2Q}\rho\circ\mathcal{C}\\
		&= (Q^*-1)C(n,\gamma)^{Q^*-2}|J_\mathcal{C}(\eta)|^\frac{2\gamma}{Q}2^{\frac{2\gamma}{Q}-2\gamma}(2|J_\mathcal{C}(\eta)|)^\frac{Q-2\gamma}{2Q}\rho\circ\mathcal{C}\\
		&= \frac{(Q^*-1)C(n,\gamma)^{Q^*-2}}{2^{2\gamma}}(2|J_{\mathcal{C}}(\eta)|)^\frac{Q+2\gamma}{2Q}\rho\circ\mathcal{C}\\
		&= (Q^*-1)\left(\frac{\Gamma(\frac{Q+2\gamma}{4})}{\Gamma(\frac{Q-2\gamma}{4})}\right)^2(2|J_{\mathcal{C}}(\eta)|)^\frac{Q+2\gamma}{2Q}\rho\circ\mathcal{C}.
	\end{aligned}
	$$
	In light of \eqref{eqn:L-gamma-g-express}, we conclude that the function $g\in \mathcal D^\gamma(\mathbb H^n)$ is a weak solution of \eqref{eqn:linear-equation-f} if and only if the corresponding profile $\rho$ satisfies the eigenvalue problem
	$$
	\mathcal{A}_\gamma \rho = (Q^*-1)\left(\frac{\Gamma(\frac{Q+2\gamma}{4})}{\Gamma(\frac{Q-2\gamma}{4})}\right)^2\rho = \lambda_{1,0}\rho = \lambda_{0,1}\rho \quad \mbox{on} \quad \mathbb S^{2n+1}.
	$$
	Here, the exact values of the eigenvalues emerge from the fundamental properties of the Gamma function
	$$
	\begin{aligned}
		\lambda_{1,0}(\gamma) &=
		\frac{\Gamma\left(\frac{Q+2\gamma}{4} + 1\right)\Gamma\left(\frac{Q+2\gamma}{4} \right)}{\Gamma\left(\frac{Q-2\gamma}{4} + 1\right)\Gamma\left(\frac{Q-2\gamma}{4} \right)} 
		= \frac{\frac{Q+2\gamma}{4}\Gamma\left(\frac{Q+2\gamma}{4} \right)\Gamma\left(\frac{Q+2\gamma}{4} \right)}{\frac{Q-2\gamma}{4}\Gamma\left(\frac{Q-2\gamma}{4} \right)\Gamma\left(\frac{Q-2\gamma}{4} \right)} \\
		&= (Q^*-1)\left(\frac{\Gamma(\frac{Q+2\gamma}{4})}{\Gamma(\frac{Q-2\gamma}{4})}\right)^2 
		= \lambda_{0,1}(\gamma).
	\end{aligned} 
	$$
	Moreover, the eigenspace associated with the eigenvalues $\lambda_{1,0}(\gamma)$ and $\lambda_{0,1}(\gamma)$ is generated by the $(n+1)$-dimensional complex coordinate functions $\zeta_j$ and their conjugates $\bar{\zeta}_j$ for $j=1,\dots,n+1$, explicitly given by
	$$
	\zeta_j = \frac{2z_j}{1 + |z|^2 - it} \quad (j=1,\dots,n),\quad \zeta_{n+1} = \frac{1 - |z|^2 + it}{1 + |z|^2 - it}
	$$
	and
	$$
	\bar{\zeta}_j = \frac{2\bar{z}_j}{1 + |z|^2 +it} \quad (j=1,\dots,n),\quad \bar{\zeta}_{n+1} = \frac{1 - |z|^2 - it}{1 + |z|^2 + it},
	$$
	where $\eta=(z,t)\in\mathbb H^n$, $z_j=x_j+i y_j$ and $\zeta=\mathcal C(\eta)=(\zeta_1\dots,\zeta_{n+1})\in\mathbb S^{2n+1}$. Drawing upon the spectral analysis of intertwining operators conducted by Branson et al. \cite[Appendix A]{Branson-Fontana-Morpurgo2013}, these eigenvalues correspond to the first non-trivial space $\mathcal{H}_{1,0} \oplus \mathcal{H}_{0,1}$ in the spherical harmonic decomposition, which possesses a total real dimension of $2n+2$. Furthermore, $\mathcal{H}_{1,0} \oplus \mathcal{H}_{0,1}$ is spanned by $\operatorname{Re}(\zeta_j)$ and $\operatorname{Im}(\zeta_j)$ for $j=1,\dots,n+1$. 
	We now identify the corresponding kernel functions by considering separately the two conjugate components $\mathcal H_{1,0}$ and $\mathcal H_{0,1}$. 
	
	\textbf{$(i)$ The space $\mathcal H_{1,0}$.}  This space is generated by the coordinate functions
	$\zeta_1,\ldots,\zeta_{n+1}$ on $\mathbb S^{2n+1}$. 
	For notational convenience, we introduce 
	$$
	\varrho(\eta)=\varrho(z,t) := (1+|z|^2)^2 + t^2.
	$$ 
	A straightforward computation then yields the following partial derivatives
	$$
	\frac{\partial \varrho}{\partial x_j}=4x_j(1+|z|^2), 
	\quad
	\frac{\partial \varrho}{\partial y_j}=4y_j(1+|z|^2),
	\quad
	\frac{\partial \varrho}{\partial t}=2t.
	$$
	It follows from \eqref{eqn:w-lambda-xi-define} that the standard positive bubble $\omega_{1,0}$ is 
	\begin{equation}\label{eqn:omega-10-varrho-rel}
		\omega_{1,0}(\eta)=\omega_{1,0}(z,t) =\frac{C(n,\gamma)}{((1+|z|^2)^2 + t^2)^\frac{Q-2\gamma}{4}}= C(n,\gamma) \varrho^{-\frac{Q-2\gamma}{4}}.
	\end{equation}
	Differentiating $\omega_{1,0}$ with respect to the variables $x_j, y_j$ and $t$ gives
	$$
	\begin{aligned}
		\frac{\partial \omega_{1,0}}{\partial x_j}&=C(n,\gamma)\frac{\partial \varrho^{-\frac{Q-2\gamma}{4}}}{\partial x_j}
		=-\frac{Q-2\gamma}{4}C(n,\gamma)\varrho^{-\frac{Q-2\gamma}{4}-1}\cdot4x_j(1+|z|^2)\\
		&=-(Q-2\gamma)C(n,\gamma)\varrho^{-\frac{Q-2\gamma}{4}-1}x_j(1+|z|^2),
	\end{aligned}
	$$
	$$
	\begin{aligned}
		\frac{\partial \omega_{1,0}}{\partial y_j}&=C(n,\gamma)\frac{\partial \varrho^{-\frac{Q-2\gamma}{4}}}{\partial y_j}
		=-\frac{Q-2\gamma}{4}C(n,\gamma)\varrho^{-\frac{Q-2\gamma}{4}-1}\cdot4y_j(1+|z|^2)\\
		&=-(Q-2\gamma)C(n,\gamma)\varrho^{-\frac{Q-2\gamma}{4}-1}y_j(1+|z|^2),
	\end{aligned}
	$$
	and
	\begin{equation}\label{eqn:partial-omega-10-t}
		\begin{aligned}
			\frac{\partial  \omega_{1,0}}{\partial t}&=C(n,\gamma)\frac{\partial \varrho^{-\frac{Q-2\gamma}{4}}}{\partial t}
			=-\frac{Q-2\gamma}{4}C(n,\gamma)\varrho^{-\frac{Q-2\gamma}{4}-1}\cdot 2t\\
			&=-\frac{Q-2\gamma}{2}C(n,\gamma)\varrho^{-\frac{Q-2\gamma}{4}-1}t.
		\end{aligned}
	\end{equation}
	Evaluating the vector fields 
	$$
	X^R_j \omega_{\lambda,\xi}:=(\partial_{x_j} -2y_j \partial_t)\omega_{\lambda,\xi}
	$$ 
	and 
	$$
	Y^R_j \omega_{\lambda,\xi}:=(\partial_{y_j} + 2x_j \partial_t)\omega_{\lambda,\xi}
	$$ 
	at the point $(\lambda, \xi) = (1,0)$, we obtain
	$$
	\begin{aligned}
		X^R_j \omega_{\lambda,\xi}\big|_{(1,0)}&= \Big(\frac{\partial \omega_{\lambda,\xi}}{\partial x_j} - 2y_j \frac{\partial  \omega_{\lambda,\xi}}{\partial t}\Big)\bigg|_{(1,0)}
		=\frac{\partial \omega_{1,0}}{\partial x_j} -2y_j \frac{\partial  \omega_{1,0}}{\partial t}\\
		&=-(Q-2\gamma)C(n,\gamma)\varrho^{-\frac{Q-2\gamma}{4}-1}x_j(1+|z|^2)+2y_j\cdot\frac{Q-2\gamma}{2}C(n,\gamma)\varrho^{-\frac{Q-2\gamma}{4}-1}t  \\
		&=-(Q-2\gamma)C(n,\gamma)\varrho^{-\frac{Q-2\gamma}{4}-1}\Big(x_j(1+|z|^2)-y_j t\Big)
	\end{aligned}
	$$
	and
	$$
	\begin{aligned}
		Y^R_j \omega_{\lambda,\xi}\big|_{(1,0)}&= \Big(\frac{\partial \omega_{\lambda,\xi}}{\partial y_j} +2x_j \frac{\partial  \omega_{\lambda,\xi}}{\partial t}\Big)\bigg|_{(1,0)}
		=\frac{\partial \omega_{1,0}}{\partial y_j} +2x_j \frac{\partial  \omega_{1,0}}{\partial t}\\
		&=-(Q-2\gamma)C(n,\gamma)\varrho^{-\frac{Q-2\gamma}{4}-1}y_j(1+|z|^2)-2x_j\cdot\frac{Q-2\gamma}{2}C(n,\gamma)\varrho^{-\frac{Q-2\gamma}{4}-1}t  \\
		&=-(Q-2\gamma)C(n,\gamma)\varrho^{-\frac{Q-2\gamma}{4}-1}\Big(y_j(1+|z|^2)+x_j t\Big).
	\end{aligned}
	$$
	To bridge these expressions with the Cayley transform $\mathcal{C}: \mathbb{H}^n \to \mathbb{S}^{2n+1}$, we observe that for $j=1, \dots, n$, the components of the transform are given by $\zeta_j = \frac{2z_j}{1+|z|^2-it}$. Separating this into real and imaginary parts yields
	$$
	\begin{aligned}
		\zeta_j = \frac{2z_j}{1+|z|^2-it} &= \frac{2(x_j+iy_j)}{(1+|z|^2) - it}= \frac{2(x_j+iy_j) \big( (1+|z|^2) + it \big)}{(1+|z|^2)^2 + t^2} \\
		&= \frac{2 \big(x_j(1+|z|^2) - y_j t \big) + 2i \big( y_j(1+|z|^2) + x_j t \big)}{\varrho},
	\end{aligned}
	$$
	which implies
	\begin{equation}\label{eqn:express-Re-zeta}
		\operatorname{Re}(\zeta_j) = \frac{2 \big( x_j(1+|z|^2) - y_j t \big)}{\varrho} 
	\end{equation}
	and
	\begin{equation}\label{eqn:express-Im-zeta}
		\operatorname{Im}(\zeta_j)= \frac{2 \big( y_j(1+|z|^2) +x_j t \big)}{\varrho}.
	\end{equation}
	Invoking the relation between $\omega_{1,0}$ and $|J_{\mathcal{C}}| $ presented in \eqref{eqn:J-C-omega-10-relation}:
	$$
	|J_{\mathcal{C}}| =2^{Q-1}\Big(\frac{\omega_{1,0}}{C(n,\gamma)}\Big)^{Q^*}=2^{Q-1}\left(\frac{C(n,\gamma)\varrho^{-\frac{Q-2\gamma}{4}}}{C(n,\gamma)}\right)^{Q^*}=2^{Q-1}\varrho^{-\frac{Q}{2}},
	$$
	we deduce that
	\begin{equation}\label{eqn:2Jc-varrho-relation}
		(2|J_{\mathcal{C}}|)^{\frac{Q-2\gamma}{2Q}} = \left( 2^Q \varrho^{-\frac{Q}{2}} \right)^{\frac{Q-2\gamma}{2Q}} = 2^{\frac{Q-2\gamma}{2}} \varrho^{-\frac{Q-2\gamma}{4}}.
	\end{equation}
	Combining this with \eqref{eqn:express-Re-zeta}-\eqref{eqn:express-Im-zeta}, we arrive at
	$$
	\begin{aligned}
		(2|J_{\mathcal{C}}|)^{\frac{Q-2\gamma}{2Q}}\operatorname{Re}(\zeta_j)  &= \left( 2 \cdot 2^{Q-1}\varrho^{-\frac{Q}{2}} \right)^{\frac{Q-2\gamma}{2Q}}\cdot\frac{2 \big( x_j(1+|z|^2) - y_j t \big)}{\varrho} \\
		& = \left( 2^{Q}\varrho^{-\frac{Q}{2}} \right)^{\frac{Q-2\gamma}{2Q}}2 \big( x_j(1+|z|^2) -y_j t \big)\varrho^{-1}\\
		& = 2^{\frac{Q-2\gamma+2}{2}}{\varrho^{-\frac{Q-2\gamma+4}{4}}}\big( x_j(1+|z|^2) -y_j t \big)
	\end{aligned}
	$$
	and
	$$
	\begin{aligned}
		(2|J_{\mathcal{C}}|)^{\frac{Q-2\gamma}{2Q}}\operatorname{Im}(\zeta_j)  &= \left( 2 \cdot 2^{Q-1}\varrho^{-\frac{Q}{2}} \right)^{\frac{Q-2\gamma}{2Q}}\cdot\frac{2 \big( y_j(1+|z|^2) +x_j t \big)}{\varrho} \\
		& = \left( 2^{Q}\varrho^{-\frac{Q}{2}} \right)^{\frac{Q-2\gamma}{2Q}}2 \big( y_j(1+|z|^2) +x_j t \big)\varrho^{-1}\\
		& = 2^{\frac{Q-2\gamma+2}{2}}{\varrho^{-\frac{Q-2\gamma+4}{4}}}\big( y_j(1+|z|^2) + x_j t \big).
	\end{aligned}
	$$
	As a result, we establish the following identifications:
	$$
	\begin{aligned}
		X^R_j \omega_{\lambda,\xi}\big|_{(1,0)}&=-(Q-2\gamma)C(n,\gamma)\varrho^{-\frac{Q-2\gamma+4}{4}}\Big(x_j(1+|z|^2)-y_j t\Big)\\
		&=\widetilde{C}(n,\gamma)(2|J_{\mathcal{C}}|)^{\frac{Q-2\gamma}{2Q}}\operatorname{Re}(\zeta_j)
	\end{aligned}
	$$
	and
	$$
	\begin{aligned}
		Y^R_j \omega_{\lambda,\xi}\big|_{(1,0)}&=-(Q-2\gamma)C(n,\gamma)\varrho^{-\frac{Q-2\gamma+4}{4}}\Big(y_j(1+|z|^2)+x_j t\Big)\\
		&=\widetilde{C}(n,\gamma)(2|J_{\mathcal{C}}|)^{\frac{Q-2\gamma}{2Q}}\operatorname{Im}(\zeta_j),
	\end{aligned}
	$$
	where 
	\begin{equation}\label{eqn:widetilde-C-define}
		\widetilde{C}(n,\gamma):=-\frac{(Q-2\gamma)C(n,\gamma)}{2^{\frac{Q-2\gamma+2}{2}}}.
	\end{equation}
	This verifies that $X^R_j \omega_{\lambda,\xi}\big|_{(1,0)}$ and $Y^R_j \omega_{\lambda,\xi}\big|_{(1,0)}$ indeed solve the linearized equation \eqref{eqn:linear-equation-f}.
	
	For $j=n+1$, the corresponding coordinate is $\zeta_{n+1} =  \frac{1 - |z|^2 + it}{1 + |z|^2 - it}$. Carrying out the analogous expansion, we have
	$$
	\begin{aligned}
		\zeta_{n+1} &= \frac{(1-|z|^2+it)(1+|z|^2+it)}{(1+|z|^2)^2 + t^2} \\
		&= \frac{(1-|z|^2)(1+|z|^2) - t^2 + i(t(1+|z|^2) + t(1-|z|^2))}{\varrho}\\
		& = \frac{1-|z|^4-t^2 + 2it}{\varrho},
	\end{aligned}
	$$
	which means that
	$$
	\operatorname{Re}(\zeta_{n+1}) = \frac{1-|z|^4-t^2}{\varrho}
	$$ 
	and 
	$$
	\operatorname{Im}(\zeta_{n+1}) = \frac{2t}{\varrho}. 
	$$
	Together with \eqref{eqn:2Jc-varrho-relation}, this yields
	\begin{equation}\label{eqn:n+1-Re-expression}
		(2|J_{\mathcal{C}}|)^{\frac{Q-2\gamma}{2Q}} \operatorname{Re}(\zeta_{n+1}) = 2^{\frac{Q-2\gamma}{2}} \varrho^{-\frac{Q-2\gamma+4}{4}} (1 - |z|^4 - t^2)
	\end{equation}
	and 
	$$
	(2|J_{\mathcal{C}}|)^{\frac{Q-2\gamma}{2Q}} \operatorname{Im}(\zeta_{n+1}) = 2^{\frac{Q-2\gamma}{2}} \varrho^{-\frac{Q-2\gamma+4}{4}} (2t)= 2^{\frac{Q-2\gamma+2}{2}} \varrho^{-\frac{Q-2\gamma+4}{4}} t.
	$$
	Taking \eqref{eqn:partial-omega-10-t} into account, one has
	\begin{equation}\label{eqn:partial-omega-lambda-Im-n+1}
		\begin{aligned}
			\frac{\partial \omega_{\lambda,\xi}}{\partial t} \bigg|_{(1,0)}=\frac{\partial \omega_{1,0}}{\partial t}&=-\frac{Q-2\gamma}{2}C(n,\gamma)\varrho^{-\frac{Q-2\gamma+4}{4}}t\\
			&= -\frac{1}{2} \frac{(Q-2\gamma)C(n,\gamma)}{2^{\frac{Q-2\gamma+2}{2}}}(2|J_{\mathcal{C}}|)^{\frac{Q-2\gamma}{2Q}} \operatorname{Im}(\zeta_{n+1})\\
			&=\frac{1}{2}\widetilde{C}(n,\gamma)(2|J_{\mathcal{C}}|)^{\frac{Q-2\gamma}{2Q}} \operatorname{Im}(\zeta_{n+1}),
		\end{aligned}
	\end{equation}
	where $\widetilde{C}(n,\gamma)$ is the constant introduced in \eqref{eqn:widetilde-C-define}. Finally, we examine the variation with respect to the dilation parameter $\lambda$. Let 
	$$
	\omega_{\lambda, 0}(z, t) := C(n, \gamma) \lambda^{\frac{Q-2\gamma}{2}} \varrho_{\lambda}^{-\frac{Q-2\gamma}{4}}
\ \ 
	\mbox{with} 
	\ \ 
	\varrho_{\lambda} := (1 + \lambda^2 |z|^2)^2 + \lambda^4 t^2.
	$$
	Taking the derivative at $\lambda=1$ yields
	\begin{equation}\label{eqn:partial-omega-10-lambda}
		\begin{aligned}
			\frac{\partial \omega_{\lambda, 0}}{\partial \lambda} \bigg|_{\lambda=1} &= C(n, \gamma) \left[ \frac{Q-2\gamma}{2} \lambda^{\frac{Q-2\gamma}{2}-1} \varrho_{\lambda}^{-\frac{Q-2\gamma}{4}} - \frac{Q-2\gamma}{4} \lambda^{\frac{Q-2\gamma}{2}} \varrho_{\lambda}^{-\frac{Q-2\gamma}{4}-1} \frac{\partial \varrho_{\lambda}}{\partial \lambda} \right] \bigg|_{\lambda=1} \\
			&= \frac{Q-2\gamma}{4} C(n, \gamma) \varrho^{-\frac{Q-2\gamma+4}{4}} \left[ 2\varrho - \frac{\partial \varrho_{\lambda}}{\partial \lambda} \bigg|_{\lambda=1} \right].
		\end{aligned}
	\end{equation}
	Since 
	$$
	\frac{\partial \varrho_{\lambda}}{\partial \lambda} \bigg|_{\lambda=1} =\Big[2(1+\lambda^2|z|^2)\cdot(2\lambda |z|^2)+4\lambda^3 t^2\Big]\bigg|_{\lambda=1}= 4|z|^2(1+|z|^2) + 4t^2,
	$$
	the term $\Big[2\varrho - \frac{\partial \varrho_{\lambda}}{\partial \lambda} \big|_{\lambda=1}\Big] $ in \eqref{eqn:partial-omega-10-lambda} reduces to
	$$
	\begin{aligned}
		2\varrho - (4|z|^2 + 4|z|^4 + 4t^2) &=2(1+|z|^2)^2 +2 t^2- (4|z|^2 + 4|z|^4 + 4t^2) \\
		&= 2 + 4|z|^2 +2 |z|^4 +2 t^2 - (4|z|^2 + 4|z|^4 + 4t^2) \\
		&= 2 - 2|z|^4 - 2t^2=2(1 - |z|^4 - t^2).
	\end{aligned}
	$$
	Combining this with \eqref{eqn:n+1-Re-expression} and \eqref{eqn:partial-omega-10-lambda}, we obtain that
	\begin{equation}\label{eqn:partial-omega-lambda-Re-n+1}
		\begin{aligned}
			\frac{\partial \omega_{\lambda, \xi}}{\partial \lambda} \bigg|_{(1,0)} &= \frac{Q-2\gamma}{4} C(n, \gamma) \varrho^{-\frac{Q-2\gamma+4}{4}} \cdot 2(1 - |z|^4 - t^2)\\
			&=  \frac{Q-2\gamma}{2} C(n, \gamma) \varrho^{-\frac{Q-2\gamma+4}{4}}(1 - |z|^4 - t^2) \\
			& = \frac{(Q-2\gamma)C(n, \gamma)}{2^{\frac{Q-2\gamma+2}{2}}} (2|J_{\mathcal{C}}|)^{\frac{Q-2\gamma}{2Q}} \operatorname{Re}(\zeta_{n+1}) \\
			&=-\widetilde{C}(n,\gamma) (2|J_{\mathcal{C}}|)^{\frac{Q-2\gamma}{2Q}} \operatorname{Re}(\zeta_{n+1}), \\
		\end{aligned}
	\end{equation}
	where $\widetilde{C}(n,\gamma)$ is defined in (\ref{eqn:widetilde-C-define}).
	\eqref{eqn:partial-omega-lambda-Im-n+1} and \eqref{eqn:partial-omega-lambda-Re-n+1} show that $\frac{\partial \omega_{\lambda, \xi}}{\partial t} \big|_{(1,0)}$ and $\frac{\partial \omega_{\lambda, \xi}}{\partial \lambda} \big|_{(1,0)}$ agree with the corresponding eigenspace basis up to nonzero constants and therefore solve  \eqref{eqn:linear-equation-f}.

	\textbf{$(ii)$ The space $\mathcal H_{0,1}$.} This space is the complex conjugate of $\mathcal H_{1,0}$ and is generated by $\bar{\zeta}_1,\ldots,\bar{\zeta}_{n+1}$. 
	For $j=1, \dots, n$, the components of the conjugated transform take the form $\bar{\zeta}_j = \frac{2\bar{z}_j}{1 + |z|^2 +it} $. Expanding into real and imaginary parts, we find that
	$$
	\begin{aligned}
		\bar{\zeta}_j = \frac{2\bar{z}_j}{1+|z|^2+it} &= \frac{2(x_j-iy_j)}{(1+|z|^2) + it}
		= \frac{2(x_j-iy_j) \big( (1+|z|^2) -it \big)}{(1+|z|^2)^2 + t^2} \\
		&= \frac{2 \big(x_j(1+|z|^2) - y_j t \big) - 2i \big( y_j(1+|z|^2) +x_j t \big)}{\varrho},\\
	\end{aligned}
	$$
	which means
	\begin{equation}\label{eqn:express-Re-bar-zeta}
		\operatorname{Re}(\bar{\zeta}_j) = \frac{2 \big( x_j(1+|z|^2) - y_j t \big)}{\varrho} 
	\end{equation}
	and
	\begin{equation}\label{eqn:express-Im-bar-zeta}
		\operatorname{Im}(\bar{\zeta}_j)= -\frac{2 \big( y_j(1+|z|^2) + x_j t \big)}{\varrho}.
	\end{equation}
	Combining this with (\ref{eqn:2Jc-varrho-relation}), we arrive at
	$$
	\begin{aligned}
		(2|J_{\mathcal{C}}|)^{\frac{Q-2\gamma}{2Q}}\text{Re}(\bar{\zeta}_j)  &= \left( 2^{Q}\varrho^{-\frac{Q}{2}} \right)^{\frac{Q-2\gamma}{2Q}}2 \big( x_j(1+|z|^2) - y_j t \big)\varrho^{-1}\\
		& = 2^{\frac{Q-2\gamma+2}{2}}{\varrho^{-\frac{Q-2\gamma+4}{4}}}\big( x_j(1+|z|^2) - y_j t \big)
	\end{aligned}
	$$
	and
	$$
	\begin{aligned}
		(2|J_{\mathcal{C}}|)^{\frac{Q-2\gamma}{2Q}}\text{Im}(\bar{\zeta}_j)  &= \left( 2^{Q}\varrho^{-\frac{Q}{2}} \right)^{\frac{Q-2\gamma}{2Q}}(-2) \big( y_j(1+|z|^2) +x_j t \big)\varrho^{-1}\\
		& = -2^{\frac{Q-2\gamma+2}{2}}{\varrho^{-\frac{Q-2\gamma+4}{4}}}\big( y_j(1+|z|^2) +x_j t \big).
	\end{aligned}
	$$
	Recalling the expressions for $X^R_j \omega_{\lambda,\xi}\big|_{(1,0)}$ and $Y^R_j \omega_{\lambda,\xi}\big|_{(1,0)}$, a direct computation yields
	$$
	\begin{aligned}
		X^R_j \omega_{\lambda,\xi}\big|_{(1,0)}&=-(Q-2\gamma)C(n,\gamma)\varrho^{-\frac{Q-2\gamma}{4}-1}\Big(x_j(1+|z|^2)-y_j t\Big)\\
		&=\widetilde{C}(n,\gamma)(2|J_{\mathcal{C}}|)^{\frac{Q-2\gamma}{2Q}}\text{Re}(\bar{\zeta}_j)
	\end{aligned}
	$$
	and
	$$
	\begin{aligned}
		Y^R_j \omega_{\lambda,\xi}\big|_{(1,0)}&=-(Q-2\gamma)C(n,\gamma)\varrho^{-\frac{Q-2\gamma}{4}-1}\Big(y_j(1+|z|^2)+x_j t\Big)\\
		&=-\widetilde{C}(n,\gamma)(2|J_{\mathcal{C}}|)^{\frac{Q-2\gamma}{2Q}}\text{Im}(\bar{\zeta}_j),
	\end{aligned}
	$$
	where $\widetilde{C}(n,\gamma)$  is given in (\ref{eqn:widetilde-C-define}). 
	Consequently, both $X^R_j \omega_{\lambda,\xi}\big|_{(1,0)}$ and $Y^R_j \omega_{\lambda,\xi}\big|_{(1,0)}$ satisfy  (\ref{eqn:linear-equation-f}).
	
	For $j=n+1$, the conjugated coordinate is $\bar{\zeta}_{n+1} = \frac{1 - |z|^2 -it}{1 + |z|^2 +it}$.  A direct expansion states
	$$
	\begin{aligned}
		\bar{\zeta}_{n+1} = \frac{1-|z|^2-it}{1+|z|^2+it} &= \frac{(1-|z|^2-it)(1+|z|^2-it)}{(1+|z|^2)^2 + t^2}\\
		&= \frac{(1-|z|^2)(1+|z|^2) - t^2 -i(t(1+|z|^2) + t(1-|z|^2))}{\varrho}\\
		& = \frac{1-|z|^4-t^2 - 2it}{\varrho}.
	\end{aligned}
	$$
	As a consequence,
	$$
	\text{Re}(\bar{\zeta}_{n+1}) = \frac{1-|z|^4-t^2}{\varrho}
\ \ 
	\mbox{and}\ \  
	\text{Im}(\bar{\zeta}_{n+1}) =- \frac{2t}{\varrho}. 
	$$
	Together this with (\ref{eqn:2Jc-varrho-relation}), we deduce that 
	\begin{equation}\label{eqn:n+1-Re-expression-bar}
		(2|J_{\mathcal{C}}|)^{\frac{Q-2\gamma}{2Q}} \text{Re}(\bar{\zeta}_{n+1}) = 2^{\frac{Q-2\gamma}{2}} \varrho^{-\frac{Q-2\gamma+4}{4}} (1 - |z|^4 - t^2)
	\end{equation}
	and 
	\begin{equation}\label{eqn:n+1-Im-expression-bar}
		(2|J_{\mathcal{C}}|)^{\frac{Q-2\gamma}{2Q}} \text{Im}(\bar{\zeta}_{n+1}) = 2^{\frac{Q-2\gamma}{2}} \varrho^{-\frac{Q-2\gamma+4}{4}} (-2t)= -2^{\frac{Q-2\gamma+2}{2}} \varrho^{-\frac{Q-2\gamma+4}{4}} t.
	\end{equation}
	Inserting (\ref{eqn:n+1-Im-expression-bar}) into the expression for the $t$ derivative from (\ref{eqn:partial-omega-10-t}) leads to
	$$
	\begin{aligned}
		\frac{\partial \omega_{\lambda,\xi}}{\partial t} \bigg|_{(1,0)}&=-\frac{Q-2\gamma}{2}C(n,\gamma)\varrho^{-\frac{Q-2\gamma+4}{4}}t\\
		&= \frac{1}{2}\frac{(Q-2\gamma)C(n,\gamma)}{2^{\frac{Q-2\gamma+2}{2}}}(2|J_{\mathcal{C}}|)^{\frac{Q-2\gamma}{2Q}} \operatorname{Im}(\bar{\zeta}_{n+1})\\
		&= -\frac{1}{2}\widetilde{C}(n,\gamma)(2|J_{\mathcal{C}}|)^{\frac{Q-2\gamma}{2Q}} \operatorname{Im}(\bar{\zeta}_{n+1}).
	\end{aligned}
	$$
	Similarly, using the variation with respect to the scaling parameter $\lambda$ derived in (\ref{eqn:partial-omega-lambda-Re-n+1}), we directly substitute (\ref{eqn:n+1-Re-expression-bar}) to obtain
	$$
	\begin{aligned}
		\frac{\partial \omega_{\lambda, \xi}}{\partial \lambda} \bigg|_{(1,0)} &= \frac{Q-2\gamma}{2} C(n, \gamma) \varrho^{-\frac{Q-2\gamma+4}{4}}(1 - |z|^4 - t^2) \\
		& = \frac{(Q-2\gamma)C(n, \gamma)}{2^{\frac{Q-2\gamma+2}{2}}} (2|J_{\mathcal{C}}|)^{\frac{Q-2\gamma}{2Q}} \text{Re}(\bar{\zeta}_{n+1}) \\
		&=-\widetilde{C}(n,\gamma) (2|J_{\mathcal{C}}|)^{\frac{Q-2\gamma}{2Q}} \text{Re}(\bar{\zeta}_{n+1}).
	\end{aligned}
	$$
	These calculations show that $\frac{\partial \omega_{\lambda, \xi}}{\partial t} \big|_{(1,0)} $ and $\frac{\partial \omega_{\lambda, \xi}}{\partial \lambda} \big|_{(1,0)}$ are the solutions to (\ref{eqn:linear-equation-f}) associated with $\bar{\zeta}_{n+1}$.

	From $(i)$ and $(ii)$, we see that the $2n+2$ partial derivatives of $\omega_{\lambda, \xi}$ at $(\lambda, \xi) = (1,0)$ span the eigenspace for $\lambda_{1,0}(\gamma)$. This means the kernel of the linearized operator is generated by these derivatives. Therefore, the standard positive solution $\omega_{1,0}$ is non-degenerate. This completes the proof of Theorem \ref{main-result-1}. 
	\\
	\qed

	Motivated by the local-branch uniqueness argument of Frank and
	Lenzmann \cite{Frank-Lenzmann2013}, we establish the local rigidity of the bubble manifold $\mathcal Z$ for (\ref{main-problem}). As shown in Lemma \ref{solution-sca-tran}, the family of standard positive bubbles is 
	$$
	\mathcal{Z} =\Big \{\omega_{\lambda,\xi} : \lambda > 0,\  \xi \in \mathbb{H}^n\Big\},
	$$
	and it is a $(2n + 2)$-dimensional manifold (see \cite[Page 24]{Chtioui-Hajaiej-Yacoub2021}). Based on the kernel computation in Theorem \ref{main-result-2}, we prove that no other positive solutions exist in a sufficiently small neighborhood of $\mathcal{Z}$. We first establish the Fredholm properties and coercivity bounds of $\mathcal{L}_{\omega_{\lambda,\xi}}$.
	
	\begin{lemma}\label{lem:fredholm-operator}
		The linearized operator $\mathcal{L}_{\omega_{\lambda,\xi}}: \mathcal{D}^\gamma(\mathbb{H}^n) \to \mathcal{D}^{-\gamma}(\mathbb{H}^n)$ is a Fredholm operator of index zero. Furthermore, there exists a constant $\widetilde{C} > 0$ depending only on $n$ and $\gamma$ such that 
		$$
		\Vert{}\mathcal{L}_{\omega_{\lambda,\xi}}v\Vert{}_{\mathcal{D}^{-\gamma}(\mathbb{H}^n)} \ge \widetilde{C} \Vert{}v\Vert{}_{\mathcal{D}^\gamma(\mathbb{H}^n)}
		$$
		for any $v \in (T_{\omega_{\lambda,\xi}}\mathcal{Z})^\perp$.
	\end{lemma}
	
	\begin{proof}
		Based on the non-degeneracy established in Theorem \ref{main-result-2}, we study the linearized operator $\mathcal{L}_{\omega_{\lambda,\xi}}$  by using the perturbation methods introduced in \cite{Chen-Wang2017}. To overcome the difficulty that the fractional operator lacks an explicit pointwise expression, we rewrite the operator as 
		$$
		\mathcal{L}_{\omega_{\lambda,\xi}} = \mathcal{L}_\gamma(I - \mathcal{K}_{\omega_{\lambda,\xi}}),
		$$
		where
		$$
		\mathcal{K}_{\omega_{\lambda,\xi}} := \mathcal{L}_\gamma^{-1} \mathcal M_{V_{\omega_{\lambda,\xi}}},
		$$
		and $\mathcal M_{V_{\omega_{\lambda,\xi}}}$ is the multiplication operator with the potential 
		\begin{equation}\label{eqn:V-omega-define}
			V_{\omega_{\lambda,\xi}} := (Q^*-1)\omega_{\lambda,\xi}^{\frac{4\gamma}{Q-2\gamma}}.
		\end{equation}
		Since the exponent is critical, the bubble solution $\omega_{\lambda,\xi}$ belongs to $L^{\frac{2Q}{Q-2\gamma}}(\mathbb{H}^n)$.  Together this with the identity
		$$
		\int_{\mathbb{H}^n} \vert{}V_{\omega_{\lambda,\xi}}\vert{}^{\frac{Q}{2\gamma}} d\eta = (Q^*-1)^{\frac{Q}{2\gamma}} \int_{\mathbb{H}^n} \left(\omega_{\lambda,\xi}^{\frac{4\gamma}{Q-2\gamma}}\right)^{\frac{Q}{2\gamma}} d\eta = (Q^*-1)^{\frac{Q}{2\gamma}} \int_{\mathbb{H}^n} \omega_{\lambda,\xi}^{\frac{2Q}{Q-2\gamma}} d\eta,
		$$
		which implies that $V_{\omega_{\lambda,\xi}}\in L^{\frac{Q}{2\gamma}}(\mathbb{H}^n)$.  Next, we claim that the operator $\mathcal{K}_{\omega_{\lambda,\xi}}$ is a compact operator in the fractional Sobolev space $\mathcal{D}^\gamma(\mathbb{H}^n)$. To rigorously prove this claim, we proceed in two steps: we first verify that $\mathcal{K}_{\omega_{\lambda,\xi}}$ is a bounded linear operator in $\mathcal{D}^\gamma(\mathbb{H}^n)$, and then show its compactness by exploiting the decay properties of the bubble solutions. 
		
		\textbf{Step 1: Boundedness of $\mathcal{K}_{\omega_{\lambda,\xi}}$ in $\mathcal{D}^\gamma(\mathbb{H}^n)$.}
		
		According to the definition of the fractional Sobolev space over $\mathbb{H}^n$, the squared norm of any function $u \in \mathcal{D}^\gamma(\mathbb{H}^n)$ is equivalently given by the quadratic form (\ref{eqn:a-gamma-u-define}):
		$$
		a_\gamma[u] = \int_{\mathbb{H}^n} {u}\mathcal{L}_\gamma u  d\eta.
		$$
		To evaluate the boundedness of $\mathcal{K}_{\omega_{\lambda,\xi}}$ in $\mathcal{D}^\gamma(\mathbb{H}^n)$, we compute the corresponding quadratic form of $\mathcal{K}_{\omega_{\lambda,\xi}} u$. By definition of $\mathcal{K}_{\omega_{\lambda,\xi}}$,  we have
		$$
		\mathcal{K}_{\omega_{\lambda,\xi}} u = \mathcal{L}_\gamma^{-1}(\mathcal{M}_{V_{\omega_{\lambda,\xi}}} u) = \mathcal{L}_\gamma^{-1}(V_{\omega_{\lambda,\xi}} u),
		$$
		which implies 
		\begin{equation}\label{eqn:mathcal-L-Ku-Vu}
			\mathcal{L}_\gamma (\mathcal{K}_{\omega_{\lambda,\xi}} u) = V_{\omega_{\lambda,\xi}} u. 
		\end{equation}
		Substituting this into the quadratic form yields:
		$$
		a_\gamma[\mathcal{K}_{\omega_{\lambda,\xi}} u] = \int_{\mathbb{H}^n} \overline{(\mathcal{K}_{\omega_{\lambda,\xi}} u)} \mathcal{L}_\gamma (\mathcal{K}_{\omega_{\lambda,\xi}} u) d\eta = \int_{\mathbb{H}^n} \overline{(\mathcal{L}_\gamma^{-1}(V_{\omega_{\lambda,\xi}} u))} (V_{\omega_{\lambda,\xi}} u) d\eta.
		$$
		As in the proof of (2.10) in \cite{Chen-Wang2017}, we have
		$$
		\overline{(\mathcal{K}_{\omega_{\lambda,\xi}} u)(\eta)} = C_1({n,\gamma} )\int_{\mathbb{H}^n} \frac{\overline{V_{\omega_{\lambda,\xi}}(\tilde{\eta})u(\tilde{\eta})}}{\vert{}\tilde{\eta}^{-1}{\eta}\vert{}^{Q-2\gamma}}  d\tilde{\eta}.
		$$
		Together with the identity
		\begin{equation}\label{eqn:mathcal-L-V-u-identity}
			\mathcal{L}_\gamma (\mathcal{K}_{\omega_{\lambda,\xi}} u)(\eta) = V_{\omega_{\lambda,\xi}}(\eta)u(\eta),
		\end{equation}
		this gives
		$$
		\begin{aligned}
			a_\gamma[\mathcal{K}_{\omega_{\lambda,\xi}} u] &= \int_{\mathbb{H}^n} \overline{(\mathcal{K}_{\omega_{\lambda,\xi}} u)(\eta)} \cdot \left( \mathcal{L}_\gamma (\mathcal{K}_{\omega_{\lambda,\xi}} u)(\eta) \right)  d\eta\\
			& =\int_{\mathbb{H}^n} \overline{(\mathcal{K}_{\omega_{\lambda,\xi}} u)(\eta)} \cdot (V_{\omega_{\lambda,\xi}}(\eta)u(\eta)) d\eta\\
			&=\int_{\mathbb{H}^n} \left( C_1({n,\gamma}) \int_{\mathbb{H}^n} \frac{\overline{V_{\omega_{\lambda,\xi}}(\tilde{\eta})u(\tilde{\eta})}}{\vert{}\tilde{\eta}^{-1}\eta\vert{}^{Q-2\gamma}} \, d\tilde{\eta}\right) (V_{\omega_{\lambda,\xi}}(\eta)u(\eta)) d\eta\\
			&= C_1({n,\gamma})\int_{\mathbb{H}^n} \int_{\mathbb{H}^n} \frac{\overline{V_{\omega_{\lambda,\xi}}(\tilde{\eta})u(\tilde{\eta})} (V_{\omega_{\lambda,\xi}}(\eta)u(\eta))}{\vert{}\tilde{\eta}^{-1}\eta\vert{}^{Q-2\gamma}}  d\tilde{\eta} d\eta.
		\end{aligned}
		$$
		Setting $d = Q-2\gamma$, the exponent $p$ in Proposition \ref{Frank-Lieb-Theorem} is $ \frac{2Q}{Q+2\gamma}$. Choosing $f = g = V_{\omega_{\lambda,\xi}}u$  and noting that $V_{\omega_{\lambda,\xi}}$ is a real-valued function, we have 
		$$
		a_\gamma[\mathcal{K}_{\omega_{\lambda,\xi}} u] \leq C_2({n,\gamma}) \Vert{}V_{\omega_{\lambda,\xi}} u\Vert{}_{L^{\frac{2Q}{Q+2\gamma}}(\mathbb{H}^n)}^2.
		$$
		Since $V_{\omega_{\lambda,\xi}} \in L^{\frac{Q}{2\gamma}}(\mathbb{H}^n)$, we apply H\"{o}lder's inequality with $\frac{2\gamma}{Q}$ and $\frac{Q-2\gamma}{2Q} $ to obtain
		$$
		\Vert{}V_{\omega_{\lambda,\xi}} u\Vert{}_{L^{\frac{2Q}{Q+2\gamma}}(\mathbb{H}^n)} \leq \Vert{}V_{\omega_{\lambda,\xi}}\Vert{}_{L^{\frac{Q}{2\gamma}}(\mathbb{H}^n)} \Vert{}u\Vert{}_{L^{\frac{2Q}{Q-2\gamma}}(\mathbb{H}^n)}.
		$$
		Furthermore, Proposition \ref{Chen-Wang2017-result} ensures that 
		\begin{equation}\label{eqn:u-LQstar-bound}
			\Vert{}u\Vert{}_{L^{\frac{2Q}{Q-2\gamma}}(\mathbb{H}^n)}^2 \leq C_3(n,\gamma) a_\gamma[u],
		\end{equation}
		thus,
		$$
		\begin{aligned}
			a_\gamma[\mathcal{K}_{\omega_{\lambda,\xi}} u]&\leq C_2(n,\gamma) \Vert{}V_{\omega_{\lambda,\xi}} u\Vert{}_{L^{\frac{2Q}{Q+2\gamma}}(\mathbb{H}^n)}^2
			\leq C_2(n,\gamma)\Vert{}V_{\omega_{\lambda,\xi}}\Vert{}_{L^{\frac{Q}{2\gamma}}(\mathbb{H}^n)}^2 \Vert{}u\Vert{}_{L^{\frac{2Q}{Q-2\gamma}}(\mathbb{H}^n)}^2\\
			&\leq C_2(n,\gamma)C_3(n,\gamma) \Vert{}V_{\omega_{\lambda,\xi}}\Vert{}_{L^{\frac{Q}{2\gamma}}(\mathbb{H}^n)}^2a_\gamma[u].
		\end{aligned}
		$$
		Because $\Vert{}V_{\omega_{\lambda,\xi}}\Vert{}_{L^{\frac{Q}{2\gamma}}(\mathbb{H}^n)} < \infty$, this confirms that $\mathcal{K}_{\omega_{\lambda,\xi}}$ is a bounded linear operator mapping $\mathcal{D}^\gamma(\mathbb{H}^n)$ continuously into itself.
		
		\textbf{Step 2: Compactness of $\mathcal{K}_{\omega_{\lambda,\xi}}$ on $\mathcal{D}^\gamma(\mathbb{H}^n)$.}
		
		For this purpose, we consider a sequence $u_m \rightharpoonup 0$ in $\mathcal{D}^\gamma(\mathbb{H}^n)$. By the boundedness of $\{u_m\}$ in $\mathcal{D}^\gamma(\mathbb{H}^n)$, we have $a_{\gamma}[u_m] \le \widetilde{M}$ for all $m$ and some constant $\widetilde{M} > 0$. We aim to show that $\mathcal{K}_{\omega_{\lambda,\xi}} u_m \to 0$ in $\mathcal{D}^\gamma(\mathbb{H}^n)$, which is equivalent to showing $a_\gamma[\mathcal{K}_{\omega_{\lambda,\xi}} u_m] \to 0$. From Step 1, it suffices to prove that $\Vert{}V_{\omega_{\lambda,\xi}} u_m\Vert{}_{L^{\frac{2Q}{Q+2\gamma}}(\mathbb{H}^n)} \to 0$. The explicit profile of $\omega_0$ in (\ref{eqn:omega-define}) implies that $\omega_{\lambda,\xi}(\eta) \to 0$ as $\vert{}\eta\vert{} \to \infty$. Consequently,  $V_{\omega_{\lambda,\xi}}(\eta) \to 0$ as $\vert{}\eta\vert{} \to \infty$. Thus, for any small $\varepsilon > 0$, we can choose a sufficiently large $R > 0$ such that  
		$$
		\Vert{}V_{\omega_{\lambda,\xi}}\Vert{}_{L^{\frac{Q}{2\gamma}}(\mathbb{H}^n \setminus B_d(0,R))} < \varepsilon.
		$$
		Combining this with (\ref{eqn:u-LQstar-bound}) and H\"{o}lder's inequality, we deduce that
		\begin{equation}\label{eqn:Hn-exter-estimate}
			\begin{aligned}
				\Vert{}V_{\omega_{\lambda,\xi}} u_m\Vert{}_{L^{\frac{2Q}{Q+2\gamma}}(\mathbb{H}^n \setminus B_d(0,R))} &\leq \Vert{}V_{\omega_{\lambda,\xi}}\Vert{}_{L^{\frac{Q}{2\gamma}}(\mathbb{H}^n \setminus B_d(0,R))} \Vert{}u_m\Vert{}_{L^{\frac{2Q}{Q-2\gamma}}(\mathbb{H}^n)}\\
				& \leq \sqrt{C_3(n,\gamma)a_\gamma[u_m] } \ \varepsilon.
			\end{aligned}
		\end{equation}
		For the interior domain $B_d(0,R)$, we note that the standard bubble $\omega_0(z,t)$ attains its unique global maximum since its denominator satisfies $(1+\vert{}z\vert{}^2)^2 + t^2 \geq 1$. Consequently, the translated and rescaled bubble $\omega_{\lambda,\xi}$ is smooth and globally bounded by $C(n,\gamma)\lambda^{\frac{Q-2\gamma}{2}}$. This guarantees that $V_{\omega_{\lambda,\xi}}$ has no singularities and is uniformly bounded in $B_d(0,R)$. Therefore, 
		$$
		\begin{aligned}
			\Vert{}V_{\omega_{\lambda,\xi}} u_m\Vert{}_{L^{\frac{2Q}{Q+2\gamma}}(B_d(0,R))} &\leq \Vert{}V_{\omega_{\lambda,\xi}}\Vert{}_{L^\infty(B_d(0,R))} \Vert{}u_m\Vert{}_{L^{\frac{2Q}{Q+2\gamma}}(B_d(0,R))} \\
			&\leq (Q^*-1)C(n,\gamma)^\frac{4\gamma}{Q-2\gamma}\lambda^{2\gamma} \Vert{}u_m\Vert{}_{L^{\frac{2Q}{Q+2\gamma}}(B_d(0,R))}.
		\end{aligned}
		$$
		Since $1<\frac{2Q}{Q+2\gamma} < \frac{2Q}{Q-2\gamma}$, the locally compact embedding (see \cite[page 14]{Guidi-Maalaoui-Martino2018}) implies that $u_m \to 0$ in $L^{\frac{2Q}{Q+2\gamma}}(B_d(0, R))$, which means that 
		$$
		\Vert{}V_{\omega_{\lambda,\xi}} u_m\Vert{}_{L^{\frac{2Q}{Q+2\gamma}}(B_d(0,R))} \to 0
		$$
		as $m \to \infty$.  Combining this with (\ref{eqn:Hn-exter-estimate}), we arrive at
		$$
		\begin{aligned}
			\limsup_{m \to \infty}\Vert{}V_{\omega_{\lambda,\xi}} u_m\Vert{}_{L^{\frac{2Q}{Q+2\gamma}}(\mathbb{H}^n)}& \leq \limsup_{m \to \infty}\Vert{}V_{\omega_{\lambda,\xi}} u_m\Vert{}_{L^{\frac{2Q}{Q+2\gamma}}(B_d(0,R))}\\
			&\ \  + \limsup_{m \to \infty}\Vert{}V_{\omega_{\lambda,\xi}} u_m\Vert{}_{L^{\frac{2Q}{Q+2\gamma}}(\mathbb{H}^n \setminus B_d(0,R))}\\
			&\leq  C'_3(n,\gamma) \varepsilon.
		\end{aligned}
		$$
		Letting $\varepsilon\to 0$, we have $\Vert{}V_{\omega_{\lambda,\xi}} u_m\Vert{}_{L^{\frac{2Q}{Q+2\gamma}}(\mathbb{H}^n)} \to 0$,
		which implies that $a_\gamma[\mathcal{K}_{\omega_{\lambda,\xi}} u_m] \to 0$ as $m\to\infty$. This proves $\mathcal{K}_{\omega_{\lambda,\xi}}$ is a compact operator in $\mathcal{D}^\gamma(\mathbb{H}^n)$.
		
		As an isometric isomorphism from $\mathcal{D}^\gamma(\mathbb{H}^n)$ to $\mathcal{D}^{-\gamma}(\mathbb{H}^n)$, the operator $\mathcal{L}_\gamma$ is a Fredholm operator of index zero. Since $\mathcal{K}_{\omega_{\lambda,\xi}}$ acts as a compact operator on $\mathcal{D}^\gamma(\mathbb{H}^n)$, the classical Riesz-Fredholm theory in \cite{Martensen1987} guarantees that $I- \mathcal{K}_{\omega_{\lambda,\xi}}$ is also a Fredholm operator of index zero. By the composition property of Fredholm operators and $\mathcal{L}_{\omega_{\lambda,\xi}} = \mathcal{L}_\gamma(I - \mathcal{K}_{\omega_{\lambda,\xi}})$, we know that $\mathcal{L}_{\omega_{\lambda,\xi}}$ is also a Fredholm operator of index zero. Furthermore, the injectivity of $\mathcal{L}_\gamma$ implies that its null space is trivial. Consequently, $\mathcal{L}_{\omega_{\lambda,\xi}} u = \mathcal{L}_\gamma[(I - \mathcal{K}_{\omega_{\lambda,\xi}})u] = 0$ holds if and only if $(I - \mathcal{K}_{\omega_{\lambda,\xi}})u = 0$, which ensures that the kernel of $\mathcal{L}_{\omega_{\lambda,\xi}}$ is completely determined by $I - \mathcal{K}_{\omega_{\lambda,\xi}}$. From Theorem \ref{main-result-2}, the kernel of the linearized operator $\mathcal{L}_{\omega_{\lambda,\xi}}$ precisely coincides with the tangent space of the manifold $\mathcal{Z}$ at $\omega_{\lambda,\xi}$, i.e., $\ker \mathcal{L}_{\omega_{\lambda,\xi}} = T_{\omega_{\lambda,\xi}} \mathcal{Z}$. Recall that $\mathcal{D}^\gamma(\mathbb{H}^n)$
		admits a direct orthogonal decomposition with respect to the tangent space $T_{\omega_{\lambda,\xi}}\mathcal{Z}$ and its orthogonal complement is 
		$$
		\mathcal{D}^\gamma(\mathbb{H}^n) = T_{\omega_{\lambda,\xi}}\mathcal{Z} \oplus (T_{\omega_{\lambda,\xi}}\mathcal{Z})^\perp.
		$$
		This implies that any $u \in \mathcal{D}^\gamma(\mathbb{H}^n)$ can be uniquely expressed as $u = w + v$, where $w \in T_{\omega_{\lambda,\xi}}\mathcal{Z}$ and $v \in (T_{\omega_{\lambda,\xi}}\mathcal{Z})^\perp$. 
		By restricting the operator $I - \mathcal{K}_{\omega_{\lambda,\xi}}$ to the orthogonal complement $(T_{\omega_{\lambda,\xi}}\mathcal{Z})^\perp$, we effectively eliminate its null space. Since $I - \mathcal{K}_{\omega_{\lambda,\xi}}$ is a Fredholm operator of index zero, the Fredholm alternative theorem ensures that this restricted operator possesses a bounded inverse. For any $f\in \mathcal D^\gamma(\mathbb H^n)$, define
		\begin{equation}\label{eqn:f-lambda-xi-define}
			f_{\lambda,\xi}(\eta):=
			\lambda^{\frac{Q-2\gamma}{2}}
			f\bigl(\delta_\lambda(\xi^{-1}\eta)\bigr)
		\end{equation}
		and write $\mathcal U_{\lambda,\xi}f:=f_{\lambda,\xi}$. By the translation invariance and homogeneity of $\mathcal L_\gamma$, we have
		\begin{equation}\label{eqn:mathcal L-flambdaxi-expression}
			\mathcal L_\gamma f_{\lambda,\xi}(\eta)=\lambda^{\frac{Q+2\gamma}{2}}
			(\mathcal L_\gamma f)\bigl(\delta_\lambda(\xi^{-1}\eta)\bigr).
		\end{equation}
		Therefore, we obtain that
		$$
		\begin{aligned}
			\|f_{\lambda,\xi}\|_{\mathcal D^\gamma(\mathbb H^n)}^2&=\int_{\mathbb H^n}
			f_{\lambda,\xi}(\eta)\mathcal L_\gamma f_{\lambda,\xi}(\eta)d\eta\\
			&=\int_{\mathbb H^n}\lambda^{\frac{Q-2\gamma}{2}}f\bigl(\delta_\lambda(\xi^{-1}\eta)\bigr)
			\lambda^{\frac{Q+2\gamma}{2}}(\mathcal L_\gamma f)\bigl(\delta_\lambda(\xi^{-1}\eta)\bigr)d\eta\\
			&=\int_{\mathbb H^n}\lambda^{\frac{Q-2\gamma}{2}}\cdot \lambda^{\frac{Q+2\gamma}{2}}\cdot \lambda^{-Q} f(\eta)
			\mathcal L_\gamma f(\eta)d\eta\\
			&=\int_{\mathbb H^n}f(\eta)\mathcal L_\gamma f(\eta)d\eta=\|f\|_{\mathcal D^\gamma(\mathbb H^n)}^2.
		\end{aligned}
		$$
		Since $\mathcal U_{\lambda,\xi}$ is also a linear bijection of $\mathcal D^\gamma(\mathbb H^n)$, it is a unitary transformation. On one hand, for any \(f\in\mathcal D^\gamma(\mathbb H^n)\), applying (\ref{eqn:mathcal-L-Ku-Vu}) and (\ref{eqn:mathcal L-flambdaxi-expression}) gives
		$$
		\begin{aligned}
			\mathcal L_\gamma\left[(\mathcal K_{\omega_{1,0}}f)_{\lambda,\xi}\right](\eta)
			&=\lambda^{\frac{Q+2\gamma}{2}}\left[\mathcal L_\gamma(\mathcal K_{\omega_{1,0}}f)\right]\bigl(\delta_\lambda(\xi^{-1}\eta)\bigr)\\
			&=\lambda^{\frac{Q+2\gamma}{2}}V_{\omega_{1,0}}\bigl(\delta_\lambda(\xi^{-1}\eta)\bigr)f\bigl(\delta_\lambda(\xi^{-1}\eta)\bigr).
		\end{aligned}
		$$
		On the other hand, it follows from (\ref{eqn:w-lambda-xi-define}) and (\ref{eqn:V-omega-define}) that
		$$
		\omega_{\lambda,\xi}(\eta)=
		\lambda^{\frac{Q-2\gamma}{2}}
		\omega_{1,0}\bigl(\delta_\lambda(\xi^{-1}\eta)\bigr)
		$$
		and
		$$
			V_{\omega_{\lambda,\xi}}(\eta)=(Q^*-1)\omega_{\lambda,\xi}^{\frac{4\gamma}{Q-2\gamma}}(\eta)
			=(Q^*-1)\lambda^{2\gamma}\omega_{1,0}^{\frac{4\gamma}{Q-2\gamma}}
			\bigl(\delta_\lambda(\xi^{-1}\eta)\bigr)
			\lambda^{2\gamma}V_{\omega_{1,0}}\bigl(\delta_\lambda(\xi^{-1}\eta)\bigr).
		$$
		Hence, 
		$$
		\begin{aligned}
			V_{\omega_{\lambda,\xi}}(\eta)f_{\lambda,\xi}(\eta)
			&=\lambda^{2\gamma}V_{\omega_{1,0}}\bigl(\delta_\lambda(\xi^{-1}\eta)\bigr)
			\lambda^{\frac{Q-2\gamma}{2}}f\bigl(\delta_\lambda(\xi^{-1}\eta)\bigr)\\
			&=\lambda^{\frac{Q+2\gamma}{2}}V_{\omega_{1,0}}\bigl(\delta_\lambda(\xi^{-1}\eta)\bigr)f\bigl(\delta_\lambda(\xi^{-1}\eta)\bigr).
		\end{aligned}
		$$
		Consequently,
		$$
		\mathcal L_\gamma\left[(\mathcal K_{\omega_{1,0}}f)_{\lambda,\xi}\right]
		=V_{\omega_{\lambda,\xi}}f_{\lambda,\xi}.
		$$
		By (\ref{eqn:mathcal-L-Ku-Vu}), we also have
		$$
		\mathcal L_\gamma\left(\mathcal K_{\omega_{\lambda,\xi}}f_{\lambda,\xi}\right)
		=V_{\omega_{\lambda,\xi}}f_{\lambda,\xi}.
		$$
		Invoking the fact that $\mathcal L_\gamma$ is injective yields
		$$
		\mathcal K_{\omega_{\lambda,\xi}}f_{\lambda,\xi}=(\mathcal K_{\omega_{1,0}}f)_{\lambda,\xi}.
		$$
		By the definition of $\mathcal  U_{\lambda,\xi}$, we have 
		$$
		f_{\lambda,\xi}=\mathcal U_{\lambda,\xi}f
		\ \ 
		\mbox{and}
		\ \ 
		(\mathcal K_{\omega_{1,0}}f)_{\lambda,\xi}=\mathcal U_{\lambda,\xi}\mathcal K_{\omega_{1,0}}f.
		$$ 
		Hence,
		$$
		\mathcal K_{\omega_{\lambda,\xi}}\mathcal U_{\lambda,\xi}f=\mathcal K_{\omega_{\lambda,\xi}}f_{\lambda,\xi}=(\mathcal K_{\omega_{1,0}}f)_{\lambda,\xi}=\mathcal  U_{\lambda,\xi}\mathcal K_{\omega_{1,0}}f
		$$
		for any $f\in\mathcal D^\gamma(\mathbb H^n)$, which implies that
		$$
		\mathcal K_{\omega_{\lambda,\xi}}\mathcal U_{\lambda,\xi}=\mathcal U_{\lambda,\xi}\mathcal K_{\omega_{1,0}}.
		$$
		Since \(\mathcal U_{\lambda,\xi}\) is unitary, it is invertible, and hence
		\begin{equation}\label{eqn:Kgeneral-K10-mathcalU-rela}
			\mathcal K_{\omega_{\lambda,\xi}}=\mathcal U_{\lambda,\xi}\mathcal K_{\omega_{1,0}}\mathcal U_{\lambda,\xi}^{-1}.
		\end{equation}
		Thus, $\mathcal K_{\omega_{\lambda,\xi}}$ and $\mathcal K_{\omega_{1,0}}$ are unitarily equivalent, and there exists a constant $M :=M(n,\gamma)> 0$ such that 
		\begin{equation}\label{eqn:I-mathcalK-inverse-bounded}
			\Vert{}[(I - \mathcal{K}_{\omega_{\lambda,\xi}})|_{(T_{\omega_{\lambda,\xi}}\mathcal{Z})^\perp}]^{-1} f\big\Vert{}_{\mathcal{D}^\gamma(\mathbb{H}^n)} \le M\Vert{}f\Vert{}_{\mathcal{D}^\gamma(\mathbb{H}^n)}.
		\end{equation}
		Substituting $f = (I - \mathcal{K}_{\omega_{\lambda,\xi}})v$ into (\ref{eqn:I-mathcalK-inverse-bounded}) yields
		$$
		\Vert{}v\Vert{}_{\mathcal{D}^\gamma(\mathbb{H}^n)} \le M\Vert{}(I - \mathcal{K}_{\omega_{\lambda,\xi}})v\Vert{}_{\mathcal{D}^\gamma(\mathbb{H}^n)}.
		$$
		Setting $C_4(n,\gamma) := \frac{1}{M} > 0$, we obtain that
		$$
		\Vert{}(I - \mathcal{K}_{\omega_{\lambda,\xi}})v\Vert{}_{\mathcal{D}^\gamma(\mathbb{H}^n)} \ge C_4(n,\gamma)  \Vert{}v\Vert{}_{\mathcal{D}^\gamma(\mathbb{H}^n)} 
		$$
		for any $v \in (T_{\omega_{\lambda,\xi}}\mathcal{Z})^\perp$. Recall that $\mathcal{L}_\gamma$ is an isometric isomorphism from $\mathcal{D}^\gamma(\mathbb{H}^n)$ to $\mathcal{D}^{-\gamma}(\mathbb{H}^n)$, which implies that the $\mathcal{D}^{-\gamma}(\mathbb{H}^n)$-norm of $\mathcal{L}_\gamma u$ is equivalent to the $\mathcal{D}^\gamma(\mathbb{H}^n)$-norm of $u$ for any $u\in \mathcal{D}^\gamma(\mathbb{H}^n)$. Combining this norm equivalence with $\mathcal{L}_{\omega_{\lambda,\xi}} = \mathcal{L}_\gamma(I - \mathcal{K}_{\omega_{\lambda,\xi}})$, we deduce that for any $v \in (T_{\omega_{\lambda,\xi}}\mathcal{Z})^\perp$,
		$$
		\begin{aligned}
			\Vert{}\mathcal{L}_{\omega_{\lambda,\xi}} v\Vert{}_{\mathcal{D}^{-\gamma}(\mathbb{H}^n)}& = \Vert{}\mathcal{L}_\gamma \big[ (I - \mathcal{K}_{\omega_{\lambda,\xi}})v \big] \Vert{}_{\mathcal{D}^{-\gamma}(\mathbb{H}^n)} 
			\geq C_5(n,\gamma)  \Vert{}(I - \mathcal{K}_{\omega_{\lambda,\xi}})v \Vert{}_{\mathcal{D}^\gamma(\mathbb{H}^n)}\\
			& \geq C_4(n,\gamma)  C_5(n,\gamma) \Vert{}v\Vert{}_{\mathcal{D}^\gamma(\mathbb{H}^n)}=:\widetilde{C} \Vert{}v\Vert{}_{\mathcal{D}^\gamma(\mathbb{H}^n)},
		\end{aligned}
		$$
		where $\widetilde{C} > 0$ is a constant depending only on $n$ and $\gamma$. This completes the proof.	
	\end{proof}
	
	The following proposition provides a local parametrization of the bubble manifold $\mathcal Z$. It corresponds to the single-bubble and
	fixed-amplitude case in \cite[Proposition~2.2]{Chtioui-Hajaiej-Yacoub2021}, where the unique minimization and the associated orthogonal decomposition are established for fractional CR bubbles on $\mathbb S^{2n+1}$. Through the Cayley correspondence, the corresponding problem is
	transferred from $H^\gamma(\mathbb S^{2n+1})$ to $\mathcal D^\gamma(\mathbb H^n)$. The translation and dilation invariance of the bubble family then yields a uniform neighborhood along
	$\mathcal Z$.
	
	\begin{proposition}\label{lem:infimum-decomposition}
		There exists $\epsilon_0>0$ sufficiently small such that if $u\in \mathcal D^\gamma(\mathbb H^n)$ satisfies
		$$
		\inf_{\omega\in\mathcal{Z}}\|u-\omega\|_{\mathcal{D}^\gamma(\mathbb{H}^n)}<\epsilon_0,
		$$
		then there exists a unique $\omega_u\in\mathcal{Z}$ such that
		$$
		\|u-\omega_u\|_{\mathcal{D}^\gamma(\mathbb{H}^n)}=\inf_{\omega\in\mathcal{Z}}\|u-\omega\|_{\mathcal{D}^\gamma(\mathbb{H}^n)}.
		$$
		Furthermore, if we denote $v:=u-\omega_u$, then for every $\varphi\in T_{\omega_u}\mathcal{Z}$, $v$ satisfies
		$$
		(v,\varphi)_{\mathcal{D}^\gamma(\mathbb{H}^n)}=0,
		$$
		where $(\cdot,\cdot)_{\mathcal{D}^\gamma(\mathbb{H}^n)}$ denotes the inner product of $\mathcal{D}^\gamma(\mathbb{H}^n)$ associated to the norm $\|\cdot\|_{\mathcal{D}^\gamma(\mathbb{H}^n)}$.
		
	\end{proposition}

	\textbf{Proof of Theorem \ref{main-result-4}}. We argue by contradiction. Suppose that there exists a sequence of positive solutions $u_k \in \mathcal{D}^\gamma(\mathbb{H}^n) \setminus \mathcal{Z}$ such that 
	$$
	\mathcal{L}_{\gamma}u_k = u_k^{\frac{Q+2\gamma}{Q-2\gamma}}
	$$ 
	and 
	$$
	\operatorname{dist}_{\mathcal D^\gamma(\mathbb H^n)}(u_k,\mathcal Z)
	:=\inf_{\omega \in \mathcal{Z}} \Vert{}u_k - \omega\Vert{}_{\mathcal{D}^\gamma(\mathbb H^n)} \to 0
	$$ 
	as $k \to \infty$. By Proposition \ref{lem:infimum-decomposition}, for sufficiently large $k$, there exists a unique $\omega_k:=\omega_{u_k}\in\mathcal Z$ such that
	$$
	\|u_k-\omega_k\|_{\mathcal{D}^\gamma(\mathbb{H}^n)}=\inf_{\omega\in\mathcal{Z}}\|u_k-\omega\|_{\mathcal{D}^\gamma(\mathbb{H}^n)}=\operatorname{dist}_{\mathcal D^\gamma(\mathbb H^n)}(u_k,\mathcal Z).
	$$
	Moreover, setting $v_k:=u_k-\omega_k$ gives $(v_k,\varphi)_{\mathcal D^\gamma(\mathbb H^n)}=0$ for any $\varphi\in T_{\omega_k}\mathcal Z$,
	and
	\begin{equation}\label{eqn:vk-to-0}
		\|v_k\|_{\mathcal D^\gamma(\mathbb H^n)}=\|u_k-\omega_k\|_{\mathcal D^\gamma(\mathbb H^n)}=\operatorname{dist}_{\mathcal D^\gamma(\mathbb H^n)}(u_k,\mathcal Z)\to 0\ \mbox{as}\ k\to\infty.
	\end{equation}
	Since $u_k\notin\mathcal Z$, we have $v_k\neq0$ and $\|v_k\|_{\mathcal D^\gamma(\mathbb{H}^n)}>0$ for every $k$. Define the functional $\mathcal{F}: \mathcal{D}^\gamma(\mathbb{H}^n) \to \mathcal{D}^{-\gamma}(\mathbb{H}^n)$ by
	$$
	\mathcal{F}(u) := \mathcal{L}_{\gamma} u - (u^+)^{\frac{Q+2\gamma}{Q-2\gamma}}.
	$$
	Since $0<\gamma<\frac{Q}{2}$ and $\frac{Q+2\gamma}{Q-2\gamma}-1=\frac{4\gamma}{Q-2\gamma}>0$, we know that the exponent $\frac{Q+2\gamma}{Q-2\gamma} > 1$. Moreover, at a positive bubble $\omega\in\mathcal Z$, it holds $\mathcal F'(\omega)v=\mathcal L_\gamma v-p\omega^{p-1}v=:\mathcal L_\omega v$ (see \cite[Page: 537]{Chen-Wang2017} or \cite[Page: 9]{Guidi-Maalaoui-Martino2018}). Applying Taylor's expansion to $\mathcal{F}$ around $\omega_k$, we obtain that 
	$$
	0 = \mathcal{F}(\omega_k + v_k) = \mathcal{F}(\omega_k) + \mathcal{F}'(\omega_k)v_k + R_{\omega_k}(v_k),
	$$
	and the Taylor remainder
	\begin{equation}\label{eqn:R-expression-1}
		R_{\omega_k}(v_k)=\mathcal F(\omega_k+v_k)-\mathcal F(\omega_k)-\mathcal F'(\omega_k)v_k
			=-\mathcal F'(\omega_k)v_k=- \mathcal{L}_{\omega_k} v_k,
	\end{equation}
	since $\mathcal{F}(u_k) =\mathcal F(\omega_k+v_k)= 0$, $\mathcal{F}(\omega_k) = 0$ and $\mathcal{F}'(\omega_k) v_k= \mathcal{L}_{\omega_k}v_k$.  Furthermore, $R_{\omega_k}(v_k)$ can also be expressed as
	\begin{equation}\label{eqn:R-expression-2}
		\begin{aligned}
			R_{\omega_k}(v_k)&=\mathcal F(\omega_k+v_k)-\mathcal F(\omega_k)-\mathcal F'(\omega_k)v_k\\
			&=\mathcal L_\gamma(\omega_k+v_k)-((\omega_k+v_k)^+)^{\frac{Q+2\gamma}{Q-2\gamma}}-\bigl(\mathcal L_\gamma\omega_k-\omega_k^{\frac{Q+2\gamma}{Q-2\gamma}}\bigr)-\bigl(\mathcal L_\gamma v_k-(Q^*-1)\omega_k^\frac{4\gamma}{Q-2\gamma}v_k\bigr)\\
			&=-((\omega_k+v_k)^+)^{\frac{Q+2\gamma}{Q-2\gamma}}+\omega_k^{\frac{Q+2\gamma}{Q-2\gamma}}+(Q^*-1)\omega_k^\frac{4\gamma}{Q-2\gamma}v_k.
		\end{aligned}
	\end{equation}
	We next estimate $R_{\omega_k}(v_k)$ uniformly with respect to $\omega_k\in\mathcal Z$. By the standard pointwise Taylor inequalities for the function $t\mapsto(t^+)^m$, there exists a constant $C_m>0$ depending only on $m$ such that for every $a\geq0$ and $b\in\mathbb R$,
	\begin{equation}\label{eqn:1-m-2-inequality}
		\left|(a+b)_+^m-a^m-pa^{m-1}b\right|\leq C_m|b|^m
	\end{equation}
	when $1<m\leq2$, see \cite[Lemma 2.3]{Kim-Merle2025} for details. Let $p:=Q^*-1=\frac{Q+2\gamma}{Q-2\gamma}$, $q:=Q^*=\frac{2Q}{Q-2\gamma}$ and $q':=\frac{2Q}{Q+2\gamma}$. If $1<p\leq 2$, using $L^{q'}(\mathbb H^n)\hookrightarrow\mathcal D^{-\gamma}(\mathbb H^n)$, (\ref{eqn:R-expression-2}), (\ref{eqn:1-m-2-inequality}), $pq'=q$ and $\mathcal D^\gamma(\mathbb H^n)\hookrightarrow L^q(\mathbb H^n)$, we obtain that 
	$$
	\begin{aligned}
		\|R_{\omega_k}(v_k)\|_{\mathcal D^{-\gamma}(\mathbb H^n)}&\leq C_p\|R_{\omega_k}(v_k)\|_{L^{q'}(\mathbb H^n)}\\
		&=C_p\|((\omega_k+v_k)^+)^p-\omega_k^p-p\omega_k^{p-1}v_k\|_{L^{q'}(\mathbb H^n)}\\
		&\leq C'_p\||v_k|^p\|_{L^{q'}(\mathbb H^n)}=C'_p\|v_k\|^p_{L^{q}(\mathbb H^n)}\leq C''_p\|v_k\|_{\mathcal D^\gamma(\mathbb H^n)}^{p},
	\end{aligned}
	$$
	where $C_p,C'_p,C''_p$ are positive constants depending only on $p$. Since $1<p\leq 2$, this gives
	\begin{equation}\label{eqn:R-estimate-1-2}
		\|R_{\omega_k}(v_k)\|_{\mathcal D^{-\gamma}(\mathbb H^n)}=o\!\left(\|v_k\|_{\mathcal D^\gamma(\mathbb H^n)}\right)
	\end{equation}
	as $\|v_k\|_{\mathcal D^\gamma}\to0$.
	
	If $p>2$, setting $g(t):=(t^+)^p$, we have 
	$$
	g'(t)=p(t^+)^{p-1}
\ \ 
	\mbox{and} 
	\ \ 
	g''(t)=p(p-1)(t^+)^{p-2}.
	$$
	For any $a\ge0$ and $b\in\mathbb R$, the second-order Taylor formula with an integral remainder gives that
	$$
	\begin{aligned}
		g(a+b)-g(a)-g'(a)b&=b^2\int_0^1(1-\theta)g''(a+\theta b) d\theta\\
		&=p(p-1)b^2
		\int_0^1(1-\theta)((a+\theta b)^+)^{p-2}d\theta.
	\end{aligned}
	$$
	Since
	$
	(a+\theta b)^+\le|a+\theta b|\le a+|b|
	$
	for \(\theta\in[0,1]\), we derive that
	$$
	((a+\theta b)^+)^{p-2}\le (a+|b|)^{p-2}\le \widetilde{C}_p\left(a^{p-2}+|b|^{p-2}\right)
	$$
	and
	$$
	\begin{aligned}
		\left|g(a+b)-g(a)-g'(a)b\right|&\le \widetilde{C}'_p|b|^2\left(a^{p-2}+|b|^{p-2}\right)\int_0^1(1-\theta)d\theta\\
		&\le \widetilde{C}'_p\left(a^{p-2}|b|^2+|b|^p\right),
	\end{aligned}
	$$
	which means that
	\begin{equation}\label{eqn:p-geq-2-inequality}
		\left|((a+b)^+)^p-a^p-pa^{p-1}b\right|\le \widetilde{C}'_p\left(a^{p-2}|b|^2+|b|^p\right).
	\end{equation}
	By H\"older’s inequality with the relation $\frac{1}{q'}=\frac{p-2}{q}+\frac{2}{q}$, we have
	\begin{equation}\label{eqn:holder-result}
		\|\omega_k^{p-2}|v_k|^2\|_{L^{q'}(\mathbb H^n)}\leq\|\omega_k\|_{L^q(\mathbb H^n)}^{p-2}\|v_k\|_{L^q(\mathbb H^n)}^{2}.
	\end{equation}
	Utilizing $L^{q'}(\mathbb H^n)\hookrightarrow\mathcal D^{-\gamma}(\mathbb H^n)$, (\ref{eqn:R-expression-2}), (\ref{eqn:p-geq-2-inequality}), (\ref{eqn:holder-result}), $pq'=q$ and $\mathcal D^\gamma(\mathbb H^n)\hookrightarrow L^q(\mathbb H^n)$, one has
	$$
	\begin{aligned}
		\|R_{\omega_k}(v_k)\|_{\mathcal D^{-\gamma}}&\leq C_p\|R_{\omega_k}(v_k)\|_{L^{q'}(\mathbb H^n)}\\
		&=C_p\|((\omega_k+v_k)^+)^p-\omega_k^p-p\omega_k^{p-1}v_k\|_{L^{q'}(\mathbb H^n)}\\
		&=C_p \widetilde{C}'_p\||\omega_k^{p-2}|v_k|^2+|v_k|^p|\|_{L^{q'}(\mathbb H^n)}\\
		&\leq C_p \widetilde{C}'_p\left(\|\omega_k^{p-2}|v_k|^2\|_{L^{q'}(\mathbb H^n)}+\||v_k|^p\|_{L^{q'}(\mathbb H^n)}\right)\\
		&\leq C_p \widetilde{C}'_p\left(\|\omega_k\|_{L^q(\mathbb H^n)}^{p-2}\|v_k\|_{L^q(\mathbb H^n)}^{2}+\|v_k\|_{L^q(\mathbb H^n)}^{p}\right)	\\
		&\leq \widetilde{C}''_p\left(\|\omega_k\|_{L^q(\mathbb H^n)}^{p-2}\|v_k\|_{\mathcal D^\gamma(\mathbb H^n)}^{2}+\|v_k\|_{\mathcal D^\gamma(\mathbb H^n)}^{p}\right).
	\end{aligned}
	$$
	Meanwhile, it follows from the fact $\|\omega_{\lambda,\xi}\|_{L^q(\mathbb H^n)}=\|\omega_{1,0}\|_{L^q(\mathbb H^n)}$ for every $\omega_{\lambda,\xi}\in\mathcal Z$ that the constant in the above estimate can be chosen independently of the bubble parameters.  Thus,
	\begin{equation}\label{eqn:R-estimate-2-infty}
		\|R_{\omega_k}(v_k)\|_{\mathcal D^{-\gamma}(\mathbb H^n)}\leq
		\widetilde{C}'''_p\left(\|v_k\|_{\mathcal D^\gamma(\mathbb H^n)}^{2}+\|v_k\|_{\mathcal D^\gamma(\mathbb H^n)}^{p}\right)=o\!\left(\|v_k\|_{\mathcal D^\gamma(\mathbb H^n)}\right)
	\end{equation}
	as $\|v_k\|_{\mathcal D^\gamma(\mathbb H^n)}\to0$.
	
	Combining (\ref{eqn:vk-to-0}), (\ref{eqn:R-estimate-1-2}) and (\ref{eqn:R-estimate-2-infty}), we derive that
	$$
	\|R_{\omega_k}(v_k)\|_{\mathcal D^{-\gamma}(\mathbb H^n)}=o\!\left(\|v_k\|_{\mathcal D^\gamma}\right)
	$$
	as $k\to\infty$, uniformly for $\omega_k\in\mathcal Z$. Since $v_k\in(T_{\omega_k}\mathcal Z)^\perp$, Lemma \ref{lem:fredholm-operator} and $R_{\omega_k}(v_k)=- \mathcal{L}_{\omega_k} v_k$ in (\ref{eqn:R-expression-1}) imply that there exists a constant $C_0>0$ independent of $k$ such that
	$$
	\begin{aligned}
		C_0\|v_k\|_{\mathcal D^\gamma(\mathbb H^n)}&\leq\|\mathcal L_{\omega_k}v_k\|_{\mathcal D^{-\gamma}(\mathbb H^n)}
		=\|R_{\omega_k}(v_k)\|_{\mathcal D^{-\gamma}(\mathbb H^n)}\\
		&\leq
		\begin{cases}
			C''_p\|v_k\|_{\mathcal D^\gamma(\mathbb H^n)}^{p}, & 1<p\le2;\\
			\widetilde{C}'''_p\left(\|v_k\|_{\mathcal D^\gamma(\mathbb H^n)}^{2}+\|v_k\|_{\mathcal D^\gamma(\mathbb H^n)}^{p}\right), & p>2.
		\end{cases}\\
	\end{aligned}
	$$
	Because $v_k\neq0$ and $\|v_k\|_{\mathcal D^\gamma(\mathbb{H}^n)}>0$ for every $k$, dividing both sides of the above inequality by $\|v_k\|_{\mathcal D^\gamma(\mathbb{H}^n)}$ yields 
	$$
	\begin{aligned}
		C_0&\leq
		\begin{cases}
			C''_p\|v_k\|_{\mathcal D^\gamma(\mathbb H^n)}^{p-1}, & 1<p\le2;\\
			\widetilde{C}'''_p\left(\|v_k\|_{\mathcal D^\gamma(\mathbb H^n)}+\|v_k\|_{\mathcal D^\gamma(\mathbb H^n)}^{p-1}\right), & p>2\\
		\end{cases}\\
		&=o(1).
	\end{aligned}
	$$
	Letting $k\to \infty$ and using ($\ref{eqn:vk-to-0}$), we derive that $C_0\leq0$, which contradicts $C_0>0$. Therefore, such a sequence cannot exist. Consequently, there exists $\epsilon>0$ such that any positive solution $u$ of (\ref{main-problem}) satisfying
	$$
	\operatorname{dist}_{\mathcal D^\gamma(\mathbb H^n)}(u,\mathcal Z)<\epsilon
	$$
	must belong to $\mathcal Z$.
	\qed

	\section{Morse Index and Sharp Linearized Estimates}\label{proof of main result-3}
	In this section, we compute the Morse index of the standard positive bubble $\omega_{1,0}$. Using the Cayley correspondence established in Proposition \ref{prop:Cayley-correspondence}, we transfer the quadratic
	form $\mathcal Q_{\omega_{1,0}}$ from the Heisenberg group $\mathbb H^n$ to the CR sphere $\mathbb S^{2n+1}$. The Morse index can then be determined from the spectrum of the intertwining operator $A_\gamma$.
	
	\begin{lemma}\label{lem:4-1}
		For any $f \in \mathcal{D}^\gamma(\mathbb{H}^n)$, let $\tilde{\rho} \in H^\gamma(\mathbb{S}^{2n+1})$ be defined by 
		$$
		\tilde{\rho}(\zeta) := f(\mathcal{C}^{-1}(\zeta)) \Big(\frac{|J_{\mathcal{C}^{-1}}(\zeta)|}{2}\Big)^{\frac{Q-2\gamma}{2Q}},
		$$
		which equivalently yields the inverse relation
		$$
		f(\eta) = (2|J_{\mathcal{C}}(\eta)|)^{\frac{Q-2\gamma}{2Q}} \tilde{\rho}(\mathcal{C}(\eta)).
		$$
		Then,
		$$
		\mathcal{Q}_{\omega_{1,0}}(f, f)  = 2\int_{\mathbb{S}^{2n+1}} \tilde{\rho} (\mathcal{A}_\gamma - \lambda_{1,0}(\gamma)) \tilde{\rho}  d\zeta,
		$$
		where $\mathcal{Q}_{\omega_{1,0}}(f, f)$ is given in (\ref{eqn:Q-w-10-define}), and $\lambda_{1,0}(\gamma)$ is the eigenvalue of $\mathcal{A}_\gamma$ corresponding to the space $\mathcal{H}_{1,0} \oplus \mathcal{H}_{0,1}$. Moreover, the Morse index of $\mathcal L_{\omega_{1,0}}$ on $\mathcal{D}^\gamma(\mathbb{H}^n)$ equals the number of strictly negative eigenvalues of 
		$$
		\mathcal{P}_\gamma = \mathcal{A}_\gamma - \lambda_{1,0}(\gamma)I 
		$$
		on $\mathbb{S}^{2n+1}$,
		where $I$ is the identity operator.
	\end{lemma}
	
	\begin{proof}
		From \eqref{eqn:L-gamma-g-express}, we know that the fractional sub-Laplacian $\mathcal{L}_\gamma$ and the operator $\mathcal{A}_\gamma$ satisfy the following relation
		$$
		\mathcal{L}_\gamma f(\eta) = (2|J_{\mathcal{C}}(\eta)|)^{\frac{Q+2\gamma}{2Q}} (\mathcal{A}_\gamma \tilde{\rho})(\mathcal{C}(\eta)).
		$$
		Multiplying this equation by $f(\eta) = (2|J_{\mathcal{C}}(\eta)|)^{\frac{Q-2\gamma}{2Q}} \tilde{\rho}(\mathcal{C}(\eta))$, integrating over $\mathbb{H}^n$, and using \eqref{eqn:int-F-FC}, we deduce that the first term in the quadratic form $\mathcal{Q}_{\omega_{1,0}}(f, f)$ is
		\begin{equation}\label{eqn:int-f-mathcalL-f}
			\begin{aligned}
				\int_{\mathbb{H}^n} f \mathcal{L}_\gamma f  d\eta 
				&= \int_{\mathbb{H}^n} \left[ (2|J_{\mathcal{C}}(\eta)|)^{\frac{Q-2\gamma}{2Q}} \tilde{\rho}(\mathcal{C}(\eta)) \right] \cdot \left[ (2|J_{\mathcal{C}}(\eta)|)^{\frac{Q+2\gamma}{2Q}} (\mathcal{A}_\gamma \tilde{\rho})(\mathcal{C}(\eta)) \right] d\eta \\
				&= \int_{\mathbb{H}^n} \tilde{\rho}(\mathcal{C}(\eta)) (\mathcal{A}_\gamma \tilde{\rho})(\mathcal{C}(\eta)) (2|J_{\mathcal{C}}(\eta)|) d\eta \\
				&= 2\int_{\mathbb{H}^n} (\tilde{\rho} \cdot (\mathcal{A}_\gamma \tilde{\rho}))(\mathcal{C}(\eta)) |J_{\mathcal{C}}(\eta)| d\eta \\
				&= 2\int_{\mathbb{S}^{2n+1}} \tilde{\rho} \mathcal{A}_\gamma \tilde{\rho}  d\zeta.
			\end{aligned}
		\end{equation}
		For the second integral of $\mathcal{Q}_{\omega_{1,0}}(f, f)$, utilizing the explicit expression for $\omega_{1,0}$
		$$
		\omega_{1,0}(\eta) = \frac{C(n,\gamma)}{((1+|z|^2)^2 + t^2)^\frac{Q-2\gamma}{4}} = \left[ \frac{\Gamma(\frac{Q+2\gamma}{4})}{\Gamma(\frac{Q-2\gamma}{4})} \right]^\frac{Q-2\gamma}{2\gamma}\frac{2^\frac{Q-2\gamma}{2}}{((1+|z|^2)^2 + t^2)^\frac{Q-2\gamma}{4}},
		$$
		and the Jacobian determinant
		$$
		|J_{\mathcal{C}}(\eta)| = \frac{2^{2n+1}}{\left( (1 + |z|^2)^2 + t^2 \right)^{n+1}} = \frac{2^{Q-1}}{\left( (1 + |z|^2)^2 + t^2 \right)^\frac{Q}{2}},
		$$
		we compute that
		$$
		\begin{aligned}
			(Q^\ast - 1)\omega_{1,0}^{\frac{4\gamma}{Q-2\gamma}}(\eta) &= (Q^\ast - 1)  \left[ \frac{\Gamma(\frac{Q+2\gamma}{4})}{\Gamma(\frac{Q-2\gamma}{4})} \right]^{\frac{Q-2\gamma}{2\gamma}\cdot{\frac{4\gamma}{Q-2\gamma}}}	
			\frac{2^{\frac{Q-2\gamma}{2}\cdot {\frac{4\gamma}{Q-2\gamma}}}}{((1+|z|^2)^2 + t^2)^{\frac{Q-2\gamma}{4}\cdot {\frac{4\gamma}{Q-2\gamma}}}}\\
			&= (Q^\ast - 1)  \left[ \frac{\Gamma(\frac{Q+2\gamma}{4})}{\Gamma(\frac{Q-2\gamma}{4})} \right]^{2}	
			\frac{2^{2\gamma}}{((1+|z|^2)^2 + t^2)^{\gamma}}\\
			&= (Q^\ast - 1)  \left[ \frac{\Gamma(\frac{Q+2\gamma}{4})}{\Gamma(\frac{Q-2\gamma}{4})} \right]^{2}	\left(\frac{2^Q}{((1+|z|^2)^2+t^2)^\frac{Q}{2}}\right)^\frac{2\gamma}{Q}\\
			&= (Q^\ast - 1) \left[ \frac{\Gamma(\frac{Q+2\gamma}{4})}{\Gamma(\frac{Q-2\gamma}{4})} \right]^2 (2|J_{\mathcal{C}}(\eta)|)^{\frac{2\gamma}{Q}} \\
			&= \lambda_{1,0}(\gamma)(2|J_{\mathcal{C}}(\eta)|)^{\frac{2\gamma}{Q}}.
		\end{aligned}
		$$
		Hence, 
		$$
		\begin{aligned}
			\int_{\mathbb{H}^n} (Q^\ast - 1)\omega_{1,0}^{\frac{4\gamma}{Q-2\gamma}} f^2d\eta
			&= \int_{\mathbb{H}^n} (Q^\ast - 1)\omega_{1,0}^{\frac{4\gamma}{Q-2\gamma}}  \left[ (2|J_{\mathcal{C}}(\eta)|)^{\frac{Q-2\gamma}{2Q}} \tilde{\rho}(\mathcal{C}(\eta)) \right]^2 d\eta \\
			&= \int_{\mathbb{H}^n} (Q^\ast - 1)\omega_{1,0}^{\frac{4\gamma}{Q-2\gamma}}  (2|J_{\mathcal{C}}(\eta)|)^{\frac{Q-2\gamma}{Q}} \tilde{\rho}^2(\mathcal{C}(\eta)) d\eta \\
			&= \int_{\mathbb{H}^n} \lambda_{1,0}(\gamma)(2|J_{\mathcal{C}}(\eta)|)^{\frac{2\gamma}{Q}} (2|J_{\mathcal{C}}(\eta)|)^{\frac{Q-2\gamma}{Q}} \tilde{\rho}^2(\mathcal{C}(\eta)) d\eta \\
			&= \lambda_{1,0}(\gamma) \int_{\mathbb{H}^n} \tilde{\rho}^2(\mathcal{C}(\eta)) (2|J_{\mathcal{C}}(\eta)|) \, d\eta \\
			&= 2\lambda_{1,0}(\gamma) \int_{\mathbb{S}^{2n+1}} \tilde{\rho}^2 d\zeta.
		\end{aligned}
		$$
		Together this with \eqref{eqn:int-f-mathcalL-f}, we arrive at
		\begin{equation}\label{eqn:quadratic-form-sphere}
			\begin{aligned}
				\mathcal{Q}_{\omega_{1,0}}(f,f) &= \int_{\mathbb{H}^n} f \mathcal{L}_\gamma f d\eta - \int_{\mathbb{H}^n} (Q^* - 1)\omega_{1,0}^{\frac{4\gamma}{Q-2\gamma}} f^2 d\eta \\
				&= 2 \int_{\mathbb{S}^{2n+1}} \tilde{\rho} \mathcal{A}_\gamma \tilde{\rho} \, d\zeta - 2 \lambda_{1,0}(\gamma) \int_{\mathbb{S}^{2n+1}} \tilde{\rho}^2 \, d\zeta \\
				&= 2 \int_{\mathbb{S}^{2n+1}} \tilde{\rho} \big( \mathcal{A}_\gamma \tilde{\rho} - \lambda_{1,0}(\gamma)\tilde{\rho} \big) \, d\zeta \\
				&= 2 \int_{\mathbb{S}^{2n+1}} \tilde{\rho} \big( \mathcal{A}_\gamma - \lambda_{1,0}(\gamma) \big) \tilde{\rho} \, d\zeta.
			\end{aligned}
		\end{equation}
		Combining the linearity and invertibility of the mapping $f \mapsto \tilde{\rho}$ with (\ref{eqn:mathcalS-Hn-norm-eq}) and (\ref{eqn:int-f-mathcalL-f}), we deduce that it is a linear isomorphism from $D^\gamma(\mathbb H^n)$ onto $H^\gamma(\mathbb S^{2n+1})$. Since the factor $2$ in (\ref{eqn:quadratic-form-sphere}) is strictly positive, this correspondence preserves the sign of the quadratic
		form. Therefore, maximal negative subspaces of $\mathcal Q_{\omega_{1,0}}$ correspond bijectively to maximal negative subspaces of the quadratic form associated with $\mathcal P_\gamma$.
		Consequently,
		$$
		\begin{aligned}
			m(\omega_{1,0})=\operatorname{ind}^{-}(\mathcal P_\gamma),
		\end{aligned}
		$$
		where
		$$
		\operatorname{ind}^{-}(P_\gamma):=
		\max\left\{\dim E:
		\begin{array}{l}
			E\subset H^\gamma(\mathbb S^{2n+1})\text{ is a linear subspace},\ \int_{\mathbb S^{2n+1}}
			U \mathcal P_\gamma Ud\zeta<0\\[1mm]
			\text{for every }U\in E\setminus\{0\}.
		\end{array}
		\right\}.
		$$ The proof is now complete.
	\end{proof}  
	
	\begin{lemma}\label{lem:4-2}
		Let the eigenvalues of the intertwining operator $\mathcal{A}_\gamma$ on the spherical harmonic space $\mathcal{H}_{j,k}$ be given by $\lambda_{j,k}(\gamma) = \lambda_j(\gamma)\lambda_k(\gamma)$, where
		$$
		\lambda_j(\gamma) = \frac{\Gamma(\frac{Q+2\gamma}{4}+j)}{\Gamma(\frac{Q-2\gamma}{4}+j)}, \quad j \in \{0, 1, 2, \dots\}.
		$$
		Then, the sequence $\{\lambda_j(\gamma)\}_{j=0}^\infty$ is strictly increasing. Consequently, 
		$$
		\lambda_{0,0}(\gamma)<\lambda_{1,0}(\gamma),
		$$
		and
		\begin{equation}\label{eqn:lambda-jk-10-geq}
			\lambda_{j,k}(\gamma)\geq\lambda_{1,0}(\gamma)\ \ \text{for}\ \ (j,k)\neq(0,0).
		\end{equation}
		Moreover, the equality holds in (\ref{eqn:lambda-jk-10-geq}) if and only if
		$$
		(j,k)=(1,0)\ \text{or}\ \ (j,k)=(0,1).
		$$
	\end{lemma}
	
	\begin{proof}
		Consider the continuous function $\phi(x) := \frac{\Gamma(x+\gamma)}{\Gamma(x)}$ for $x > 0$. To determine its monotonicity, we compute the logarithmic derivative
		$$
		\begin{aligned}
			\frac{d}{dx} \ln \phi(x) &= \frac{d}{dx} \Big(\ln \Gamma(x+\gamma) - \ln \Gamma(x)\Big) \\
			&=  \psi(x+\gamma) - \psi(x),
		\end{aligned}
		$$
		where $\psi(x) := \frac{\Gamma'(x)}{\Gamma(x)}$ is the digamma function, as defined in \cite[(6.3.1)]{Abramowitz-Stegun-1948}. Recall the well-known series expansion for the trigamma function $\psi'(x)$ given by \cite[(6.4.10)]{Abramowitz-Stegun-1948}
		$$
		\psi'(x) = \sum_{k=0}^\infty (x+k)^{-2}.
		$$
		Because this series is strictly positive for all $x > 0$, the digamma function $\psi(x)$ is strictly increasing. Given $\gamma > 0$, it follows that $\psi(x+\gamma) - \psi(x) > 0$, rendering $\phi(x)$ is strictly increasing with respect to $x$. 
		
		Let $x_j := \frac{Q-2\gamma}{4} + j$, clearly,  $x_j > 0$ for all $j \ge 0$. As $x_j$ strictly increases with the index $j$, we deduce the strict inequalities
		$$
		\phi(x_0) <  \phi(x_1) <  \phi(x_2) <\cdots,
		$$
		which equivalently yields
		$$
		\lambda_0(\gamma) < \lambda_1(\gamma) < \lambda_2(\gamma) < \cdots.
		$$
		Since $\lambda_{j,k}(\gamma) = \lambda_j(\gamma)\lambda_k(\gamma)$ and each $\lambda_j(\gamma)>0$ (see Branson et al. \cite[Page 17]{Branson-Fontana-Morpurgo2013}), the smallest eigenvalue is uniquely taken at  $j = k = 0$, i.e. $\lambda_{0,0}(\gamma) < \lambda_{1,0}(\gamma) = \lambda_{0,1}(\gamma)$. Furthermore, the strict positivity and monotonicity of $\{\lambda_j(\gamma)\}_{j=0}^\infty$ guarantee that any other pair of indices $(j,k) \neq (0,0)$ results in $\lambda_{j,k}(\gamma) \ge \lambda_{1,0}(\gamma)$. Thus, $\lambda_{0,0}(\gamma)$ is indeed the unique eigenvalue strictly smaller than $\lambda_{1,0}(\gamma)$.
	\end{proof}

	\textbf{Proof of Theorem \ref{main-result-3}}. By Lemma \ref{lem:4-1}, evaluating the Morse index of $\omega_{1,0}$ on $\mathcal{D}^\gamma(\mathbb{H}^n)$ is completely equivalent to finding the number of strictly negative eigenvalues (counted with their geometric multiplicity) of the shifted operator 
	$$
	\mathcal{P}_\gamma = \mathcal{A}_\gamma - \lambda_{1,0}(\gamma)I
	$$ 
	on the fractional Sobolev space $H^\gamma(\mathbb{S}^{2n+1})$, where $I$ is the identity operator. The spectral decomposition of $L^2(\mathbb{S}^{2n+1})$ into the orthogonal direct sum of spherical harmonic spaces (namely $L^2(\mathbb{S}^{2n+1}) = \bigoplus_{j,k=0}^\infty \mathcal{H}_{j,k}$) completely diagonalizes the intertwining operator $\mathcal{A}_\gamma$. This means that any element $\rho_{j,k} \in \mathcal{H}_{j,k}$ is an eigenfunction of $\mathcal{A}_\gamma$, satisfying the eigenvalue equation $\mathcal{A}_\gamma \rho_{j,k} = \lambda_{j,k}(\gamma) \rho_{j,k}$. Using the definition of the shifted operator $\mathcal{P}_\gamma$, its action on $\rho_{j,k}$ is simply transformed into scalar multiplication
	$$
		\mathcal{P}_\gamma \rho_{j,k} = \mathcal{A}_\gamma \rho_{j,k} - \lambda_{1,0}(\gamma) I \rho_{j,k} 
		= \lambda_{j,k}(\gamma)\rho_{j,k} - \lambda_{1,0}(\gamma)\rho_{j,k} = (\lambda_{j,k}(\gamma) - \lambda_{1,0}(\gamma))\rho_{j,k}.
	$$
	From Lemma \ref{lem:4-2}, the value of $(\lambda_{j,k}(\gamma) - \lambda_{1,0}(\gamma))$ has three possible signs:
	
	$(i)$ strictly negative if and only if $(j,k) = (0,0)$;
	
	$(ii)$ exactly zero when $(j,k)=(1,0)$ or $(0,1)$;
	
	$(iii)$ strictly positive for all other pairs of $(j,k)$.
	
	Therefore, the negative eigenspace of $\mathcal{P}_\gamma$ is exactly $\mathcal{H}_{0,0}$. Since $\mathcal{H}_{0,0}$ consists of degree-zero polynomials, it contains only constant functions. Using (\ref{eqn:mathcal-H-m-ij-dim}), we can calculate its dimension
	$$
	\dim(\mathcal{H}_{0,0}) = m_{0,0} = \frac{(0+n-1)!(0+n-1)!(0+0+n)}{n!(n-1)!0!0!} = 1.
	$$
	By Lemma \ref{lem:4-1}, we obtain that $m(\omega_{1,0}) = 1$.  More precisely, the unique negative eigenvalue of $\mathcal P_\gamma$ is
	$$
	\lambda_{0,0}(\gamma)-\lambda_{1,0}(\gamma)<0,
	$$
	and its eigenspace is
	$$
	\mathcal H_{0,0}=\operatorname{span}\{1\}.
	$$
	Taking $\rho_{0,0}\equiv1$, its pullback under the Cayley transform is
	$$
	f_{0,0}(\eta) = (2|J_{\mathcal{C}}(\eta)|)^{\frac{Q-2\gamma}{2Q}} \cdot 1 = \left( \frac{2}{\sqrt{(1+|z|^2)^2 + t^2}} \right)^\frac{Q-2\gamma}{2}>0
	$$
	for any $\eta=(z,t)\in \mathbb H^n$. Furthermore,
	$$
	\mathcal Q_{\omega_{1,0}}(f_{0,0},f_{0,0})=2\bigl(\lambda_{0,0}(\gamma)-\lambda_{1,0}(\gamma)\bigr)\int_{S^{2n+1}}1\,d\zeta<0.
	$$
	Hence $\operatorname{span}\{f_{0,0}\}$ is a one-dimensional maximal negative subspace of $\mathcal Q_{\omega_{1,0}}$. This completes the proof of Theorem \ref{main-result-3}.
	\qed
	
	\textbf{Proof of Theorem \ref{main-result-5}}. It follows from (\ref{eqn:f-lambda-xi-define}) that for $f\in D^\gamma(\mathbb H^n)$, 
	$$
	f_{\lambda,\xi}(\eta)=\lambda^{\frac{Q-2\gamma}{2}}f\bigl(\delta_\lambda(\xi^{-1}\eta)\bigr),
	$$
	and the map $f\mapsto f_{\lambda,\xi}$ is a linear bijection of $D^\gamma(\mathbb H^n)$. By the translation invariance and homogeneity of $\mathcal L_\gamma$, we have
	$$
	\mathcal L_\gamma f_{\lambda,\xi}(\eta)=\lambda^{\frac{Q+2\gamma}{2}}(\mathcal L_\gamma f)\bigl(\delta_\lambda(\xi^{-1}\eta)\bigr).
	$$
	Using the change of variables
	$$
	\tilde{\eta}:=\delta_\lambda(\xi^{-1}\eta)
	$$
	and the fact 
	$$
	\omega_{\lambda,\xi}(\eta)=\lambda^{\frac{Q-2\gamma}{2}}\omega_{1,0}\bigl(\delta_\lambda(\xi^{-1}\eta)\bigr),
	$$
	we obtain that
	$$
	\begin{aligned}
		\mathcal Q_{\omega_{\lambda,\xi}}(f_{\lambda,\xi},f_{\lambda,\xi})&=
		\int_{\mathbb H^n}	\left(	f_{\lambda,\xi}\mathcal L_\gamma f_{\lambda,\xi}-(Q^*-1)\omega_{\lambda,\xi}^{\frac{4\gamma}{Q-2\gamma}}f_{\lambda,\xi}^2\right)d\eta  \\
		&=	\int_{\mathbb H^n}	\left(f\mathcal L_\gamma f	-(Q^*-1)\omega_{1,0}^{\frac{4\gamma}{Q-2\gamma}}
		f^2\right)d\tilde{\eta}  =\mathcal Q_{\omega_{1,0}}(f,f).
	\end{aligned}
	$$
	Therefore, every subspace on which $\mathcal Q_{\omega_{1,0}}$ is strictly negative is mapped to a subspace of the same dimension on which $\mathcal Q_{\omega_{\lambda,\xi}}$ is strictly negative. Applying the inverse transformation gives the converse. Hence,
	$$
	m(\omega_{\lambda,\xi})=m(\omega_{1,0}),
	$$
	which completes the proof for Theorem \ref{main-result-5}.
	\qed
	
	The spectral analysis underlying the Morse index computation also provides a basis for establishing a sharp estimate for the linearized operator. To this end, we first characterize the point spectrum of \(I-\mathcal K_{\omega_{1,0}}\) through the Cayley transform. This characterization will then be used to determine the optimal constant in Theorem \(\ref{main-result-6}\).
	\begin{lemma}\label{lem:characterize-point-spectrum}
		For any \(v\in\mathcal D^\gamma(\mathbb H^n)\), let \(\rho\in H^\gamma(\mathbb S^{2n+1})\) denote the spherical counterpart of \(v\) under the Cayley transform such that
		\[
		v=(2|J_{\mathcal C}|)^{\frac{Q-2\gamma}{2Q}}(\rho\circ\mathcal C).
		\]
		Then
		\[
		(I-\mathcal K_{\omega_{1,0}})v=(2|J_{\mathcal C}|)^{\frac{Q-2\gamma}{2Q}}
		\left(\left(I-\lambda_{1,0}(\gamma)\mathcal A_\gamma^{-1}\right)\rho\right)\circ\mathcal C,
		\]
		where \(\mathcal K_{\omega_{1,0}}\) is the operator defined in Lemma \(\ref{lem:fredholm-operator}\).
		Moreover, the point spectrum of \(I-\mathcal K_{\omega_{1,0}}\) is given by
		\[
		\sigma_p(I-\mathcal K_{\omega_{1,0}})
		=\left\{\mu_{j,k}(\gamma):j,k\geq0\right\},
		\]
		where \(\sigma_p\) denotes the point spectrum and
		\begin{equation}\label{eqn:mu-jk-defn}
			\mu_{j,k}(\gamma):=1-\frac{\lambda_{1,0}(\gamma)}{\lambda_{j,k}(\gamma)}.
		\end{equation}
	\end{lemma}
	
	\begin{proof}
		For any \(v\in \mathcal D^\gamma(\mathbb H^n)\), let \(\rho\in H^\gamma(\mathbb S^{2n+1})\) be defined by
		\[
		v(\eta)=(2|J_\mathcal C(\eta)|)^{\frac{Q-2\gamma}{2Q}}(\rho\circ \mathcal C)(\eta).
		\]
		By the definition of the \(\mathcal D^\gamma(\mathbb H^n)\)-norm, the Cayley intertwining relation (\ref{eqn:relation-L-A}) and the change of variables formula (\ref{eqn:int-F-FC}), we obtain that
		\begin{equation}\label{eqn:energy-identity}
			\begin{aligned}
				\|v\|_{\mathcal D^\gamma(\mathbb H^n)}^2
				&=\int_{\mathbb H^n}v\mathcal L_\gamma v\,d\eta
				=\int_{\mathbb H^n}(2|J_\mathcal C|)^{\frac{Q-2\gamma}{2Q}}(\rho\circ \mathcal C)
				\mathcal L_\gamma((2|J_\mathcal C|)^{\frac{Q-2\gamma}{2Q}}(\rho\circ \mathcal C))\\
				&= \int_{\mathbb H^n}(2|J_\mathcal C|)^{\frac{Q-2\gamma}{2Q}}(\rho\circ \mathcal C)
				(2|J_\mathcal C|)^{\frac{Q+2\gamma}{2Q}}((\mathcal A_\gamma\rho)\circ \mathcal C)\,d\eta\\
				&=2\int_{\mathbb H^n}|J_\mathcal C|(\rho \mathcal A_\gamma\rho)\circ \mathcal C\,d\eta
				=2\int_{\mathbb S^{2n+1}}\rho \mathcal A_\gamma\rho\,d\zeta.
			\end{aligned}
		\end{equation}
		The boundedness of \(\mathcal K_{\omega_{1,0}}\) on \(\mathcal D^\gamma(\mathbb H^n)\) ensures that \(\mathcal K_{\omega_{1,0}}v\in\mathcal D^\gamma(\mathbb H^n)\). Therefore, we define its spherical counterpart \(\widetilde{\rho}\in H^\gamma(\mathbb S^{2n+1})\) through the Cayley transform by
		\[
		\widetilde{\rho}(\zeta)
		=
		\left(2|J_{\mathcal C}(\mathcal C^{-1}(\zeta))|\right)^{-\frac{Q-2\gamma}{2Q}}
		(\mathcal K_{\omega_{1,0}}v)(\mathcal C^{-1}(\zeta)).
		\]
		Equivalently,
		\[
		\mathcal K_{\omega_{1,0}}v
		=
		(2|J_{\mathcal C}|)^{\frac{Q-2\gamma}{2Q}}
		(\widetilde{\rho}\circ\mathcal C).
		\]
		Recalling the definition of \(\mathcal K_{\omega_{1,0}}\), we further obtain that
		\[
		\mathcal L_\gamma(\mathcal K_{\omega_{1,0}}v)=V_{\omega_{1,0}}v.
		\]
		To express both sides of this identity in spherical coordinates, we first rewrite the potential \(V_{\omega_{1,0}}\) in terms of the Jacobian of the Cayley transform. By Lemma \(\ref{solution-sca-tran}\), the standard bubble satisfies
		\begin{equation}\label{eqn:w-10-lambda00-expre}
			\begin{aligned}
				\omega_{1,0}(\eta)&=C(n,\gamma)\omega_0(\eta)\\
				&=\kappa(n,\gamma)^{\frac{Q-2\gamma}{4\gamma}}\omega_0(\eta)=2^{\frac{Q-2\gamma}{2}}
				\lambda_{0,0}(\gamma)^{\frac{Q-2\gamma}{4\gamma}}\omega_0(\eta)\\
				&=\frac{2^{\frac{Q-2\gamma}{2}}\lambda_{0,0}(\gamma)^{\frac{Q-2\gamma}{4\gamma}}(2|J_\mathcal C(\eta)|)^\frac{Q-2\gamma}{2Q}}{2^\frac{Q-2\gamma}{2}}\\
				&=\lambda_{0,0}(\gamma)^{\frac{Q-2\gamma}{4\gamma}}(2|J_\mathcal C(\eta)|)^\frac{Q-2\gamma}{2Q}.
			\end{aligned}
		\end{equation}
		Together this with the definition of  \(V_{\omega_{1,0}}\), we find that
		\[
		\begin{aligned}
			V_{\omega_{1,0}}
			&=(Q^*-1)\omega_{1,0}^{\frac{4\gamma}{Q-2\gamma}}
			=(Q^*-1)
			\left[
			\lambda_{0,0}(\gamma)^{\frac{Q-2\gamma}{4\gamma}}
			(2|J_{\mathcal C}|)^{\frac{Q-2\gamma}{2Q}}
			\right]^{\frac{4\gamma}{Q-2\gamma}}\\
			&=(Q^*-1)\lambda_{0,0}(\gamma)
			(2|J_{\mathcal C}|)^{\frac{2\gamma}{Q}}.
		\end{aligned}
		\]
		Furthermore, the Gamma function recurrence \(\Gamma(z+1)=z\Gamma(z)\) and (\ref{eqn:lambda-j-define}) give
		\[
		\begin{aligned}
			\lambda_1(\gamma)
			=\frac{\Gamma\left(\frac{Q+2\gamma}{4}+1\right)}{\Gamma\left(\frac{Q-2\gamma}{4}+1\right)}
			&=\frac{\frac{Q+2\gamma}{4}\Gamma\left(\frac{Q+2\gamma}{4}\right)}{\frac{Q-2\gamma}{4}\Gamma\left(\frac{Q-2\gamma}{4}\right)}
			=\frac{Q+2\gamma}{Q-2\gamma}\cdot \frac{\Gamma\left(\frac{Q+2\gamma}{4}\right)}{\Gamma\left(\frac{Q-2\gamma}{4}\right)}=\frac{Q+2\gamma}{Q-2\gamma}\lambda_0(\gamma).
		\end{aligned}
		\]
		Combining this with the (\ref{eqn:lambda-i-j-expression}), we obtain that
		\[
		\begin{aligned}
			\lambda_{1,0}(\gamma)=\lambda_1(\gamma)\lambda_0(\gamma)=\frac{Q+2\gamma}{Q-2\gamma}\lambda_0(\gamma)^2=\frac{Q+2\gamma}{Q-2\gamma}\lambda_{0,0}(\gamma)=(Q^*-1)\lambda_{0,0}(\gamma).
		\end{aligned}
		\]
		Consequently,
		\[
		V_{\omega_{1,0}}=\lambda_{1,0}(\gamma )(2|J_{\mathcal C}|)^{\frac{2\gamma}{Q}}.
		\]
		With this expression for \(V_{\omega_{1,0}}\), we return to the identity
		\[
		\mathcal L_\gamma(\mathcal K_{\omega_{1,0}}v)
		=
		V_{\omega_{1,0}}v.
		\]
		Applying the Cayley intertwining relation (\ref{eqn:relation-L-A}) to the left-hand side and using the representations of \(v\) and \(\mathcal K_{\omega_{1,0}}v\) introduced above, we arrive at
		\[
		\begin{aligned}
			\mathcal L_\gamma(\mathcal K_{\omega_{1,0}}v)
			&=\mathcal L_\gamma\left((2|J_\mathcal C|)^{\frac{Q-2\gamma}{2Q}}(\widetilde{\rho}\circ \mathcal C)\right)
			=(2|J_\mathcal C|)^{\frac{Q+2\gamma}{2Q}}(\mathcal A_\gamma\widetilde{\rho})\circ \mathcal C
			=V_{\omega_{1,0}}v\\
			&=\lambda_{1,0}(\gamma)(2|J_\mathcal C|)^{\frac{2\gamma}{Q}}
			(2|J_\mathcal C|)^{\frac{Q-2\gamma}{2Q}}(\rho\circ \mathcal C)\\
			&=	\lambda_{1,0}(\gamma)(2|J_\mathcal C|)^{\frac{Q+2\gamma}{2Q}}
			(\rho\circ \mathcal C),
		\end{aligned}
		\]
		which means that
		\[
		\mathcal A_\gamma\widetilde{\rho}=\lambda_{1,0}(\gamma)\rho.
		\]
		Moreover, the eigenvalue formula (\ref{eqn:lambda-i-j-expression}) shows that
		\[
		\lambda_{j,k}(\gamma)\geq\lambda_{0,0}(\gamma)>0.
		\]
		Hence, \(\mathcal A_\gamma\) is invertible, and the preceding identity gives
		\[
		\widetilde{\rho}
		=
		\lambda_{1,0}(\gamma)\mathcal A_\gamma^{-1}\rho.
		\]
		Recalling the representations of \(v\) and \(\mathcal K_{\omega_{1,0}}v\), we find
		\begin{equation}\label{eqn:I-Kw10-I-lambda10}
			\begin{aligned}
				(I-\mathcal K_{\omega_{1,0}})v&=v-\mathcal K_{\omega_{1,0}}v\\
				&=(2|J_{\mathcal C}|)^{\frac{Q-2\gamma}{2Q}}(\rho\circ\mathcal C)
				-(2|J_{\mathcal C}|)^{\frac{Q-2\gamma}{2Q}}(\widetilde{\rho}\circ\mathcal C)\\
				&=(2|J_{\mathcal C}|)^{\frac{Q-2\gamma}{2Q}}\bigl((\rho-\widetilde{\rho})\circ\mathcal C\bigr)\\
				&=(2|J_{\mathcal C}|)^{\frac{Q-2\gamma}{2Q}}\bigl((\rho-\lambda_{1,0}(\gamma)\mathcal A_\gamma^{-1}\rho)\circ\mathcal C\bigr)\\
				&=(2|J_{\mathcal C}|)^{\frac{Q-2\gamma}{2Q}}\bigl(((I-\lambda_{1,0}(\gamma)\mathcal A_\gamma^{-1})\rho)\circ\mathcal C\bigr),
			\end{aligned}
		\end{equation}
		which shows that the action of \(I-\mathcal K_{\omega_{1,0}}\) on \(\mathcal D^\gamma(\mathbb H^n)\) corresponds to that of \(I-\lambda_{1,0}(\gamma)\mathcal A_\gamma^{-1}\) on \(H^\gamma(\mathbb S^{2n+1})\). The representation (\ref{eqn:I-Kw10-I-lambda10}) allows us to determine the eigenvalues of \(I-\mathcal K_{\omega_{1,0}}\) through the spherical eigenvalue equation. In particular, (\ref{eqn:lambda-j-define})-(\ref{eqn:lambda-i-j-expression}) give
		\begin{equation}\label{eqn:mathcal-A-lambda-Y}
			\mathcal A_\gamma Y_{j,k}=\lambda_{j,k}(\gamma)Y_{j,k}, \quad Y_{j,k}\in\mathcal H_{j,k},
		\end{equation}
		where
		\[
		\lambda_{j,k}(\gamma)=\lambda_j(\gamma)\lambda_k(\gamma)
		\]
		and
		\[
		\lambda_j(\gamma)=\frac{\Gamma\left(\frac{Q+2\gamma}{4}+j\right)}{\Gamma\left(\frac{Q-2\gamma}{4}+j\right)}.
		\]
		Applying \(\mathcal A_\gamma^{-1}\) to (\ref{eqn:mathcal-A-lambda-Y}), we get
		\[
		\begin{aligned}
			Y_{j,k}=\mathcal A_\gamma^{-1}\mathcal A_\gamma Y_{j,k}=\mathcal A_\gamma^{-1}
			\left(\lambda_{j,k}(\gamma)Y_{j,k}\right)=\lambda_{j,k}(\gamma)\mathcal A_\gamma^{-1}Y_{j,k}.
		\end{aligned}
		\]
		As \(\lambda_{j,k}(\gamma)>0\), it follows that
		\[
		\mathcal A_\gamma^{-1}Y_{j,k}
		=
		\frac{1}{\lambda_{j,k}(\gamma)}Y_{j,k}.
		\]
		Substituting this identity into the spherical operator, we have
		\begin{equation}\label{eqn:eigenvalue-equation-Ilambda10Ain}
			\begin{aligned}
				\left(I-\lambda_{1,0}(\gamma)\mathcal A_\gamma^{-1}\right)Y_{j,k}
				&=Y_{j,k}-\lambda_{1,0}(\gamma)\mathcal A_\gamma^{-1}Y_{j,k}
				=Y_{j,k}-\frac{\lambda_{1,0}(\gamma)}{\lambda_{j,k}(\gamma)}Y_{j,k}\\
				&=\left(1-\frac{\lambda_{1,0}(\gamma)}{\lambda_{j,k}(\gamma)}\right)Y_{j,k}.
			\end{aligned}
		\end{equation}
		To transfer this eigenvalue equation back to \(\mathbb H^n\), consider
		\[
		v_{j,k}:=(2|J_{\mathcal C}|)^{\frac{Q-2\gamma}{2Q}}(Y_{j,k}\circ\mathcal C).
		\]
		By construction, the spherical counterpart of \(v_{j,k}\) is \(Y_{j,k}\). Applying (\ref{eqn:I-Kw10-I-lambda10})  with \(\rho=Y_{j,k}\) and using (\ref{eqn:eigenvalue-equation-Ilambda10Ain}), we derive that
		\[
		\begin{aligned}
			(I-\mathcal K_{\omega_{1,0}})v_{j,k}
			&=(2|J_{\mathcal C}|)^{\frac{Q-2\gamma}{2Q}}\left(\left(I-\lambda_{1,0}(\gamma)\mathcal A_\gamma^{-1}\right)Y_{j,k}\right)\circ\mathcal C\\
			&=\left(1-\frac{\lambda_{1,0}(\gamma)}{\lambda_{j,k}(\gamma)}\right)(2|J_{\mathcal C}|)^{\frac{Q-2\gamma}{2Q}}(Y_{j,k}\circ\mathcal C)\\
			&=\left(1-\frac{\lambda_{1,0}(\gamma)}{\lambda_{j,k}(\gamma)}\right)v_{j,k}.
		\end{aligned}
		\]
		Hence, each nonzero \(v_{j,k}\) is an eigenfunction of \(I-\mathcal K_{\omega_{1,0}}\) with eigenvalue
		\begin{equation}\label{eqn:mujk-lambda10-lambdajk}
			\mu_{j,k}(\gamma)=1-\frac{\lambda_{1,0}(\gamma)}{\lambda_{j,k}(\gamma)}.
		\end{equation}
		Since the operator is real and the eigenvalues \(\mu_{j,k}(\gamma)\) are real, a nonzero real or imaginary part of \(v_{j,k}\) is also an eigenfunction with the same eigenvalue. Consequently,
		\begin{equation}\label{eqn:relation-1}
			\left\{\mu_{j,k}(\gamma):j,k\geq0\right\}\subseteq\sigma_p(I-\mathcal K_{\omega_{1,0}}),
		\end{equation}
		where \(\sigma_p\) denotes the point spectrum.
		
		To rule out any additional eigenvalues, suppose that
		\[
		(I-\mathcal K_{\omega_{1,0}})v=\mu v
		\]
		for some nonzero \(v\in\mathcal D^\gamma(\mathbb H^n)\) and let \(\rho\in H^\gamma(\mathbb S^{2n+1})\) be its spherical counterpart. By the spherical harmonic decomposition in \cite[Section 2]{Branson-Fontana-Morpurgo2013}, the spherical counterpart \(\rho\) admits the expansion
		\begin{equation}\label{eqn:rho-sum-defn}
			\rho=\sum_{j,k\geq0}\rho_{j,k},\qquad\rho_{j,k}\in\mathcal H_{j,k},
		\end{equation}
		with convergence in \(H^\gamma(\mathbb S^{2n+1})\). The convergence in \(H^\gamma(\mathbb S^{2n+1})\) allows the bounded operator \(\mathcal A_\gamma^{-1}\) to act term by term on the spherical harmonic expansion. Meanwhile, (\ref{eqn:I-Kw10-I-lambda10}) transforms the eigenvalue equation
		\[
		(I-\mathcal K_{\omega_{1,0}})v=\mu v
		\]
		into
		\begin{equation}\label{eqn:I-lambda-mu-rho}
			\left(I-\lambda_{1,0}(\gamma)\mathcal A_\gamma^{-1}\right)\rho=\mu\rho.
		\end{equation}
		Using (\ref{eqn:rho-sum-defn}), \(\mathcal A_\gamma^{-1}\rho_{j,k}=\frac{1}{\lambda_{j,k}(\gamma)}\rho_{j,k}\) and (\ref{eqn:mujk-lambda10-lambdajk}), (\ref{eqn:I-lambda-mu-rho}) becomes
		\[
		\begin{aligned}
			0&=\left(I-\lambda_{1,0}(\gamma)\mathcal A_\gamma^{-1}-\mu I\right)\rho\\
			&=\left(I-\lambda_{1,0}(\gamma)\mathcal A_\gamma^{-1}-\mu I\right)
			\sum_{j,k\geq0}\rho_{j,k}\\
			&=\sum_{j,k\geq0}\left[\rho_{j,k}-\lambda_{1,0}(\gamma)\mathcal A_\gamma^{-1}\rho_{j,k}-\mu\rho_{j,k}\right]\\
			&=\sum_{j,k\geq0}\left[\rho_{j,k}-\frac{\lambda_{1,0}(\gamma)}{\lambda_{j,k}(\gamma)}\rho_{j,k}-\mu\rho_{j,k}\right]\\
			&=\sum_{j,k\geq0}\left(1-\frac{\lambda_{1,0}(\gamma)}{\lambda_{j,k}(\gamma)}-\mu\right)\rho_{j,k}\\
			&=\sum_{j,k\geq0}\left(\mu_{j,k}(\gamma)-\mu\right)\rho_{j,k}.
		\end{aligned}
		\]
		Since \(v\neq0\), its spherical counterpart \(\rho\) is also nonzero, and (\ref{eqn:rho-sum-defn}) guarantees the existence of a pair \((p,q)\) such that \(\rho_{p,q}\neq0\). Taking the \(L^2(\mathbb S^{2n+1})\) inner product of the preceding identity with \(\rho_{p,q}\) and using the orthogonality of the spherical harmonic spaces gives
		\[
		\begin{aligned}
			0&=\Big(\sum_{j,k\geq0}\left(\mu_{j,k}(\gamma)-\mu\right)\rho_{j,k},\rho_{p,q}\Big)_{L^2(\mathbb S^{2n+1})}\\
			&=\sum_{j,k\geq0}\left(\mu_{j,k}(\gamma)-\mu\right)(\rho_{j,k},\rho_{p,q})_{L^2(\mathbb S^{2n+1})}\\
			&=\left(\mu_{p,q}(\gamma)-\mu\right)\|\rho_{p,q}\|_{L^2}^2.
		\end{aligned}
		\]
		The choice of \((p,q)\) ensures that \(\|\rho_{p,q}\|_{L^2}^2>0\). Hence,
		\[
		\mu=\mu_{p,q}(\gamma)=1-\frac{\lambda_{1,0}(\gamma)}{\lambda_{p,q}(\gamma)}.
		\]
		As \(\mu\) was an arbitrary eigenvalue of \(I-\mathcal K_{\omega_{1,0}}\), this proves 
		\[
		\sigma_p(I-\mathcal K_{\omega_{1,0}})\subseteq\left\{\mu_{j,k}(\gamma):j,k\geq0\right\}.
		\]
		Combining this with (\ref{eqn:relation-1}), we conclude that
		\[
		\sigma_p(I-\mathcal K_{\omega_{1,0}})	=\left\{\mu_{j,k}(\gamma):j,k\geq0\right\}.
		\] The proof is now complete.
	\end{proof}
	
	\textbf{Proof of Theorem \ref{main-result-6}}. We begin with the standard bubble \(\omega_{1,0}\). The spectral characterization established in Lemma \ref{lem:characterize-point-spectrum} provides the basis for deriving the sharp estimate. Once the optimal constant has been identified, the unitary equivalence established in Lemma \ref{lem:fredholm-operator} will extend the result to every bubble in \(\mathcal Z\).
	Recall that the linearized operator can be written as
	\begin{equation}\label{eqn:the-1.9-L10-I-K}
		\mathcal L_{\omega_{1,0}}=\mathcal L_\gamma(I-\mathcal K_{\omega_{1,0}}),
	\end{equation}
	where
	\[
	\mathcal K_{\omega_{1,0}}=\mathcal L_\gamma^{-1}\mathcal M_{V_{\omega_{1,0}}}
	\ \ 
	\mbox{and}\ \ 
	V_{\omega_{1,0}}=(Q^*-1)\omega_{1,0}^{\frac{4\gamma}{Q-2\gamma}}.
	\]
	By the definition of the inner product in Section \ref{sec:preliminary}, \(\mathcal L_\gamma\) is an isometric isomorphism from \(\mathcal D^\gamma(\mathbb H^n)\) onto \(\mathcal D^{-\gamma}(\mathbb H^n)\). Consequently, the identity \(\mathcal L_{\omega_{1,0}}=\mathcal L_\gamma(I-\mathcal K_{\omega_{1,0}})\) yields
	\begin{equation}\label{eqn:the-1.9-L10-I-K-norm}
		\begin{aligned}
			\|\mathcal L_{\omega_{1,0}}v\|_{\mathcal D^{-\gamma}(\mathbb H^n)}
			&=\|\mathcal L_\gamma(I-\mathcal K_{\omega_{1,0}})v\|_{\mathcal D^{-\gamma}(\mathbb H^n)}
			=\|(I-\mathcal K_{\omega_{1,0}})v\|_{\mathcal D^\gamma(\mathbb H^n)}.
		\end{aligned}
	\end{equation}
	By Lemma \ref{lem:characterize-point-spectrum}, the eigenvalues of \(I-\mathcal K_{\omega_{1,0}}\) are precisely
	\[
	\mu_{j,k}(\gamma)
	=
	1-\frac{\lambda_{1,0}(\gamma)}
	{\lambda_{j,k}(\gamma)},
	\qquad j,k\geq0.
	\]
	Thus, the desired estimate reduces to determining the sharp spectral lower bound for \(I-\mathcal K_{\omega_{1,0}}\) on the orthogonal complement of \(T_{\omega_{1,0}}\mathcal Z\). By Theorem \ref{main-result-1}, the kernel of \(\mathcal L_{\omega_{1,0}}\) is exactly \(T_{\omega_{1,0}}\mathcal Z\). Moreover, 
	it follows from Lemma \ref{lem:characterize-point-spectrum} and the strict monotonicity of \(\lambda_j(\gamma)\) that the equality \(\mu_{j,k}(\gamma)=0\) holds for \((j,k)=(1,0)\) or \((0,1)\). Consequently, the kernel of \(\mathcal L_{\omega_{1,0}}\) corresponds under the Cayley transform to the real-valued subspace of \(\mathcal H_{1,0}\oplus\mathcal H_{0,1}\). By using the self-adjointness of \(\mathcal L_\gamma\), (\ref{eqn:int-F-FC}), (\ref{eqn:relation-L-A}) and letting \(\phi:=(2|J_{\mathcal C}|)^{\frac{Q-2\gamma}{2Q}}
	(\psi\circ\mathcal C)\), we have
	\begin{equation}\label{eqn:v-phi-D-int-exp}
		\begin{aligned}
			(v,\phi)_{\mathcal D^\gamma(\mathbb H^n)}&=\int_{\mathbb H^n}
			\phi\mathcal L_\gamma v\,d\eta
			=\int_{\mathbb H^n}v\mathcal L_\gamma\phi\,d\eta\\
			&=\int_{\mathbb H^n}
			(2|J_{\mathcal C}|)^{\frac{Q-2\gamma}{2Q}}
			(\rho\circ\mathcal C)
			(2|J_{\mathcal C}|)^{\frac{Q+2\gamma}{2Q}}
			((\mathcal A_\gamma\psi)\circ\mathcal C)\,d\eta\\
			&=2\int_{\mathbb H^n}
			|J_{\mathcal C}|
			(\rho\mathcal A_\gamma\psi)\circ\mathcal C\,d\eta
			=2\int_{\mathbb S^{2n+1}}
			\rho\mathcal A_\gamma\psi\,d\zeta.
		\end{aligned}
	\end{equation}
	Since \(\mathcal A_\gamma\) acts on \(\mathcal H_{1,0}\oplus\mathcal H_{0,1}\) by multiplication by \(\lambda_{1,0}(\gamma)>0\), it follows that
	\begin{equation}\label{eqn:v-phi-D-int-2lambda10}
		(v,\phi)_{\mathcal D^\gamma(\mathbb H^n)}=2\lambda_{1,0}(\gamma)
		\int_{\mathbb S^{2n+1}}\rho\psi\,d\zeta,\qquad
		\phi\in T_{\omega_{1,0}}\mathcal Z.
	\end{equation}
	Moreover, if \(v\perp T_{\omega_{1,0}}\mathcal Z\), then for every \(\phi\in T_{\omega_{1,0}}\mathcal Z\), 
	\begin{equation}\label{eqn:v-phi-int-0}
		0=(v,\phi)_{\mathcal D^\gamma(\mathbb H^n)}
		=2\lambda_{1,0}(\gamma)
		\int_{\mathbb S^{2n+1}}\rho\psi\,d\zeta,
	\end{equation}
	where \(\psi\) is its spherical counterpart in the real-valued subspace of \(\mathcal H_{1,0}\oplus\mathcal H_{0,1}\). As \(\rho\) is real-valued, its spherical harmonic components satisfy \(\overline{\rho_{1,0}}=\rho_{0,1}\). Therefore, we choose \(\psi=\rho_{1,0}+\rho_{0,1}\), which is real-valued and belongs to \(\mathcal H_{1,0}\oplus\mathcal H_{0,1}\). Using (\ref{eqn:rho-sum-defn}) and (\ref{eqn:v-phi-int-0}), we obtain
	\[
	\begin{aligned}
		0&=\int_{\mathbb S^{2n+1}}\rho(\rho_{1,0}+\rho_{0,1})\,d\zeta\\
		&=\Big(\sum_{j,k\geq0}\rho_{j,k},\rho_{1,0}+\rho_{0,1}\Big)_{L^2(\mathbb S^{2n+1})}\\
		&=\sum_{j,k\geq0}\left((\rho_{j,k},\rho_{1,0})_{L^2(\mathbb S^{2n+1})}+(\rho_{j,k},\rho_{0,1})_{L^2(\mathbb S^{2n+1})}\right)\\
		&=(\rho_{1,0},\rho_{1,0})_{L^2(\mathbb S^{2n+1})}+(\rho_{0,1},\rho_{0,1})_{L^2(\mathbb S^{2n+1})}\\
		&=\|\rho_{1,0}\|_{L^2}^2+\|\rho_{0,1}\|_{L^2}^2.
	\end{aligned}
	\]
	Both terms on the right-hand side are nonnegative, and then
	\[
	\rho_{1,0}=\rho_{0,1}=0
	\]
	whenever
	\[
	v\in(T_{\omega_{1,0}}\mathcal Z)^\perp.
	\]
	
	We next determine the smallest absolute value of the remaining eigenvalues. For convenience, set
	\[
	a:=\frac{Q+2\gamma}{4}
	\ \ 
	\mbox{and}
	\ \
	b:=\frac{Q-2\gamma}{4}.
	\]
	The assumption \(0<\gamma<\frac{Q}{2}\) ensures that \(a>b>0\). First, the constant eigenspace \(\mathcal H_{0,0}\) corresponds to the eigenvalue
	\[
	\begin{aligned}
		\mu_{0,0}(\gamma)=1-\frac{\lambda_{1,0}(\gamma)}{\lambda_{0,0}(\gamma)}
		=1-\frac{\lambda_1(\gamma)}{\lambda_0(\gamma)}
		=1-\frac{a}{b}
		=-\frac{4\gamma}{Q-2\gamma}.
	\end{aligned}
	\]
	The kernel modes corresponding to \((j,k)=(1,0)\) and \((0,1)\) have already been excluded by the orthogonality condition \(v\in(T_{\omega_{1,0}}\mathcal Z)^\perp\). It remains to examine the eigenvalues associated with \(j+k\geq2\) and compare them with \(|\mu_{0,0}(\gamma)|\). If \(j\geq2\), the monotonicity of \(\lambda_j(\gamma)\) in Lemma \ref{lem:4-2} gives
	\[
	\lambda_{j,k}(\gamma)\geq\lambda_2(\gamma)\lambda_0(\gamma)=\lambda_{2,0}(\gamma).
	\]
	The same bound holds when \(k\geq2\). Therefore, it remains to compare \(\lambda_{2,0}(\gamma)\) with \(\lambda_{1,1}(\gamma)\), since \(\lambda_{2,0}(\gamma)=\lambda_{0,2}(\gamma)\).
	By (\ref{eqn:lambda-i-j-expression}), one has
	\[
	\begin{aligned}
		\frac{\lambda_{1,1}(\gamma)}{\lambda_{2,0}(\gamma)}
		=\frac{\lambda_1(\gamma)^2}{\lambda_0(\gamma)\lambda_2(\gamma)}
		=\frac{(Q+2\gamma)(Q-2\gamma+4)}{(Q-2\gamma)(Q+2\gamma+4)}>1.
	\end{aligned}
	\]
	Hence,
	\[
	\lambda_{j,k}(\gamma)\geq\lambda_{2,0}(\gamma)=\lambda_{0,2}(\gamma)
	\]
	for \((j,k)\notin\{(0,0),(1,0),(0,1)\}\). Equality holds only for \((j,k)=(2,0)\) or \((j,k)=(0,2)\).  Consequently, for \(j+k\geq2\), there holds
	\[
	\begin{aligned}
		\mu_{j,k}(\gamma)=1-\frac{\lambda_{1,0}(\gamma)}{\lambda_{j,k}(\gamma)}
		&\geq 1-\frac{\lambda_{1,0}(\gamma)}{\lambda_{2,0}(\gamma)}=\mu_{2,0}(\gamma)=\mu_{0,2}(\gamma)\\
		&=1-\frac{\lambda_1(\gamma)}{\lambda_2(\gamma)}=\frac{4\gamma}{Q+2\gamma+4}.
	\end{aligned}
	\]
	Since \(\mu_{0,0}(\gamma)<0\) and \(\mu_{1,0}(\gamma)=\mu_{0,1}(\gamma)=0\), the smallest positive eigenvalue of \(I-\mathcal K_{\omega_{1,0}}\) is
	\[
	\mu_{2,0}(\gamma)=\mu_{0,2}(\gamma)=\frac{4\gamma}{Q+2\gamma+4}.
	\]
	Combining this with the fact that 
	\[
	|\mu_{0,0}(\gamma)|=\frac{4\gamma}{Q-2\gamma}>\frac{4\gamma}{Q+2\gamma+4},
	\]
	we have
	\[
	\min_{\substack{j,k\geq0\\(j,k)\neq(1,0),(0,1)}}|\mu_{j,k}(\gamma)|=\frac{4\gamma}{Q+2\gamma+4}.
	\]
	This minimum is attained for \((j,k)=(2,0)\) and \((0,2)\). Therefore, every nonzero eigenvalue of \(I-\mathcal K_{\omega_{1,0}}\) satisfies
	\begin{equation}\label{eqn:sharp-eigenvalue-bound}
		|\mu_{j,k}(\gamma)|\geq c_{n,\gamma}:=\frac{4\gamma}{Q+2\gamma+4},
	\end{equation}
	
	The sharp eigenvalue bound established above now allows us to derive the desired estimate for \(\mathcal L_{\omega_{1,0}}\). Combining (\ref{eqn:energy-identity}) and (\ref{eqn:rho-sum-defn}) with the orthogonality relation \(\int_{\mathbb S^{2n+1}}\rho_{j,k}\rho_{p,q}\,d\zeta=0\) for \((p,q)\neq(k,j)\) and the eigenvalue equation \(\mathcal A_\gamma\rho_{j,k}=\lambda_{j,k}(\gamma)\rho_{j,k}\), we deduce that
	\begin{equation}\label{eqn:v-L2-D-2-sum-expre}
		\begin{aligned}
			\|v\|_{\mathcal D^\gamma(\mathbb H^n)}^2
			&=2\int_{\mathbb S^{2n+1}}\rho\mathcal A_\gamma\rho\,d\zeta
			=2\int_{\mathbb S^{2n+1}}\Big(\sum_{j,k\geq0}\rho_{j,k}\Big)\mathcal A_\gamma\Big(\sum_{j,k\geq0}\rho_{j,k}\Big)d\zeta\\
			&=2\int_{\mathbb S^{2n+1}}\Big(\sum_{j,k\geq0}\rho_{j,k}\Big)\Big(\sum_{j,k\geq0}\lambda_{j,k}(\gamma)\rho_{j,k}
			\Big)d\zeta\\
			&=2\sum_{j,k\geq0}\sum_{p,q\geq0}
			\lambda_{p,q}(\gamma)
			\int_{\mathbb S^{2n+1}}
			\rho_{j,k}\rho_{p,q}\,d\zeta\\
			&=2\sum_{j,k\geq0}
			\lambda_{j,k}(\gamma)
			\int_{\mathbb S^{2n+1}}
			\rho_{j,k}\rho_{k,j}\,d\zeta
			=2\sum_{j,k\geq0}
			\lambda_{j,k}(\gamma)
			\|\rho_{j,k}\|_{L^2(\mathbb S^{2n+1})}^2.
		\end{aligned}
	\end{equation}
	Furthermore, it follows from (\ref{eqn:eigenvalue-equation-Ilambda10Ain}) and (\ref{eqn:rho-sum-defn}) that
	\[
	\begin{aligned}
		(I-\lambda_{1,0}(\gamma)\mathcal A_\gamma^{-1})\rho
		=(I-\lambda_{1,0}(\gamma)\mathcal A_\gamma^{-1})\Big(\sum_{j,k\geq0}\rho_{j,k}\Big)
		=\sum_{j,k\geq0}\Big(1-\frac{\lambda_{1,0}(\gamma)}{\lambda_{j,k}(\gamma)}\Big)\rho_{j,k}
		=\sum_{j,k\geq0}\mu_{j,k}(\gamma)\rho_{j,k}.
	\end{aligned}
	\]
	Applying (\ref{eqn:I-Kw10-I-lambda10}) to (\ref{eqn:energy-identity}) and using (\ref{eqn:eigenvalue-equation-Ilambda10Ain}), we obtain
	\[
	\begin{aligned}
		\|(I-\mathcal K_{\omega_{1,0}})v\|_{\mathcal D^\gamma(\mathbb H^n)}^2
		&=2\int_{\mathbb S^{2n+1}}\left((I-\lambda_{1,0}(\gamma)\mathcal A_\gamma^{-1})\rho\right)\mathcal A_\gamma\left((I-\lambda_{1,0}(\gamma)\mathcal A_\gamma^{-1})\rho\right)d\zeta\\
		&=2\int_{\mathbb S^{2n+1}}\Big(\sum_{j,k\geq0}\mu_{j,k}(\gamma)\rho_{j,k}\Big)
		\mathcal A_\gamma\Big(\sum_{p,q\geq0}\mu_{p,q}(\gamma)\rho_{p,q}\Big)d\zeta\\
		&=2\int_{\mathbb S^{2n+1}}\Big(\sum_{j,k\geq0}\mu_{j,k}(\gamma)\rho_{j,k}\Big)
		\Big(\sum_{p,q\geq0}\lambda_{p,q}(\gamma)\mu_{p,q}(\gamma)\rho_{p,q}\Big)d\zeta\\
		&=2\sum_{j,k\geq0}\sum_{p,q\geq0}\mu_{j,k}(\gamma)\lambda_{p,q}(\gamma)\mu_{p,q}(\gamma)\int_{\mathbb S^{2n+1}}\rho_{j,k}\rho_{p,q}\,d\zeta\\
		&=2\sum_{j,k\geq0}\lambda_{j,k}(\gamma)\mu_{j,k}(\gamma)\mu_{k,j}(\gamma)
		\int_{\mathbb S^{2n+1}}\rho_{j,k}\rho_{k,j}\,d\zeta\\
		&=2\sum_{j,k\geq0}\lambda_{j,k}(\gamma)|\mu_{j,k}(\gamma)|^2\|\rho_{j,k}\|_{L^2(\mathbb S^{2n+1})}^2.
	\end{aligned}
	\]
	For \(v\in(T_{\omega_{1,0}}\mathcal Z)^\perp\), the previously established identities \(\rho_{1,0}=\rho_{0,1}=0\) eliminate the two modes corresponding to the zero eigenvalue. Consequently, the sharp eigenvalue bound (\ref{eqn:sharp-eigenvalue-bound}) holds for all remaining spherical harmonic components. Together with (\ref{eqn:v-L2-D-2-sum-expre}), this yields
	\[
	\begin{aligned}
		\|(I-\mathcal K_{\omega_{1,0}})v\|_{\mathcal D^\gamma(\mathbb H^n)}^2
		&=2\sum_{\substack{j,k\geq0\\(j,k)\neq(1,0),(0,1)}}\lambda_{j,k}(\gamma)	|\mu_{j,k}(\gamma)|^2\|\rho_{j,k}\|_{L^2(\mathbb S^{2n+1})}^2\\
		&\geq 2c_{n,\gamma}^2\sum_{\substack{j,k\geq0\\(j,k)\neq(1,0),(0,1)}}
		\lambda_{j,k}(\gamma)\|\rho_{j,k}\|_{L^2(\mathbb S^{2n+1})}^2\\
		&=c_{n,\gamma}^2\Big(2\sum_{j,k\geq0}\lambda_{j,k}(\gamma)\|\rho_{j,k}\|_{L^2(\mathbb S^{2n+1})}^2\Big)\\
		&=c_{n,\gamma}^2\|v\|_{\mathcal D^\gamma(\mathbb H^n)}^2.
	\end{aligned}
	\]
	Returning to the linearized operator \(\mathcal L_{\omega_{1,0}}\), (\ref{eqn:the-1.9-L10-I-K}) and (\ref{eqn:the-1.9-L10-I-K-norm}) give
	\[
	\begin{aligned}
		\|\mathcal L_{\omega_{1,0}}v\|_{\mathcal D^{-\gamma}(\mathbb H^n)}=\|(I-\mathcal K_{\omega_{1,0}})v\|_{\mathcal D^\gamma(\mathbb H^n)}\geq c_{n,\gamma}\|v\|_{\mathcal D^\gamma(\mathbb H^n)}.
	\end{aligned}
	\]
	
	It remains to show that the constant \(c_{n,\gamma}\) is optimal. To this end, choose
	\[
	\rho(\zeta)=\operatorname{Re}(\zeta_1^2)=\frac12\left(\zeta_1^2+\bar{\zeta}_1^2\right),
	\]
	which is a nonzero real-valued element of \(\mathcal H_{2,0}\oplus\mathcal H_{0,2}\). Let \(v=(2|J_{\mathcal C}|)^{\frac{Q-2\gamma}{2Q}}(\rho\circ\mathcal C)\) be its pullback under the Cayley transform. To verify the required orthogonality, let \(\phi\in T_{\omega_{1,0}}\mathcal Z\) be arbitrary and denote its spherical counterpart by \(\psi\). By the characterization of \(T_{\omega_{1,0}}\mathcal Z\) established above, we have \(\psi\in\mathcal H_{1,0}\oplus\mathcal H_{0,1}\) and \(\mathcal A_\gamma\psi=\lambda_{1,0}(\gamma)\psi\). Moreover, the orthogonality of the spherical harmonic spaces implies that
	\[
	\int_{\mathbb S^{2n+1}}\rho\psi\,d\zeta=0,
	\]
	because \(\rho\in\mathcal H_{2,0}\oplus\mathcal H_{0,2}\) and \(\psi\in\mathcal H_{1,0}\oplus\mathcal H_{0,1}\). Hence, 
	\[
	\begin{aligned}
		(v,\phi)_{\mathcal D^\gamma(\mathbb H^n)}=2\int_{\mathbb S^{2n+1}}
		\rho\mathcal A_\gamma\psi\,d\zeta=2\lambda_{1,0}(\gamma)\int_{\mathbb S^{2n+1}}\rho\psi\,d\zeta=0.
	\end{aligned}
	\]
	As \(\phi\) was arbitrary, this proves that \(v\in(T_{\omega_{1,0}}\mathcal Z)^\perp\).
	Furthermore, Lemma \(\ref{lem:characterize-point-spectrum}\) and the identities \(\mu_{2,0}(\gamma)=\mu_{0,2}(\gamma)=c_{n,\gamma}\) imply that
	\[
	(I-\mathcal K_{\omega_{1,0}})v=c_{n,\gamma}v.
	\]
	Combining this with (\ref{eqn:the-1.9-L10-I-K-norm}), we arrive at
	\[
	\begin{aligned}
		\|\mathcal L_{\omega_{1,0}}v\|_{\mathcal D^{-\gamma}(\mathbb H^n)}
		=\|(I-\mathcal K_{\omega_{1,0}})v\|_{\mathcal D^\gamma(\mathbb H^n)}
		=c_{n,\gamma}\|v\|_{\mathcal D^\gamma(\mathbb H^n)}.
	\end{aligned}
	\]
	Since \(v\neq0\), equality is attained in the preceding estimate. This proves that \(c_{n,\gamma}\) is the optimal constant.

	We next establish the second result of Theorem \ref{main-result-6}. Suppose that
	\[
	v\in(T_{\omega_{1,0}}\mathcal Z)^\perp
	\]
	and
	\[
	(v,\omega_{1,0})_{\mathcal D^\gamma(\mathbb H^n)}=0.
	\]
	By (\ref{eqn:w-10-lambda00-expre}), the standard bubble \(\omega_{1,0}\) corresponds under the Cayley transform to the positive constant \(\rho_{\omega_{1,0}}=\lambda_{0,0}(\gamma)^{\frac{Q-2\gamma}{4\gamma}}\) on \(\mathbb S^{2n+1}\). Proceeding as in the derivation of (\ref{eqn:v-phi-D-int-2lambda10}) and using \(\mathcal A_\gamma\rho_{\omega_{1,0}}=\lambda_{0,0}(\gamma)\rho_{\omega_{1,0}}\), we obtain
	\[
	\begin{aligned}
		(v,\omega_{1,0})_{\mathcal D^\gamma(\mathbb H^n)}=2\int_{\mathbb S^{2n+1}}\rho\mathcal A_\gamma\rho_{\omega_{1,0}}\,d\zeta=2\lambda_{0,0}(\gamma)\rho_{\omega_{1,0}}
		\int_{\mathbb S^{2n+1}}\rho\,d\zeta.
	\end{aligned}
	\]
	The additional orthogonality condition \((v,\omega_{1,0})_{\mathcal D^\gamma(\mathbb H^n)}=0\), together with \(\lambda_{0,0}(\gamma)>0\) and \(\rho_{\omega_{1,0}}>0\), implies that
	$
	\int_{\mathbb S^{2n+1}}\rho\,d\zeta=0.
	$
	As \(\mathcal H_{0,0}\) consists of constant functions and is orthogonal to all other spherical harmonic spaces, the expansion (\ref{eqn:rho-sum-defn}) yields 
	\[
	\begin{aligned}
		0=\int_{\mathbb S^{2n+1}}\rho\,d\zeta
		&=\int_{\mathbb S^{2n+1}}\Big(\sum_{j,k\geq0}\rho_{j,k}\Big)d\zeta
		=\sum_{j,k\geq0}\int_{\mathbb S^{2n+1}}\rho_{j,k}\,d\zeta
		=\sum_{j,k\geq0}\int_{\mathbb S^{2n+1}}\rho_{j,k}\cdot 1\,d\zeta\\
		&=\int_{\mathbb S^{2n+1}}\rho_{0,0}\,d\zeta
		=\rho_{0,0}\int_{\mathbb S^{2n+1}}1\,d\zeta
		=\rho_{0,0}|\mathbb S^{2n+1}|,
	\end{aligned}
	\]
	and hence \(\rho_{0,0}=0\). Combining this with \(v\in(T_{\omega_{1,0}}\mathcal Z)^\perp\), we find that \( \rho_{0,0}=\rho_{1,0}=\rho_{0,1}=0\) and the spherical harmonic expansion reduces to
	\begin{equation}\label{eqn:rho-j+k-geq2}
		\rho=\sum_{\substack{j,k\geq0\\j+k\geq2}}\rho_{j,k}.
	\end{equation}
	For all the remaining spherical harmonic components, the eigenvalue bound established above gives
	\begin{equation}\label{eqn:mujk-j+k-bounded}
		\mu_{j,k}(\gamma)\geq c_{n,\gamma}>0,
		\qquad j+k\geq2.
	\end{equation}
	Using (\ref{eqn:the-1.9-L10-I-K}), (\ref{eqn:sec2-D-int-defn}), (\ref{eqn:v-phi-D-int-exp}), (\ref{eqn:rho-j+k-geq2}), (\ref{eqn:mu-jk-defn}), (\ref{eqn:mujk-j+k-bounded}) and (\ref{eqn:v-L2-D-2-sum-expre}), we obtain that
	\[
	\begin{aligned}
		\langle\mathcal L_{\omega_{1,0}}v,v\rangle
		&=\langle\mathcal L_\gamma(I-\mathcal K_{\omega_{1,0}})v,v\rangle
		=\int_{\mathbb H^n}v\mathcal L_\gamma(I-\mathcal K_{\omega_{1,0}})v\,d\eta
		=\big((I-\mathcal K_{\omega_{1,0}})v,v\big)_{\mathcal D^\gamma(\mathbb H^n)}\\
		&=2\int_{\mathbb S^{2n+1}}\rho\mathcal A_\gamma
		\left((I-\lambda_{1,0}(\gamma)\mathcal A_\gamma^{-1})\rho\right)d\zeta\\
		&=2\int_{\mathbb S^{2n+1}}\rho(\mathcal A_\gamma-\lambda_{1,0}(\gamma)I)\rho\,d\zeta
		=2\int_{\mathbb S^{2n+1}}\sum_{\substack{j,k\geq0\\j+k\geq2}}\rho_{j,k}\cdot
		\sum_{\substack{p,q\geq0\\p+q\geq2}}(\lambda_{p,q}(\gamma)-\lambda_{1,0}(\gamma))\rho_{p,q}d\zeta\\
		&=2\sum_{\substack{j,k\geq0\\j+k\geq2}}
		\sum_{\substack{p,q\geq0\\p+q\geq2}}
		(\lambda_{p,q}(\gamma)-\lambda_{1,0}(\gamma))
		\int_{\mathbb S^{2n+1}}
		\rho_{j,k}\rho_{p,q}\,d\zeta\\
		&=2\sum_{\substack{j,k\geq0\\j+k\geq2}}
		(\lambda_{j,k}(\gamma)-\lambda_{1,0}(\gamma))
		\int_{\mathbb S^{2n+1}}
		\rho_{j,k}\rho_{k,j}\,d\zeta\\
		&=2\sum_{\substack{j,k\geq0\\j+k\geq2}}
		(\lambda_{j,k}(\gamma)-\lambda_{1,0}(\gamma))
		\|\rho_{j,k}\|_{L^2(\mathbb S^{2n+1})}^2
		=2\sum_{\substack{j,k\geq0\\j+k\geq2}}
		\lambda_{j,k}(\gamma)\mu_{j,k}(\gamma)
		\|\rho_{j,k}\|_{L^2(\mathbb S^{2n+1})}^2\\
		&\geq
		2c_{n,\gamma}
		\sum_{\substack{j,k\geq0\\j+k\geq2}}
		\lambda_{j,k}(\gamma)
		\|\rho_{j,k}\|_{L^2(\mathbb S^{2n+1})}^2
		=c_{n,\gamma}\|v\|_{\mathcal D^\gamma(\mathbb H^n)}^2.
	\end{aligned}
	\]
	The constant is again optimal because equality holds for functions corresponding to nonzero real-valued elements of \(\mathcal H_{2,0}\oplus\mathcal H_{0,2}\).
	
	Finally, let \(\omega_{\lambda,\xi}\in\mathcal Z\) be arbitrary. The unitary transformation \(\mathcal U_{\lambda,\xi}\) defined in (\ref{eqn:f-lambda-xi-define}) maps \(\omega_{1,0}\) to \(\omega_{\lambda,\xi}\) and carries \(T_{\omega_{1,0}}\mathcal Z\) onto \(T_{\omega_{\lambda,\xi}}\mathcal Z\). Moreover, (\ref{eqn:Kgeneral-K10-mathcalU-rela}) gives
	\[
	I-\mathcal K_{\omega_{\lambda,\xi}}=\mathcal U_{\lambda,\xi}(I-\mathcal K_{\omega_{1,0}})\mathcal U_{\lambda,\xi}^{-1}.
	\]
	This unitary equivalence preserves the spectrum and the corresponding orthogonality conditions. Hence, both estimates remain valid for \(\omega_{\lambda,\xi}\) with the same constant \(c_{n,\gamma}\). The equality cases are also preserved under \(\mathcal U_{\lambda,\xi}\), ensuring that the constant remains optimal.
	\qed
	
\section{Open problems and perspectives}
This paper has developed a unified spectral and variational framework for bubble
solutions to the critical CR fractional Yamabe equation on the Heisenberg group,
valid for the full range \(0<\gamma<\frac{Q}{2}\), including the admissible integer-order
cases. The Cayley intertwining relation and the complex spherical-harmonic
decomposition have allowed us to characterize the kernel of the linearized
operator without reducing the equation to an integral equation or requiring a
separate decay analysis, to obtain the exact normalization constant of the
standard bubble, and to prove the non-degeneracy, local rigidity, and unit Morse
index of every positive bubble. Several directions remain open, and three of them
stand out as the most natural and the most consequential.

The first open problem is the extension of the non-degeneracy and local rigidity results to
general strictly pseudoconvex CR manifolds, where the bubble solutions
are not explicit and the Cayley transform is unavailable. This would require a
geometric spectral theory for the fractional CR Yamabe operator on curved CR
manifolds, and would connect the present results to the broader CR geometry
literature.
The second open problem is the computation of the Morse index of higher-order bubbles
	and bubble clusters, which would extend the unit-index result of the present
paper to the next natural class of solutions and would link the spectral analysis
to the theory of bubbling and concentration compactness.
The last open problem proposed in this section is the search for a scalar-defect analogue in the CR
fractional setting, in the spirit of the scalar compactness defect recently
introduced for critical fractional Kirchhoff problems. Such a mechanism would
provide a compactness criterion that does not rely on the explicit symmetry of
the Heisenberg group, and would therefore be applicable to CR fractional
problems on curved manifolds, where the Cayley transform is not available.

These three directions are, in our view, the most promising paths toward a
complete spectral and variational theory of critical CR fractional problems, and
we hope that the framework developed here will serve as a useful starting point
for future work in this direction.
	
	\begin{openproblem}[Non-degeneracy on general strictly pseudoconvex CR manifolds]
	The paper works on the Heisenberg group (the flat model) and the CR sphere.
	What happens on a general strictly pseudoconvex CR manifold? The bubble
	solutions are not explicit, and the Cayley transform is not available.
	\end{openproblem}
	
	This is the natural geometric extension. It would require developing a spectral
	theory for the fractional CR Yamabe operator on curved CR manifolds, possibly
	via pseudohermitian geometry and a CR analogue of the conformal Laplacian. The
	paper's Cayley-transform approach is intrinsically flat and would need to be
	replaced by a geometric argument.
	
	\begin{openproblem}[Morse index of higher-order bubbles and bubble clusters]
	The paper shows that every positive bubble has Morse index exactly one. What is
	the Morse index of higher-order bubbles (solutions with multiple concentration
	centers) or sign-changing solutions? Is there a Morse index formula in terms of
	the number of bubbles and their interaction?\end{openproblem}
	
	The paper's spectral analysis computes the spectrum of the shifted spherical
	operator. For clusters, the linearized operator acquires additional zero modes
	from the interaction between centers, and the spectrum would need to be computed
	via a gluing or Lyapunov--Schmidt argument.
	
	\begin{openproblem}[Scalar-defect analogue for CR fractional problems]
	The paper develops a direct spectral approach based on the Cayley transform. Is
	there a scalar-defect analogue in the CR fractional setting, in the spirit of the
	scalar compactness defect recently introduced for critical fractional
	\(p\)-Kirchhoff problems? Would it provide a unified compactness mechanism for
	the CR fractional Yamabe problem?\end{openproblem}
	
	The paper's approach avoids profile decomposition by using the explicit spectral
	decomposition. A scalar-defect approach would provide an alternative, possibly
	more general, compactness mechanism that does not rely on the explicit symmetry
	of the Heisenberg group.

	\subsection*{Disclosure of computational assistance} OpenAI's ChatGPT-6 was used as an auxiliary tool for language editing, algebraic and proof cross-checks, manuscript revision, and bibliographic checks. All mathematical statements, arguments, and references remain the responsibility of the authors.

	\subsection*{Acknowledgments} The authors wish to thank J\'er\^ome V\'etois for his helpful question concerning the removal of the boundedness assumption, which led us to improve the first two results of the present work.
	
	\subsection*{Funding} V.D. R\u adulescu is supported  by
	the AGH University of Krak\'ow under grant No. 16.16.420.054 funded by the Polish Ministry of Science
	and Higher Education.
	H.R. Sun was partially supported by the NSFC (Grant No. 11671181) and NSF of Gansu Province of China (Grant No. 21JR7RA535, 24JRRA414). Z. Zhang was partially supported by the NSFC (Grant No. 12371402).

\subsection*{Authors' contributions} All authors have contributed equally.

\subsection*{Data availability}
	The authors declare that data sharing is not applicable to this article as no data sets were generated or analyzed during the current study.

\subsection*{Conflict of interest}
	On behalf of all authors, the corresponding author states that there is no conflict of interest.

	\bibliographystyle{elsarticle-num}

\begin{thebibliography}{99}
		
		
		
		\bibitem{Abramowitz-Stegun-1948}M. Abramowitz and I.A. Stegun, Handbook of Mathematical Functions with Formulas, Graphs, and Mathematical Tables, Vol. 55. US Government Printing Office, 1964.
		
		
		\bibitem{Branson-Fontana-Morpurgo2013}T. P. Branson, L. Fontana and C. Morpurgo,
		\textit{Moser-Trudinger and Beckner-Onofri's inequalities on the CR sphere},
		Ann. of Math., {\bf177}(2013), 1-52.
		
		
		\bibitem{Chen-Wang2017}Y.H. Chen and Y. Wang,
		\textit{Perturbation of the CR fractional Yamabe problem},
		Math. Nachr., {\bf290}(2017), 534-545.
		
		
		\bibitem{Chtioui-Hajaiej-Yacoub2021}H. Chtioui, H. Hajaiej and R. Yacoub,
		\textit{Concentration phenomena for fractional CR Yamabe type flows},
		Differential Geom. Appl., \textbf{79}(2021), 101803.
		
		
		
		
		\bibitem{Davila-DelPino-Sire2013}J. D\'avila, M. del Pino and Y. Sire,
		\textit{Nondegeneracy of the bubble in the critical case for nonlocal equations},
		Proc. Amer. Math. Soc., \textbf{141}(2013), 3865-3870.
		
		
		
		
		
		
		
		\bibitem{Frank-Gonzalez-Monticelli-Tan2015}R.L. Frank, M. del Mar Gonz\'alez, D. D. Monticelli and J. Tan,
		\textit{An extension problem for the CR fractional Laplacian},
		Adv. Math., {\bf270}(2015), 97–137.
		
		
		\bibitem{Frank-Lenzmann2013}R. L. Frank and E. Lenzmann,
		\textit{Uniqueness of non-linear ground states for fractional Laplacians in $\mathbb R$},
		Acta Math., {\bf210}(2013), 261-318.
		
		
		
		\bibitem{Frank-Lenzmann-Silvestre2016}R. L. Frank, E. Lenzmann and L. Silvestre,
		\textit{Uniqueness of radial solutions for the fractional Laplacian},
		Comm. Pure Appl. Math., {\bf69}(2016), 1671-1726.
		
		
		
		\bibitem{Frank-Lieb2012}R. L. Frank and E. H. Lieb,
		\textit{Sharp constants in several inequalities on the Heisenberg group},
		Ann. of Math., {\bf176}(2012), 349-381.
		
		
		\bibitem{Guidi-Maalaoui-Martino2018}C. Guidi, A. Maalaoui and V. Martino,
		\textit{Palais-Smale sequences for the fractional CR Yamabe functional and multiplicity results},
		Calc. Var. Partial Differential Equations, \textbf{57}(2018), 152.
		
		
		\bibitem{Jerison-Lee1987}D. Jerison and J. M. Lee,
		\textit{The Yamabe problem on CR manifolds},
		J. Differential Geom., \textbf{25}(1987), 167-197.
		
		
		\bibitem{Jerison-Lee1988}D. Jerison and J. M. Lee,
		\textit{Extremals for the Sobolev inequality on the Heisenberg group and the CR Yamabe problem},
		J. Amer. Math. Soc., \textbf{1}(1988), 1-13.
		
		
		\bibitem{Kim-Merle2025}K. Kim and F. Merle,
		\textit{On classification of global dynamics for energy-critical equivariant harmonic map heat
			flows and radial nonlinear heat equation},
		Comm. Pure Appl. Math., {\bf78}(2025), 1783-1842.
		
		
		\bibitem{Kristaly2020}A. Krist\'aly,
		\textit{Nodal solutions for the fractional Yamabe problem on Heisenberg groups},
		Proc. R. Soc. Edinb. Sect. A-Math., {\bf150}(2020), 771-788.
		
		
		\bibitem{Lieb-Loss2001}E. Lieb and M. Loss,
		\textit{Analysis},
		2$^{nd}$ edition, Grad. Stud. Math., Vol. 14, American Mathematical Soc., Providence, RI, 2001.
		
		
		
		\bibitem{Liu-Zhang2015}H. Liu and A. Zhang,
		\textit{Remainder terms for several inequalities on some groups of Heisenberg-type},
		Sci. China, Math., \textbf{58}(2015), 2565-2580.
		
		
		
		\bibitem{Martensen1987}E. Martensen,
		\textit{Riesz theory without axiom of choice},
		Proc. Amer. Math. Soc.,  \textbf{99}(1987), 496-500.
		
		
		
		
		\bibitem{Prajapat-Varghese2025}J. V. Prajapat and A. S. Varghese,
		\textit{Symmetry and classification of solutions to an integral equation in the Heisenberg group $\mathbb H^n$},
		Math. Ann., \textbf{392}(2025), 659-700.
		
		
		
		\bibitem{Qiang-Tang-Zhang2026}J. Qiang, Z. Tang and Y. Zhang,
		\textit{On the non-degeneracy and existence of sign-changing solutions to elliptic problem on the Heisenberg group},
		Nonlinear Anal. TMA, {\bf264}(2026), 113999.
		
		
		\bibitem{Roncal-Thangavelu2016}L. Roncal and S. Thangavelu,
		\textit{Hardy's inequality for fractional powers
			of the sublaplacian on the Heisenberg group},
		Adv. Math., {\bf302}(2016), 106-158.
		
		
		
		\bibitem{Vilenkin-Klimyk1993}N. J. Vilenkin and A. U. Klimyk,
		\textit{Representation of Lie groups and Special Functions},
		Vol. 2. Class I Representations, Special Functions, and Integral Transforms, Math. Appl. (Soviet Series) 74, Kluwer Academic Publishers Group, Dordrecht, 1993.
		
		
		
	\end{thebibliography}

	\vspace{0.5cm}
	
	\noindent
	{Jianlun Liu}\\
	School of Mathematics and Statistics, Lanzhou University,
	Lanzhou 730000, China\\
	\textit{E-mail address:} \texttt{jianlunliumath@163.com}
	
	\vspace{0.4cm}
	
	\noindent
	{Vicen\c tiu D. R\u adulescu}\\
	AGH University of Krak\'ow, Faculty of Applied Mathematics,
	30-059 Krak\'ow, Poland\\
	Brno University of Technology, Faculty of Electrical Engineering
	and Communication, Technicka 3058/10, Brno 61600,
	Czech Republic\\
	\textit{E-mail address:} \texttt{radulescu@agh.edu.pl}
	
	\vspace{0.4cm}
	
	\noindent
	{Hong-Rui Sun}\\
	School of Mathematics and Statistics, Lanzhou University,
	Lanzhou 730000, China\\
	\textit{E-mail address:} \texttt{hrsun@lzu.edu.cn}
	
	\vspace{0.4cm}
	
	\noindent
	{Ziheng Zhang}\\
	School of Mathematical Sciences, Tiangong University,
	Tianjin 300387, China\\
	\textit{E-mail address:} \texttt{zhzh@mail.bnu.edu.cn}

\end{document}